\documentclass{tac}

\usepackage{amssymb,amsmath,stmaryrd,mathrsfs}
\usepackage[all,2cell]{xy}
\usepackage{enumitem}
\usepackage{color}
\definecolor{darkgreen}{rgb}{0,0.45,0} 
\usepackage[pagebackref,colorlinks,citecolor=darkgreen,linkcolor=darkgreen]{hyperref}
\usepackage{mathtools}          % for all sorts of things
\usepackage{url}                % for citations to web sites
\usepackage{xspace}             % put spaces after a \command in text

\title{Enriched indexed categories}
\author{Michael Shulman}
\date{\today}

\thanks{This material is based upon work supported by the National
  Science Foundation under a postdoctoral fellowship and agreement
  No. DMS-1128155.  Any opinions, findings, and conclusions or
  recommendations expressed in this material are those of the author
  and do not necessarily reflect the views of the National Science
  Foundation.}

\address{Department of Mathematics and Computer Science\\
  University of San Diego\\
  5998 Alcala Park, San Diego, CA 92110, USA}

\copyrightyear{2013}

\eaddress{shulman@sandiego.edu}

\makeatletter
\let\ea\expandafter

\def\mdef#1#2{\ea\ea\ea\gdef\ea\ea\noexpand#1\ea{\ea\ensuremath\ea{#2}\xspace}}
\def\alwaysmath#1{\ea\ea\ea\global\ea\ea\ea\let\ea\ea\csname your@#1\endcsname\csname #1\endcsname
  \ea\def\csname #1\endcsname{\ensuremath{\csname your@#1\endcsname}\xspace}}

\newcount\foreachcount

\def\foreachletter#1#2#3{\foreachcount=#1
  \ea\loop\ea\ea\ea#3\@alph\foreachcount
  \advance\foreachcount by 1
  \ifnum\foreachcount<#2\repeat}

\def\foreachLetter#1#2#3{\foreachcount=#1
  \ea\loop\ea\ea\ea#3\@Alph\foreachcount
  \advance\foreachcount by 1
  \ifnum\foreachcount<#2\repeat}

\def\definescr#1{\ea\gdef\csname s#1\endcsname{\ensuremath{\mathscr{#1}}\xspace}}
\foreachLetter{1}{27}{\definescr}
\def\definecal#1{\ea\gdef\csname c#1\endcsname{\ensuremath{\mathcal{#1}}\xspace}}
\foreachLetter{1}{27}{\definecal}
\def\definebold#1{\ea\gdef\csname b#1\endcsname{\ensuremath{\mathbf{#1}}\xspace}}
\foreachLetter{1}{27}{\definebold}
\def\definebb#1{\ea\gdef\csname l#1\endcsname{\ensuremath{\mathbb{#1}}\xspace}}
\foreachLetter{1}{27}{\definebb}
\def\definefrak#1{\ea\gdef\csname k#1\endcsname{\ensuremath{\mathfrak{#1}}\xspace}}
\foreachletter{1}{27}{\definefrak}
\foreachLetter{1}{27}{\definefrak}
\def\definehat#1{\ea\gdef\csname #1hat\endcsname{\ensuremath{\widehat{#1}}\xspace}}
\foreachLetter{1}{27}{\definehat}
\foreachletter{1}{27}{\definehat}
\def\defineul#1{\ea\gdef\csname u#1\endcsname{\ensuremath{\underline{#1}}\xspace}}
\foreachLetter{1}{27}{\defineul}
\foreachletter{1}{27}{\defineul}

\def\autofmt@n#1\autofmt@end{\mathrm{#1}}
\def\autofmt@b#1\autofmt@end{\mathbf{#1}}
\def\autofmt@l#1#2\autofmt@end{\mathbb{#1}\mathsf{#2}}
\def\autofmt@c#1#2\autofmt@end{\mathcal{#1}\mathit{#2}}
\def\autofmt@s#1#2\autofmt@end{\mathscr{#1}\!\mathit{#2}}

\def\auto@drop#1{}
\def\autodef#1{\ea\ea\ea\@autodef\ea\ea\ea#1\ea\auto@drop\string#1\autodef@end}
\def\@autodef#1#2#3\autodef@end{%
  \ea\def\ea#1\ea{\ea\ensuremath\ea{\csname autofmt@#2\endcsname#3\autofmt@end}\xspace}}
\def\autodefs@end{blarg!}
\def\autodefs#1{\@autodefs#1\autodefs@end}
\def\@autodefs#1{\ifx#1\autodefs@end%
  \def\autodefs@next{}%
  \else%
  \def\autodefs@next{\autodef#1\@autodefs}%
  \fi\autodefs@next}

\autodefs{\nCat\nMonCat\nSymMonCat\nSet\nTop\nAb\nCAT}
\autodefs{\nMon\nFam\nGrp}
\autodefs{\sAb\sAct\sFam\sSelf\sPsh\sConst\sAlg}
\autodefs{\cFIB\cCat\cCAT\cProf\cPROF\cMon\sTop}
\autodefs{\lMat\lMod\lProf\lPROF}
\autodefs{\cBimor\cMultimor}

\DeclareSymbolFont{bbold}{U}{bbold}{m}{n}
\DeclareSymbolFontAlphabet{\mathbbb}{bbold}
\newcommand{\lDelta}{\ensuremath{\mathbbb{\Delta}}\xspace}
\newcommand{\ltwo}{\ensuremath{\mathbbb{2}}\xspace}

\let\sm\wedge
\newcommand{\op}{^{\mathrm{op}}}
\let\adj\dashv
\SelectTips{cm}{}
\newdir{ >}{{}*!/-10pt/@{>}}    % extra spacing for tail arrows in XYpic

\newcommand{\pullback}[1][dr]{\save*!/#1-1.2pc/#1:(-1,1)@^{|-}\restore}
\let\iso\cong

\mdef\Id{\mathrm{Id}}
\mdef\id{\mathrm{id}}
\alwaysmath{ell}
\alwaysmath{infty}
\alwaysmath{odot}
\mdef\ten{\mathrel{\otimes}}

\DeclareMathOperator\lan{Lan}

\DeclareMathOperator\colim{colim}
\DeclareMathOperator\coeq{coeq}
\DeclareMathOperator\eq{eq}

\newcommand{\too}[1][]{\ensuremath{\overset{#1}{\longrightarrow}}}
\let\toot\rightleftarrows

\let\toto\rightrightarrows
\let\into\hookrightarrow

\let\maps\colon

\let\xto\xrightarrow

\def\slashedarrowfill@#1#2#3#4#5{%
  $\m@th\thickmuskip0mu\medmuskip\thickmuskip\thinmuskip\thickmuskip
   \relax#5#1\mkern-7mu%
   \cleaders\hbox{$#5\mkern-2mu#2\mkern-2mu$}\hfill
   \mathclap{#3}\mathclap{#2}%
   \cleaders\hbox{$#5\mkern-2mu#2\mkern-2mu$}\hfill
   \mkern-7mu#4$%
}
\def\rightslashedarrowfill@{%
  \slashedarrowfill@\relbar\relbar\mapstochar\rightarrow}
\newcommand\xslashedrightarrow[2][]{%
  \ext@arrow 0055{\rightslashedarrowfill@}{#1}{#2}}
\mdef\hto{\xslashedrightarrow{}}
\mdef\htoo{\xslashedrightarrow{\quad}}
\let\xhto\xslashedrightarrow

\def\toiso{\xto{\smash{\raisebox{-.5mm}{$\scriptstyle\sim$}}}}

\newif\ifhyperref
\@ifpackageloaded{hyperref}{\hyperreftrue}{\hyperreffalse}
  \let\your@state\state
  \def\state#1{\my@state#1}
  \def\my@state#1.{\gdef\currthmtype{#1}\your@state{#1.}}
  \let\your@staterm\staterm
  \def\staterm#1{\my@staterm#1}
  \def\my@staterm#1.{\gdef\currthmtype{#1}\your@staterm{#1.}}
  \let\defthm\newtheorem
  \def\switchtotheoremrm{\let\defthm\newtheoremrm}
  \def\currthmtype{}
  \ifhyperref
    \def\autoref#1{\ref*{label@name@#1}~\ref{#1}}
  \else
    \def\autoref#1{\ref{label@name@#1}~\ref{#1}}
  \fi
  \AtBeginDocument{%
    \let\old@label\label%
    \def\label#1{%
      {\let\your@currentlabel\@currentlabel%
        \edef\@currentlabel{\currthmtype}%
        \old@label{label@name@#1}}%
      \old@label{#1}}
  }

\newtheorem{thm}{Theorem}[section]

\defthm{cor}{Corollary}
\defthm{prop}{Proposition}
\defthm{lem}{Lemma}
\switchtotheoremrm
\defthm{defn}{Definition}
\switchtotheoremrm
\defthm{rmk}{Remark}
\defthm{eg}{Example}
\defthm{egs}{Examples}

\renewcommand{\theenumi}{(\roman{enumi})}

\setitemize[1]{leftmargin=2em}
\setenumerate[1]{leftmargin=*}

\let\c@equation\c@subsection
\numberwithin{equation}{section}

\@ifpackageloaded{mathtools}{\mathtoolsset{showonlyrefs,showmanualtags}}{}

\alwaysmath{alpha}
\alwaysmath{beta}
\alwaysmath{gamma}
\alwaysmath{Gamma}
\alwaysmath{delta}
\alwaysmath{Delta}
\alwaysmath{epsilon}
\mdef\ep{\varepsilon}
\alwaysmath{zeta}
\alwaysmath{eta}
\alwaysmath{theta}
\alwaysmath{Theta}
\alwaysmath{iota}
\alwaysmath{kappa}
\alwaysmath{lambda}
\alwaysmath{Lambda}
\alwaysmath{mu}
\alwaysmath{nu}
\alwaysmath{xi}
\alwaysmath{pi}
\alwaysmath{rho}
\alwaysmath{sigma}
\alwaysmath{Sigma}
\alwaysmath{tau}
\alwaysmath{upsilon}
\alwaysmath{Upsilon}
\alwaysmath{phi}
\alwaysmath{Pi}
\alwaysmath{Phi}
\mdef\ph{\varphi}
\alwaysmath{chi}
\alwaysmath{psi}
\alwaysmath{Psi}
\alwaysmath{omega}
\alwaysmath{Omega}
\let\al\alpha
\let\be\beta

\let\Si\Sigma
\let\om\omega
\let\ka\kappa

\mdef\topg{\sTop_\cG\xspace}
\let\V\sV
\let\W\sW
\let\SS\S
\let\S\bS
\let\I\lI
\mdef\uV{\underline{\V}}
\mdef\usV{\underline{\V}}
\mdef\usA{\underline{\sA}}
\mdef\usB{\underline{\sB}}
\mdef\usC{\underline{\sC}}
\mdef\usK{\underline{\sK}}
\mdef\usW{\underline{\sW}}
\mdef\ubC{\underline{\bC}}
\mdef\ubV{\underline{\bV}}
\mdef\ubA{\underline{\bA}}
\mdef\uWW{\underline{\underline{\sW}}}
\let\fam\sFam
\let\self\sSelf
\mdef\tV{{\textstyle\int}\V}
\mdef\tW{{\textstyle\int}\sW}
\def\tot{\textstyle\int}
\let\U\underline
\def\ord#1{#1_{\textrm{o}}}
\let\oast\varoast
\mdef\ellbar{{\overline{\ell}}}

\mdef\cPd{\cP^{\dagger}}
\mdef\cPk{\cP_{\ka}}
\mdef\cPkd{\cP^{\dagger}_{\ka}}
\mdef\Yd{Y^{\dagger}}
\mdef\uYd{\uY^{\dagger}}

\let\ch\kZ

\newcommand{\cten}[3]{\{#1,#2\}^{#3}}

\newcommand{\cat}[1]{\ensuremath{#1\text{-}\cCat}\xspace}
\newcommand{\icat}[1]{\ensuremath{#1\text{-}\mathcal{C}\text{\textsc{at}}}\xspace}
\newcommand{\CAT}[1]{\ensuremath{#1\text{-}\cCAT}\xspace}
\newcommand{\PROF}[1]{\ensuremath{#1\text{-}\cPROF}\xspace}
\newcommand{\iprof}[1]{\ensuremath{#1\text{-}\mathcal{P}\text{\textsc{rof}}}\xspace}
\newcommand{\prof}[1]{\ensuremath{#1\text{-}\cProf}\xspace}
\newcommand{\lprof}[1]{\ensuremath{#1\text{-}\lProf}\xspace}
\newcommand{\Lprof}[1]{\ensuremath{#1\text{-}\lPROF}\xspace}
\newcommand{\bimor}[1]{\ensuremath{#1\text-\cBimor}\xspace}
\newcommand{\mmor}[1]{\ensuremath{#1\text-\cMultimor}\xspace}
\newcommand{\VCat}{\cat{\sV}}
\newcommand{\VCAT}{\CAT{\sV}}
\newcommand{\VProf}{\prof{\sV}}
\newcommand{\VPROF}{\PROF{\sV}}
\newcommand{\VBimor}{\bimor{\V}}
\newcommand{\Vmmor}{\mmor{\V}}

\newcommand{\iVCAT}{\icat{\V}}
\newcommand{\iVPROF}{\iprof{\V}}
\newcommand{\VFib}{\ensuremath{\V\text-\cFIB}\xspace}
\def\mod#1#2{\prescript{}{#1}{#2}}

\newcommand{\e}{{\ensuremath{\epsilon}}}
\renewcommand{\ea}{{\ensuremath{\epsilon a}}}
\newcommand{\eb}{{\ensuremath{\epsilon b}}}
\newcommand{\ec}{{\ensuremath{\epsilon c}}}
\newcommand{\ex}{{\ensuremath{\epsilon x}}}
\newcommand{\ey}{{\ensuremath{\epsilon y}}}
\newcommand{\ez}{{\ensuremath{\epsilon z}}}

\newcommand{\eA}{{\ensuremath{\epsilon A}}}
\newcommand{\eB}{{\ensuremath{\epsilon B}}}

\makeatother

\newdir{ |}{{}*!/-10pt/@{|}}
\newdir{> }{!/10pt/@{>}*{}}

\begin{document}
\maketitle

\begin{abstract}
  We develop a theory of categories which are simultaneously (1)
  indexed over a base category \S with finite products, and (2)
  enriched over an \S-indexed monoidal category \V.  This includes
  classical enriched categories, indexed and fibered categories, and
  internal categories as special cases.  We then describe the
  appropriate notion of ``limit'' for such enriched indexed
  categories, and show that they admit ``free cocompletions''
  constructed as usual with a Yoneda embedding.
\end{abstract}

\tableofcontents

\section{Introduction}
\label{sec:introduction}

It is well-known that ordinary category theory admits several
important generalizations, such as the following.
\begin{itemize}
\item A category \emph{enriched} in a monoidal category \bV has a set
  of objects, but hom-objects belonging to \bV.
\item A category \emph{internal} to a category \S with pullbacks has
  both an \S-object of objects and an \S-object of morphisms.
\item A category \emph{indexed} or \emph{fibered} over a category \S
  has sets of objects and morphisms, each of which lives over a
  specified object or morphism in \S.
\end{itemize}
% In each case, we have versions of all the basic notions of category
% theory: limits and colimits, adjunctions, monads, presheaf
% categories, and so on.  (Of course, for a well-behaved theory, we
% often need to assume additional properties of \bV or \S.)  And each
% of these generalizations has many applications throughout
% mathematics.
Sometimes, however, we encounter category-like objects which appear
simultaneously enriched and internal, or enriched and indexed.  Here
are a few examples.
\begin{enumerate}
\item \emph{Parametrized homotopy theory} (as developed
  in~\cite{maysig:pht}) studies spaces and spectra parametrized over a
  given base space.  Each of these forms a category indexed over the
  category of base spaces.  In addition, however, parametrized spectra
  are enriched over parametrized spaces, in a sense which was
  recognized in~\cite{maysig:pht} but not given a general
  context.\label{item:egs1}
\item The free abelian group on a monoid is a ring; applying this
  functor homwise to an ordinary category, we obtain a category
  enriched over abelian groups.  Similarly, the suspension spectrum of
  a topological monoid is a ring spectrum, and so the fiberwise
  suspension spectrum of a topologically-internal category should be a
  category which is simultaneously internal to spaces and enriched
  over spectra.  Such categories played an important role
  in~\cite{kate:traces}.
\item \emph{Equivariant homotopy theory} studies spaces and spectra
  with actions of a topological group.  As the group in question
  varies, we find categories simultaneously indexed over groups and
  enriched over spaces.  More recently, \emph{global equivariant
    homotopy theory}~\cite{bohmann:globalspectra} studies equivariant
  spaces and spectra constructed in a coherent way across all groups
  of equivariance.  Such ``global spectra'' can be defined just like
  ordinary diagram spectra, if we work in the context of categories
  simultaneously indexed over groups and enriched over equivariant
  spaces.
\item In~\cite{bunge:bddcplt}, enriched indexed categories (which were
  discovered independently by Bunge) provide a general context to
  compare completions such as the Karoubi completion, stack
  completion, Grothendieck completion, and Cauchy completion.
\item The category of abelian sheaves is simultaneously indexed over
  base spaces and enriched over abelian groups.  Similarly, chain
  complexes of sheaves are indexed over spaces and enriched over chain
  complexes.
\item When doing mathematics relative to a base topos, we must replace
  small categories by internal ones and large categories by indexed
  ones.  Therefore, wherever enriched category theory is used in
  classical mathematics, in topos-relative mathematics we should
  expect to combine it with internalization and indexing.
\end{enumerate}

In this paper we show that enriched indexed categories
support a category theory as rich and powerful as all three classical
cases.  To a large extent, this is entirely straightforward.
Unsurprisingly, the resulting theory exhibits aspects that are
characteristic both of classical enriched category theory and internal
and indexed category theory.  Notable among the former is the need for
a notion of \emph{weighted} limit.  Notable among the latter is the
nontriviality of the passage from \emph{small} categories (e.g.\
internal ones) to \emph{large} ones (e.g.\ indexed ones).

\subsection{Some remarks about formal category theory.}
\label{sec:fct}

The theory of enriched indexed categories is clarified by using tools
from formal category theory, which are already known to encompass both
enriched and internal/indexed categories separately.  These tools
center around \emph{profunctors} between categories (in the classical
case, a profunctor from $A$ to $B$ is a functor $B\op\times
A\to\nSet$).  Every functor gives rise to an adjoint pair of
profunctors, and altogether categories, functors, and profunctors of
any fixed sort form a \emph{proarrow
  equipment}~\cite{wood:proarrows-i} (or ``framed
bicategory''~\cite{shulman:frbi}).

The central observations are the following:
\begin{enumerate}
\item In any equipment, there is a formal notion of a \emph{weighted
    limit} in an object (e.g.\ a category) weighted by a morphism
  (e.g.\ a profunctor);
  see~\cite{street-walters:yoneda,wood:proarrows-i}.  Starting from
  this we can develop large amounts of category theory purely
  formally.\label{item:e1}
\item For any well-behaved equipment \lW, there is an equipment of
  ``\lW-enriched categories'', functors, and profunctors; see
  e.g.~\cite{bcsw:variation-enr,street:enr-cohom,gs:freecocomp}.
  Moreover, we can remove the qualifier ``well-behaved'' by
  generalizing to \emph{virtual equipments} in the sense
  of~\cite{cs:multicats} (in which not all profunctors may be
  composable).  This was originally observed
  in~\cite{leinster:fc-multicategories,leinster:higher-opds,leinster:gen-enr-cats}.
  \label{item:e2}
\end{enumerate}
Observation~\ref{item:e1} means that in order to automatically obtain
a formally well-behaved theory of ``enriched indexed categories'',
essentially all we need is to define suitably related notions of
category, functor, and profunctor.  And observation~\ref{item:e2}
means that for this, we can start with a simpler (virtual) equipment
and apply the general enriched-categories construction.  Finally, the
relevant simpler equipment to begin with was already constructed
in~\cite{shulman:frbi}, starting from an indexed monoidal category \V
(the relevant ``base for enrichment'').

It would seem, then, that there is very little left to do; so why is
this paper so long?  There are several reasons.

Firstly, for the purposes of exposition, application, and wide
accessibility, it seems valuable to have explicit descriptions of what
the formal equipment-theoretic notions reduce to in our particular
case of interest, not requiring the reader to be familiar with the
literature of formal category theory.  For this reason, I will
minimize references to equipments, generally confining them to remarks
and to the proofs of lemmas (all of which could also easily be done
``by hand'').

Secondly, not all of the formal category theory existing in any
equipment has yet been generalized to the virtual case.  (The
generalization should be entirely straightforward, but for the most
part it has not yet been written out.)  However, the virtual case is
necessary in order to deal with \emph{large} categories, since even
when the enriching category is cocomplete, profunctors between large
categories may not be composable.\footnote{In some of the literature,
  such as~\cite{street-walters:yoneda}, this is avoided by invoking an
  embedding theorem to make the enriching category into an even larger
  one.  However, an analogous process for enriched indexed categories
  would be rather more complicated and obscure the important ideas.}
A reader who is so inclined can read parts of
sections~\ref{sec:large-cats}, \ref{sec:limits-colimits},
and~\ref{sec:psh} of this paper as contributions to this theory, since
wherever possible, we give proofs that apply in any virtual equipment.

% Thirdly, not all aspects of the theory of enriched indexed categories
% make sense in the general case of equipment-enriched categories.  The
% most notable example of this is \emph{iterated enrichment}\dots
% Anything which classically requires the enriching category to be
% \emph{symmetric} monoidal, such as opposites, tensor products, functor
% categories, and monoidal structures, usually has an easy extension to
% enriched indexed categories where the base \V is symmetric.  

Finally, enriched indexed categories share with classical indexed
categories the property of having multiple not-obviously-equivalent
definitions.  Given an indexed monoidal category \V, there is an
obvious notion of \emph{small \V-category}, which directly generalizes
internal categories and small enriched categories.  On the other hand,
there is also a fairly obvious notion of (large) \emph{indexed
  \V-category}, which generalizes locally small indexed categories
(incarnated as ``locally internal categories'' in the sense
of~\cite{penon:locintern}).  There are plenty of good examples of both
definitions, but it is not entirely trivial how to regard a small
\V-category as an indexed one!

The equipment-theoretic approach actually yields a \emph{third}
notion, which we will call simply a \emph{(large) \V-category}.  This
notion manifestly includes small \V-categories as a special case, but
its connection to indexed \V-categories is not entirely trivial.  In
the special case of ordinary (unenriched) indexed categories, this relationship
was established by~\cite{bcsw:variation-enr,cplt-locintern}.  Thus, we
spend some time producing the analogous correspondence between indexed
\V-categories and large \V-categories.  It turns out to behave even
better when we work with equipments, rather than merely bicategories
as~\cite{bcsw:variation-enr,cplt-locintern} did.

The nontriviality of this correspondence also means that it also takes
a little work to rephrase equipment-theoretic notions (such as
weighted limits) in the language of indexed \V-categories.
Pleasingly, the results are exactly what one might hope for.

\subsection{Historical remarks.}
\label{sec:historical-remarks}

Apparently, the idea of enriched indexed categories was first proposed
by Lawvere~\cite{lawvere:tdcsutdb}.  A formal definition was given by
Gouzou and Grunig in~\cite{gg:fib-rel}, corresponding to what I will
call an ``indexed \V-category''.  They did not apply general
equipment-theoretic methods (which did not exist at that time).
Perhaps because of a lack of applications, the definition did not
become well-known, and the theory was not extensively developed.

I discovered enriched indexed categories myself around 2007, with
examples such as~\cite{maysig:pht} and~\cite{kate:traces} in mind.
Michal Przybylek~\cite{przybylek:enriched-internal} independently
invented them at about the same time as well.  After a brief
discussion on the categories mailing list, Thomas Streicher very
kindly sent me a copy of the work of Gouzou and Grunig.  Seeing this,
and not having any real applications in mind yet, and feeling that it
was all a special case of equipment theory, I put the notion aside for
a while.

However, recently two new applications have appeared.  Firstly, the
notion of ``global equivariant spectrum'' developed
in~\cite{bohmann:globalspectra} requires categories that are both
indexed over groups and enriched in spaces with group actions, and was
made possible by an early draft of this paper.  Secondly, during
Octoberfest 2012 in Montreal, I found that Marta Bunge had also
independently arrived at the same notion of ``indexed \V-category'',
with the goal of comparing various idempotent completion
monads~\cite{bunge:bddcplt}.

This suggests that it is time to publish a careful development of the
theory, using modern technology, and with due credit given to everyone
who discovered it independently.

\subsection{Outline of the paper.}
\label{sec:outline}

We begin in \SS\ref{sec:indexed-moncats} by studying \emph{indexed
  monoidal categories}, which provide the ``base'' of enrichment and
indexing for our enriched indexed categories.  More specifically,
there is a category \S which provides the base of the indexing, and an
\S-indexed category \V with a monoidal structure which provides the
enrichment.  Much of this theory can be found in~\cite{shulman:frbi}
and~\cite{ps:indexed}, but we recall it all here for convenience.  We
also give a large number of examples.

In \SS\SS\ref{sec:small-cats}--\ref{sec:large-cats}, we define
respectively the three kinds of \V-category: small, indexed, and
large.  In each case we also define the relevant notions of
\V-functor, \V-natural transformation, and \V-profunctor, and give
several examples.  In \SS\ref{sec:small-cats} and
\SS\ref{sec:large-cats} we also study profunctors in some more detail,
making use of some equipment-theoretic notions.

Then in \SS\ref{sec:v-fibrations} we compare the notions of
\V-category.  Small \V-categories are manifestly a special case of
large ones.  As for the indexed ones, we identify a particular
subclass of large \V-categories, called \emph{\V-fibrations}, and a
subclass of functors between them, called \emph{indexed}, which form a
2-category that is 2-equivalent to the 2-category of indexed
\V-categories.  We regard this correspondence as closely analogous to
the classical equivalence between pseudofunctors (indexed categories)
and fibrations, although it is not strictly a generalization of it.

Moreover, it turns out that the 2-category of \V-fibrations and
indexed \V-functors is \emph{biequivalent} to the entire 2-category of
large \V-categories and all \V-functors.
% For this biequivalence, it
% is crucial that we define functors between large \V-categories using
% the equipment-theoretic viewpoint as in~\cite{gs:freecocomp}, rather
% than merely the bicategorical one.
The equivalence also carries over to profunctors, so we have a
complete equivalence of ``category theories''.

It seems that in applications, the \V-categories which act like the
``large categories'' in classical category theory are always
\V-fibrations (or, equivalently, indexed \V-categories), while those
that act like the ``small categories'' are not always so.  (Sometimes
they are small in the sense of \SS\ref{sec:small-cats}; other times
they are small only in a weaker, non-elementary sense.)  Thus, it is
useful to have the context of large \V-categories which includes both.
% and as mentioned above, this is only possible when we
% use equipments rather than bicategories.

In \SS\ref{sec:cocuf} we consider ``change of enrichment'' along a
morphism $\V\to\V'$.  The definitions are all straightforward and
mostly omitted; mainly we give a lot of examples to show the
generality of the concept.  A particularly important case is that of
the ``underlying indexed category'' of an enriched indexed category,
which generalizes the classical ``underlying ordinary category'' of an
enriched category.

In \SS\ref{sec:limits-colimits}--\ref{sec:indexed-limits} we study the
very important topic of limits and colimits.  In
\SS\ref{sec:limits-colimits} we work purely equipment-theoretically,
defining limits in terms of profunctors and proving their basic
properties abstractly.  Then in \SS\ref{sec:indexed-limits} we
specialize these notions to the case of \emph{indexed} \V-categories,
where they turn out to reduce exactly to a combination of well-known
indexed and enriched notions of limit.  In the case of ordinary (unenriched)
indexed categories, this perspective on limits was explored
in~\cite{cplt-locintern} and sequels such
as~\cite{chbase-locintern,desc-locintern,bw:ends}.  On the other hand,
the same combination of indexed and enriched notions of limit was
studied in~\cite{gg:fib-rel}, but without the equipment-theoretic
context for justification and formal properties.

With the basic theory of limits and colimits available, there are of
course many different directions in which to develop category theory.
We choose only two: presheaf categories in \SS\ref{sec:psh}, and
monoidal structures in \SS\ref{sec:monoidal}.

The goal of \SS\ref{sec:psh} is to prove that presheaf \V-categories
are free cocompletions.  The arguments are purely formal and
equipment-theoretic.  Rather than restrict ourselves to presheaves on
small categories, we consider more generally \emph{small presheaves}
in the sense of~\cite{dl:lim-smallfr}, which form free cocompletions
of not-necessarily-small categories.

Finally, in \SS\ref{sec:monoidal} we study monoidal \V-categories,
using for the first time in an essential way the symmetry of \V.  We
define two tensor products of \V-categories, one ``indexed'' and one
not, which extend the biequivalence of \SS\ref{sec:v-fibrations} to a
monoidal biequivalence.  Monoidal \V-categories are then pseudomonoids
with respect to either of these tensor products.  We also define
\emph{closed} monoidal \V-categories by way of profunctors,
essentially specializing the general definitions
of~\cite{ds:monbi-hopfagbd,dms:antipodes,street:frob-psmon} to
\V-categories.  As in the classical case, closed monoidal
\V-categories correspond closely to monoidal adjunctions involving \V.
We conclude with the Day convolution monoidal structure for
\V-presheaf categories, which, combined with the previous theory,
yields some of the most important examples, from~\cite{maysig:pht}
and~\cite{bohmann:globalspectra}.

\subsection{Acknowledgments.}
\label{sec:acknowledgements}

I would like to thank my thesis advisor, Peter May, for many things
too numerous to mention.  I would also like to thank Anna Marie
Bohmann and Marta Bunge, for providing the final impetus for
publication.  Finally, I would like to thank Hurricane Sandy, for
providing a week free from other commitments, in which I was able to
polish my notes into a readable paper; and the Institute for Advanced
Study, for providing an electricity generator during that week.

\section{Indexed monoidal categories}
\label{sec:indexed-moncats}

Let \S be a category with finite products.  We write $\Delta_X : X\to
X\times X$ for the diagonal of $X\in\S$, and for related maps such as
$X\times Y \to X\times X\times Y$.  Similarly, we write $\pi_X$ for
any projection map in which $X$ is projected away, such as $X\to 1$ or
$X\times Y\to Y$.  If this would be ambiguous, such as for the product
projections $X\times X\to X$, we use numerical subscripts which again
denote the copy being projected \emph{away}; thus $\pi_1:X\times X\to
X$ is the projection \emph{onto} the \emph{second} copy.

Our enriched indexed categories will be indexed over \S.  Their
enrichment, on the other hand, will not be over a monoidal category in
the classical sense, but over the following type of category.

\begin{defn}
  An \textbf{\S-indexed monoidal category} is a pseudofunctor
  $\V:\S\op \to\nMonCat$, where \nMonCat is the 2-category of
  monoidal categories, strong monoidal functors, and monoidal
  transformations.
\end{defn}

As usual, we write the image of $X\in \S$ as $\V^X$, and the image of
$f:X\to Y$ as $f^*:\V^Y \to \V^X$.  We write the tensor product and
unit of $\V^X$ as $\otimes_X$ and $\I_X$.  The monoidality of $f^*$
means we have isomorphisms such as
\begin{equation}
  f^*A \otimes_X f^*B \;\cong\; f^*(A\otimes_Y B)
  \qquad\text{and}\qquad \I_X \cong f^* \I_Y.
\end{equation}

Of course, by applying the ``Grothendieck construction'' we can
equally regard \V as a fibration $\tV \to \S$.  Moreover, the monoidal
structures on the fibers $\V^X$ can equivalently be described by
giving a monoidal structure on the category $\tV$ such that
\begin{enumerate}
\item the fibration $\tV \to \S$ is strict monoidal, and
\item the tensor product of \tV preserves cartesian
  arrows.\label{item:tensprescart}
\end{enumerate}
In~\cite{shulman:frbi} this is called a \textbf{monoidal fibration};
see there for a proof of the equivalence.  If we write $\otimes$ and
\I for the monoidal structure and unit of \tV, then the relationships
between these and the fiberwise monoidal structures are as follows.
For $A\in\V^X$ and $B\in\V^Y$, we have
\begin{equation}
  A \otimes B = \pi_Y^* A \otimes_{X\times Y} \pi_X^* B
  \quad\in\V^{X\times Y}
  \qquad\text{and}\qquad
  \I = \I_1
  \quad\in \V^{1},
\end{equation}
while for $A,B\in \V^X$ we have
\begin{equation}
  A\otimes_X B = \Delta_X^* (A\otimes B)
  \quad\in\V^X
  \qquad\text{and}\qquad
  \I_X = \pi_X^* \I
  \quad\in\V^X.
\end{equation}
We sometimes call $\otimes$ the \textbf{external} product, and
$\otimes_X$ the \textbf{fiberwise} or \textbf{internal} one.
Property~\ref{item:tensprescart} gives us isomorphisms such as
\begin{equation}
  f^* A \otimes g^* B \;\cong\; (f\times g)^*(A\otimes B).
\end{equation}

We say that \V is \textbf{symmetric} if the pseudofunctor $\V:\S\op
\to\nMonCat$ lifts to $\nSymMonCat$; this is equivalent to asking \tV
and the fibration $\tV\to\S$ to be symmetric monoidal.  We say that \V
is \textbf{cartesian} if each category $\V^X$ is cartesian monoidal,
or equivalently if \tV is cartesian monoidal.

The two most important examples, corresponding to classical enrichment
and classical internalization/indexing, are as follows.

\begin{eg}\label{eg:fam-mf}
  Let \bV be an ordinary monoidal category, let $\S=\nSet$, and let
  $\V^X = \bV^X$ be the category of $X$-indexed families of objects of
  \bV with the pointwise tensor product.  We call this the
  \textbf{naive indexing} of \bV and write it as $\fam(\bV)$.  Its
  total category $\tot\fam(\bV)$ is the category $\nFam(\bV)$ of all
  set-indexed families of objects of \bV, where a morphism from
  $(A_x)_{x\in X}$ to $(B_y)_{y\in Y}$ consists of a function $f:X\to
  Y$ and a family of morphisms $f_x:A_x \to B_{f(x)}$.  The external
  product is defined by
  \begin{equation}
    (A\otimes B)_{(x,y)} = A_x \otimes B_y.
  \end{equation}
\end{eg}

\begin{eg}\label{eg:self-mf}
  Let \S be a category with finite limits, and let $\V^X = \S/X$, with
  the cartesian product (which is pullback in \S).  This is a
  cartesian monoidal fibration called the \textbf{self-indexing} of
  \S; we write it as $\self(\S)$.  Its total category $\tot\self(\S)$
  is the category $\S^\ltwo$ of arrows in \S, its external product is
  just the cartesian product in \S.
\end{eg}

See~\cite{shulman:frbi} and~\cite{ps:indexed} for further study of
indexed monoidal categories; the latter includes an informal string
diagram calculus.

Now, as is the case with classical enriched category theory, we
frequently need completeness conditions on \V.  By a
\textbf{fiberwise} limit or colimit, we mean a limit or colimit in a
fiber category $\V^X$ which is preserved by all functors $f^*$.
If \ka is a regular cardinal, we say that \V is \textbf{fiberwise
  \ka-complete} if it has all fiberwise limits of cardinality $<\ka$,
and we say \V is \textbf{fiberwise complete} if it has all small
fiberwise limits.  Of course we have similar notions of fiberwise
cocompleteness.

The other important sort of (co)limit for indexed categories is the
following.

\begin{defn}
  \V has \textbf{\S-indexed coproducts} if
  \begin{enumerate}
  \item each functor $f^*:\V^Y\to \V^X$ has a left adjoint $f_!$, and
  \item for any pullback square
    \begin{equation}
      \vcenter{\xymatrix{
          \ar[r]^h\ar[d]_k &
          \ar[d]^f\\
          \ar[r]_g &
        }}
    \end{equation}
    in \S, the induced Beck-Chevalley transformation $k_! h^* \to g^*
    f_!$ is an isomorphism.
  \end{enumerate}
  Dually, \V has \textbf{\S-indexed products} if each $f^*$ has a
  right adjoint $f_*$ satisfying an analogous condition.
\end{defn}

It is well-known that the adjoints $f_!$ exist if and only if the
fibration $\tV\to\S$ is also an opfibration.

\begin{eg}
  $\fam(\bV)$ has any fiberwise limits and colimits that \bV has, and
  has \nSet-indexed (co)products iff \bV has (co)products.
\end{eg}

\begin{eg}
  If \S has finite limits, then $\self(\S)$ has fiberwise finite
  limits.  It is fiberwise complete if \S is complete, and has any
  fiberwise colimits that \S has.  It always has \S-indexed
  coproducts, and has \S-indexed products if and only if \S is locally
  cartesian closed.
\end{eg}

We say that \V is \textbf{\ka-complete} if it is fiberwise
\ka-complete and has indexed products, and similarly it is
\textbf{\ka-cocomplete} if it is fiberwise \ka-cocomplete and has
indexed coproducts.

Now in the case when \V has indexed coproducts, there is a third
variant of the monoidal structure.  For $A\in\V^{X\times Y}$ and
$B\in\V^{Y\times Z}$, we define
\[A\ten_{[Y]} B = \pi_{Y!} \Delta_Y^* (A\ten B),
\]
which lies in $\V^{X\times Z}$.  We can also express this in terms of
the fiberwise product as:
\[A\ten_{[Y]} B \iso
\pi_{Y!} \big(\pi_Z^*A \ten_{X\times Y\times Z} \pi_X^* B\big).
\]
We call this the \textbf{canceling product} because the object $Y$ no
longer appears in the base of the result.  We have induced
isomorphisms such as
\begin{equation}
  (f\times 1)^* A \otimes_{[Y]} (1\times g)^* B \;\cong\;
  (f\times g)^*(A\otimes_{[Y]} B)\label{eq:canceling-compat-4}
\end{equation}
for $f:X'\to X$ and $g:Z'\to Z$.

\begin{eg}
  When $\V=\fam(\bV)$, the canceling product is
  \[(A\ten_{[Y]} B)_{(x,z)} = \coprod_{y\in Y} A_{(x,y)}\ten B_{(y,z)}.\]
\end{eg}

\begin{eg}
  When $\V=\self(\S)$, the canceling product is just a pullback in \S,
  but which then forgets the map to the object we pulled back over.
\end{eg}

Classically, in a monoidal category one often needs colimits to be
preserved by the tensor product in each variable.  For fiberwise
colimits, we can simply impose this condition fiberwise.  For indexed
coproducts, the relevant condition is the following.

\begin{defn}
  If \V is a monoidal fibration with indexed coproducts, we say that
  \textbf{\ten preserves indexed coproducts} if for any $f\maps X\to
  Y$ in \S and any $A\in\sC^Y$ and $B\in\sC^X$, the canonical map
  \begin{equation}\label{eq:projection}
    f_!(f^*A \otimes_X B) \to A \otimes_Y f_! B
  \end{equation}
  is an isomorphism, and symmetrically.
  (This condition is sometimes called ``Frobenius reciprocity'', or
  said to make the adjunction $f_! \dashv f^*$ into a ``Hopf
  adjunction''.)
\end{defn}

An exercise in pasting mates implies that this condition is equivalent
to the external product $\otimes$ preserving opcartesian arrows,
yielding isomorphisms such as
\begin{equation}
  (f\times g)_!(A\otimes B) \;\cong\; f_! A \otimes g_! B.
\end{equation}
When this condition holds, the canceling product has its own
derived commutativity isomorphisms, namely
\begin{align}
  (1\times f)^*A \ten_{[X]} B &\iso A\ten_{[Z]} f_!B
  \label{eq:canceling-compat-1}\\
  (f\times 1)_!A \ten_{[Y]} B &\iso f_!\big(A\ten_{[Y]}B\big)\\
  (1\times f)_!A \ten_{[Z]} B &\iso A\ten_{[X]} f^*B.
  \label{eq:canceling-compat-3}
\end{align}

% \subsection{Closed indexed monoidal categories.}
% \label{sec:closed}

Finally, we consider what it means for a monoidal fibration to be
closed.  For simplicity, we consider only the symmetric case (in the
non-symmetric case, we would have two homs of each type, a ``right
one'' and a ``left one'').

\begin{thm}\label{thm:three-homs}
  Let \V be an \S-indexed symmetric monoidal category with indexed
  products, and indexed coproducts preserved by $\otimes$.  Then the
  following are equivalent.
  \begin{enumerate}
  \item Each fiber $\V^X$ is closed symmetric monoidal and each
    restriction functor $f^*$ is closed symmetric monoidal.  This
    means that for $B,C\in\V^X$ we have $\uV^X(B,C)\in\V^X$ and
    isomorphisms
    \[\V^X(A\ten_X B,C) \iso \V^X\left(A,\uV^X(B,C)\right),\]
    natural in $A$, and moreover the canonical maps
    \begin{equation}\label{eq:intclosedmap}
      f^*\uV^Y(B,C) \too \uV^X(f^*B,f^* C)
    \end{equation}
    are isomorphisms.\label{item:closed-1}
  \item For any $B\in\V^Y$ and $C\in\V^{X\times Y}$, we have a
    $\uV^{[Y]}(B,C)\in \V^{X}$ and isomorphisms
    \[\V^{X\times Y}(A\ten B, C) \iso \V^{X}\left(A, \uV^{[Y]}(B,C)\right)\]
    natural in $A$, and moreover the resulting canonical maps
    \[f^*\uV^{[Y]}(B,C) \too \uV^{[Y]}(B,(f\times 1)^*C)\]
    are isomorphisms.\label{item:closed-2}
  \item For any $B\in\V^Y$ and $C\in\V^X$ we have a
    $\uV(B,C)\in\V^{X\times Y}$ and isomorphisms
    \begin{equation}\label{eq:closed-3-adjn}
      \V^X(A\ten_{[Y]} B, C) \iso \V^{X\times Y}(A, \uV(B,C))
    \end{equation}
    natural in $A$, and moreover the resulting canonical maps
    \[(f\times g)^*\uV(B,C) \to \uV(g^*B,f^*C)\]
    are isomorphisms.\label{item:closed-3}
  \end{enumerate}
  When these conditions hold, we say that \V\ is \textbf{closed}.
\end{thm}
\begin{proof}
  The relationships between the three kinds of hom-functors are
  \begin{equation}
    \begin{array}{rclcl}
      \uV^X(B,C) &\iso& \uV^{[X]}\big(B,\Delta_{X*} C\big)
      &\iso& \Delta_X^*\uV(B,C)\\
      \uV(B,C) &\iso& \uV^{Y\times X}\big(\pi_X^*B,\pi_Y^*C\big)
      &\iso& \uV^{[Y]}\big(B, \Delta_{Y*} \pi_Y^* C\big)\\
      \uV^{[Y]}(B,C) &\iso&
      \pi_{Y*}\uV^{X\times Y}\big(\pi_X^*B, C\big) &\iso&
      \pi_{Y*}\Delta_Y^*\uV(B,C).
    \end{array}
  \end{equation}
  Checking that the canonical maps coincide is an exercise in diagram
  chasing.  The equivalence of~\ref{item:closed-1}
  and~\ref{item:closed-2} can be found in~\cite{shulman:frbi}.
\end{proof}

When the conditions of \autoref{thm:three-homs} hold, we say that \V is
\textbf{closed}.  We call $\uV^X(-,-)$ the \textbf{fiberwise hom},
$\uV(-,-)$ the \textbf{external hom}, and $\uV^{[X]}(-,-)$ the
\textbf{canceling hom}.

\begin{eg}\label{eg:enriched-homs}
  If \bV\ is complete and cocomplete closed symmetric monoidal with
  internal-homs $\ubV(-,-)$, then $\fam(\bV)$ is closed; we have
  \begin{align}
    \uV^X(B,C) &= \Big(\ubV(B_x,C_x)\Big)_{x\in X}\\
    \uV^{[Y]}(B,C) &= \left(\prod_{y\in Y} \ubV(B_y,C_{x,y}) \right)_{x\in X}\\
    \uV(B,C) &= \Big(\ubV(B_x,C_y)\Big)_{x\in X,y\in Y}.
  \end{align}
\end{eg}

\begin{eg}
  $\self(\S)$ is closed just when \S is locally cartesian closed.
\end{eg}

\begin{rmk}\label{rmk:weaker-homs}
  The construction of the canceling hom from the fiberwise or external
  hom, and vice versa, do require indexed products as assumed.  This
  is natural when looking at \autoref{eg:enriched-homs}, in which the
  canceling hom involves a product whereas the other two do not.

  On the other hand, the \emph{definitions} of the fiberwise and
  external homs in terms of each other do not require any indexed
  products or coproducts, although the adjunction
  isomorphism~\eqref{eq:closed-3-adjn} does require indexed coproducts
  since it involves the canceling tensor product.  Thus, in the
  absence of any completeness or cocompleteness conditions on \V, we
  should define closedness by~\ref{item:closed-1}, and we are free to
  use the external hom defined by $\uV(B,C) = \uV^{Y\times
    X}(\pi_X^*B, \pi_Y^*C)$, although not its universal
  property~\eqref{eq:closed-3-adjn}.  (In fact, the external hom does
  have a universal property even in the absence of indexed coproducts,
  but we defer mention of it until \SS\ref{sec:monoidal}, where it
  will seem more natural.)
\end{rmk}

Classically, the tensor product in a closed monoidal category
preserves colimits in each variable.  It is similarly immediate that
the tensor product in an indexed closed monoidal category preserves
\emph{fiberwise} colimits in each variable, while for indexed colimits
we have:

\begin{lem}
  If \V is closed and has indexed coproducts, then its indexed
  coproducts are preserved by $\otimes$.
\end{lem}
\begin{proof}
  The morphism~\eqref{eq:projection} is a mate
  of~\eqref{eq:intclosedmap}, such that each is an isomorphism if and
  only if the other is.
\end{proof}

We can also define combination fiberwise/external/canceling products
and homs, which satisfy a more symmetric-looking adjunction.  If
$A\in\V^{X\times Y\times Z}$, $B\in\V^{Y\times Z\times W}$, and
$C\in\V^{X\times Y\times W}$, then we define
\begin{alignat*}{2}
  A\ten_{Y,[Z]} B &\;=\; \pi_{Z!} \Delta_{Y\times Z}^* \big(A\ten B\big)
  &&\;\iso\; \pi_{Z!}(\pi_{W}^*A\ten_{X\times Y\times Z\times W}
  \pi_X^*B)\\
  \uV^{Y,[W]}(B,C) &\;=\; \pi_{W*}\Delta_{Y\times W}^*\uV(B,C)
  &&\;\iso\; \pi_{W*} \uV^{X\times Y\times Z\times W}(\pi_X^*B, \pi_Z^*C).
\end{alignat*}
We then have
\begin{equation}\label{eq:general-closed}
  \V^{X\times Y\times W}\big(A\ten_{Y,[Z]} B, C\big) \iso
  \V^{X\times Y\times Z}\big(A, \uV^{Y,[W]}(B,C)\big).
\end{equation}
All the other products, homs, and adjunctions can be seen as special
cases of these when $X$, $Y$, $Z$, and/or $W$ are taken to be the
terminal object $1$, and in such cases we omit them from the notation.

Furthermore, when \V\ is closed, the base change and tensor-hom
adjunctions also automatically become enriched, in a suitable sense.

\begin{prop}\label{thm:bc-enriched}
  For any $f\maps X\to Y$ we have natural isomorphisms
  \begin{align}
    \uV^Y(f_!B, C) &\;\iso\; f_*\uV^X(B,f^*C)
    \mathrlap{\qquad\text{and}}\label{eq:bc-enriched-left}\\
    \uV^Y(C,f_*B) &\;\iso\; f_*\uV^X(f^*C,B).\label{eq:bc-enriched-right}
  \end{align}
\end{prop}
\begin{proof}
  The isomorphism~\eqref{eq:bc-enriched-right}, for any colax/lax
  adjunction $f^*\adj f_*$ between closed monoidal categories, is
  actually equivalent to $f^*$ being strong monoidal (it is a mate of
  $f^*(A\ten B)\iso f^*A\ten f^*B$).  The
  isomorphism~\eqref{eq:bc-enriched-left} is perhaps less well-known,
  but a similar argument shows that given a chain of adjunctions
  $f_!\adj f^*\adj f_*$ between closed monoidal categories with $f^*$
  strong, the isomorphism~\eqref{eq:bc-enriched-left} is equivalent to
  $f^*$ being \emph{closed} monoidal.
\end{proof}

\begin{rmk}
  In fact,~\eqref{eq:bc-enriched-left} makes sense and is true even if
  \V lacks indexed products in general, as an assertion that
  $\uV^Y(f_!B, C)$ has the universal property that $f_*\uV^X(B,f^*C)$
  would have if it existed.  The same is true of the first isomorphism
  in the following proposition.
\end{rmk}

\begin{prop}\label{thm:bc-enriched-ext}
  For any $f\maps X\to Y$ we have natural isomorphisms
  \begin{align*}
    \uV(f_!B,C) &\iso (1\times f)_*\uV(B,C) \\
    \uV(B,f^*C) &\iso (f\times 1)^*\uV(B,C) \\
    \uV(f^*C,B) &\iso (1\times f)^*\uV(C,B) \\
    \uV(C,f_*B) &\iso (f\times 1)_*\uV(C,B).
  \end{align*}
\end{prop}
\begin{proof}
  These are mates of the compatibility
  relations~\eqref{eq:canceling-compat-4}
  and~\eqref{eq:canceling-compat-1}--\eqref{eq:canceling-compat-3} for
  the canceling product.
\end{proof}

\begin{prop}\label{thm:bc-enriched-canceling}
  For any $f\maps X\to Y$ we have isomorphisms
  \begin{align*}
    \uV^{[Y]}(f_!B,C) &\iso \uV^{[X]}(B,(1\times f)^*C)\\
    \uV^{[X]}(f^*B,C) &\iso \uV^{[Y]}(B,(1\times f)_*C)\\
    f^*\uV^{[Z]}(B,C) &\iso \uV^{[Z]}(B,(f\times 1)^*C)\\
    f_*\uV^{[Z]}(B,C) &\iso \uV^{[Z]}(B,(f\times 1)_*C)
  \end{align*}
\end{prop}
\begin{proof}
  These are the mates under adjunction of the compatibility relations
  such as $(1\times f)_!(A\ten B) \iso A\ten f_!B$.
\end{proof}

It is a standard result that in any closed symmetric monoidal
category, such as $\V^X$, the ordinary hom-tensor adjunction
isomorphism
\[\V^X(A\ten_X B, C)\iso \V^X\big(A, \uV^X(B,C)\big)\]
enriches to an isomorphism of internal-hom objects in $\V^X$:
\[\uV^X(A\ten_X B, C)\iso \uV^X\big(A, \uV^X(B,C)\big).\]
It follows that in a closed monoidal fibration we also have other
enriched hom-tensor adjunction isomorphisms, such as the `purely
external' isomorphism
\[\uV(A\ten B,C) \iso \uV(A,\uV(B,C)).\]

B\'enabou has used the word \emph{cosmos} for a complete and
cocomplete closed symmetric monoidal category (the ideal situation for
classical enriched category theory).  By analogy, we define:

\begin{defn}
  An \textbf{\S-indexed cosmos} is an \S-indexed closed symmetric
  monoidal category which is \om-complete and \om-cocomplete.
\end{defn}

We have chosen \om as the cardinality of completeness and
cocompleteness in this definition so as to make the notion an
elementary one.  That is, at least if we reformulate indexed
categories using fibrations, then indexed cosmoi are models of a
first-order theory, making them appropriate for foundational contexts
such as category theory over a base topos.  However, many indexed
cosmoi arising in other applications are also fiberwise complete and
cocomplete.

Any cosmos \bV in B\'enabou's sense gives rise to a \nSet-indexed
cosmos $\fam(\bV)$, although not every \nSet-indexed cosmos arises in
this way.  Moreover, we will see in \SS\ref{sec:large-cats} that any
fiberwise cocomplete indexed cosmos gives rise to a cosmos in the
sense of~\cite{street:cauchy-enr}.

% \subsection{Examples}
% \label{sec:examples}

We now collect a number of further examples.

\begin{eg}\label{eg:const-mf}
  For any \S, and any ordinary monoidal category \bV, the constant
  pseudofunctor $X \mapsto \bV$ is an \S-indexed monoidal category with
  indexed products and coproducts preserved by $\otimes$.  Limits and
  colimits in \bV\ give fiberwise limits and colimits, and it is
  closed if \bV is.  The fiberwise, external, and canceling products
  and homs are all identical.  We call this a \emph{constant} indexed
  monoidal category and denote it by $\sConst(\S,\bV)$.  Its total
  category $\tot\sConst(\S,\bV)$ is $\S\times\bV$.
\end{eg}

\begin{eg}\label{eg:const-pt}
  As a particular case of the previous example, we may take \S to be
  the terminal category $\star$, in which case
  $\tot\sConst(\star,\bV)\cong\bV$.
\end{eg}

\begin{eg}\label{eg:famv}
  Recall that $\nFam(\S)$ denotes the category of all set-indexed
  families of objects of \S.  For any \S-indexed category \V, there is
  a $\nFam(\S)$-indexed category $\fam(\V)$, whose total category is
  $\nFam(\tV)$ and whose fiber over $(X_i)_{i\in I} \in\nFam(\S)$ is
  $\prod_{i\in I} \V^{X_i}$.  It inherits a monoidal structure,
  closedness, and fiberwise limits and colimits from \V, while it has
  indexed (co)products if and only if \V has both indexed (co)products
  and small fiberwise ones.

  Noting that $\nFam(\star)\cong\nSet$, we see that by applying this
  construction to $\sConst(\star,\bV)$ we reproduce our original
  example $\fam(\bV)$ from \autoref{eg:fam-mf}.
\end{eg}

% \begin{eg}
%   For any \S-indexed category \V and any $X\in\S$, there is an
%   $(\S/X)$-indexed category $\V/X$ defined by $(\V/X)^{Y\xto{f} X} =
%   \V^Y$.  When \V is monoidal, its fiberwise products pass immediately
%   to $\V/X$, making it also monoidal.
% \end{eg}

\begin{eg}
  If \V is an \S-indexed monoidal category and $F:\S'\to\S$ is any
  functor, then there is an $\S'$-indexed monoidal category $F^*\V$
  defined by $(F^*\V)^X = \V^{F(X)}$: the fiberwise monoidal structure
  of \V passes immediately to $F^*\V$.
  (If $F$ does not preserve finite products, then the external product
  of $F^*\V$ may differ from that of \V.)
\end{eg}

\begin{eg}\label{eg:top-mf}
  The category of topological spaces has pullbacks, but is not locally
  cartesian closed, so its self-indexing does not have indexed
  products and is not closed.
  Its subcategories of ``compactly generated spaces'' and
  ``$k$-spaces'' are cartesian closed, but still not locally cartesian
  closed.
  However, the references given
  in~\cite[\SS1.3]{maysig:pht} show that if we take \S to be the
  category of \emph{compactly generated} spaces, and for $X\in\S$ we take
  $\sK^X$ to be the category of \emph{$k$-spaces} over $X$, then we do obtain
  an indexed cosmos \sK.
  Since not every $k$-space is compactly generated, this indexed
  cosmos \sK is larger than $\self(\S)$ (which is not an indexed
  cosmos).
\end{eg}

\begin{eg}\label{eg:ret-mf}
  If \S is a category with finite limits and finite colimits which are
  preserved by pullback, then there is an \S-indexed monoidal category
  $\self_*(\S)$ whose fiber $\self_*(\S)^X$ is the category of
  \emph{sectioned} objects over $X$.  For such $A$ and $B$, the
  \emph{fiberwise smash product} is the following pushout.
  \[\xymatrix{A\sqcup_X B \ar[r]\ar[d] & A\times_X B \ar[d]\\
    X \ar[r] & A \sm_X B.}\]
  This defines a monoidal structure with the unit object $X\to X\sqcup
  X \to X$, which has indexed coproducts preserved by $\sm$.  If \S\
  is locally cartesian closed, it is an indexed cosmos.

  More generally, a similar construction can be applied to any \V with
  fiberwise finite limits and colimits, with the fiberwise colimits
  preserved by $\otimes$.  For instance, starting from
  \autoref{eg:top-mf} we obtain an indexed cosmos $\sK_*$ of sectioned
  topological spaces.
\end{eg}

\begin{eg}\label{eg:ab-mf}
  Let \S be locally cartesian closed with countable colimits, and let
  $\sAb(\S)^X$ be the category of abelian group objects in $\S/X$.
  The countable colimits in \S\ enable us to define free abelian group
  objects.  Thus by~\cite[D5.3.2]{ptj:elephant}, $\sAb(\S)$ has
  indexed products and coproducts and fiberwise finite limits and
  colimits.  The countable colimits also enable us to define a tensor
  product, making $\sAb(\S)$ into an indexed cosmos.

  More generally, if \V is an \S-indexed cartesian cosmos with
  countable fiberwise colimits, we can define an \S-indexed cosmos
  $\sAb(\V)$ whose fiber over $X$ is the category of abelian groups in
  $\V^X$.  For example, from \autoref{eg:top-mf} we obtain an indexed
  cosmos of topological abelian groups.
\end{eg}

\begin{eg}\label{eg:modules-mf}
  Let \V be an \S-indexed cosmos and $R$ a commutative monoid object
  in $\V^1$, so that $\pi_X^*R$ is a commutative monoid in $\V^X$ for
  any $X$.  Let $(\mod R\V)^X$ be the category of objects in $\V^X$
  with a $\pi_X^*R$-action.  This is closed symmetric monoidal with
  finite limits and colimits; its tensor product is the coequalizer of
  the two actions
  \begin{equation}
    A \otimes_X (\pi_X^*R) \otimes_X B \;\toto\; A\otimes_X B.
  \end{equation}
  Since each $f^*$ preserves $\pi^*R$-actions and their limits, tensor
  product, and hom, $\mod R\V$ is an \S-indexed closed symmetric
  monoidal category with finite fiberwise limits and colimits.

  To show that it has indexed products and coproducts, we verify that
  $f_!$ and $f_*$ preserve $\pi^*R$-actions.  Let $f\maps X\to Y$;
  then if $M$ is a $\pi_X^*R$-object in $\V^X$, we have
  \begin{align*}
    \pi_Y^*R \ten_Y f_!M &\iso f_!(f^* \pi_Y^*R \ten_X M)\\
    &\iso f_!(\pi_X^*R \ten_X M)
  \end{align*}
  so we can use the action $\pi_X^*R\ten_X M\to M$ to induce an action
  of $\pi_Y^*R$ on $f_!M$.  On the other hand, since $f_*$ is lax
  monoidal (being the right adjoint of the strong monoidal $f^*$), it
  preserves monoids and their actions; thus $f_*\pi_X^*R$ is a monoid in
  $\V^Y$ and $f_*M$ is a $f_*\pi_X^*R$-object.  However,
  $f_*\pi_X^*R= f_*f^*\pi_Y^*R$ and the unit $\pi_Y^*R \to
  f_*f^*\pi_Y^*R$ is a monoid homomorphism, so this induces a
  $\pi_Y^*R$-action on $M$.  It is straightforward to check
  that these functors $f_!$ and $f_*$ are left and right adjoints to
  $f^*$ on $\pi^*R$-actions.  The Beck-Chevalley condition follows
  from that for \V, so $\mod R\V$ is an indexed cosmos.

  For example, if $(\S,R)$ is a ringed topos, such as the topos of
  sheaves on a scheme, then $\mod R{\sAb(\S)}$ is an indexed cosmos of
  sheaves of modules over the structure sheaf.
\end{eg}

\begin{eg}
  \autoref{eg:ab-mf} also works with the theory of abelian groups
  replaced by any other \emph{finitary commutative theory}
  (see~\cite[3.10]{borceaux:handbook-2}), such as the theory of
  $R$-modules for a fixed commutative ring $R$ (in sets), or the
  theory of pointed sets.  Applying the latter case to $\self(\S)$, we
  recover \autoref{eg:ret-mf}.

  In fact, we can also consider \emph{internal} finitary commutative
  theories in \S, such as modules for an internal commutative ring
  object in \S.  This gives another way to approach
  \autoref{eg:modules-mf}.
\end{eg}

\begin{eg}\label{eg:gactions-mf}
  On the other hand, if \V is an \S-indexed \emph{cartesian} indexed
  cosmos, and $G\in\V^1$ is any monoid (not necessarily commutative),
  then each category of $(\pi_X^* G)$-modules in $\V^X$ is also
  cartesian monoidal, with $(\pi_X^*G)$-action induced on the products
  by the diagonal of $G$.  We denote this indexed monoidal category by
  $G\V$; note that when $G$ is commutative, the underlying indexed
  category of $G\V$ is the same as that of $\mod G \V$.  For instance,
  $G$ could be a topological group in $\sK$, yielding the indexed
  monoidal category of equivariant unsectioned spaces
  from~\cite{maysig:pht}.  Applying \autoref{eg:ret-mf} we obtain the
  sectioned versions.
\end{eg}

\begin{eg}\label{eg:actions}
  For \S with finite products, let $\nGrp(\S)$ denote the category of
  group objects in \S.  For any such group object $G$, let
  $\sAct(\S)^G = G\S$ denote the category of objects with a
  $G$-action.  With fiberwise cartesian monoidal structures (and
  $G$-actions induced by the diagonal), this yields a
  $\nGrp(\S)$-indexed cartesian monoidal category $\sAct(\S)$.

  Since $G\S$ is monadic over \S with monad $(G\times -)$, $\sAct(\S)$
  has any fiberwise limits that \S has, and any fiberwise colimits
  that \S has and that are preserved by $\times$ in each variable.
  And if \S is cartesian closed, then $\sAct(\S)$ is closed, with the
  exponentials in $G\S$ being those of \S with a conjugation action.

\begin{figure}
  \centering
  $\vcenter{\xymatrix@C=4pc{G\times K \ar[r]^{f\times 1_K}\ar[d]_{1_G\times g}
    & H\times K \ar[d]^{1_H\times g}\\
    G\times L\ar[r]_{f\times 1_L} & H\times L}}$
  \hspace{2cm}
  $\vcenter{\xymatrix@C=3pc{G \ar[r]^{\Delta}\ar[d]_{\Delta}
    & G\times G \ar[d]^{1_G\times \Delta}\\
    G\times G\ar[r]_-{\Delta\times 1_G} & G\times G\times G}}$
  \\
  $\vcenter{\xymatrix@C=3pc{G \ar[r]^-{(1_G,f)}\ar[d]_f
    & G\times H \ar[d]^{f\times 1_H}\\
    H\ar[r]_-{\Delta_H} & H\times H}}$
  \caption{Homotopy pullback squares}
  \label{fig:pullbacks}
\end{figure}

  Now if \S has coequalizers preserved by $\times$ in each variable,
  then the restriction functors $f^*:G\to H$ have left adjoints.
  Namely, for a group homomorphism $f:G\to H$ and $X\in G\S$, we
  define $f_! X$ to be the coequalizer of the two maps
  \[ H\times G \times X \;\toto\; H\times X
  \]
  induced by the action of $G$ on $X$, and by $f$ followed by the
  multiplication of $H$.  Similarly, if \S is cartesian closed and has
  equalizers, then $f^*$ has a right adjoint, with $f_* X$ defined to
  be the equalizer of the analogous pair of maps
  \[ X^H \;\toto\; X^{G\times H}.
  \]
  Unfortunately, these adjoints do \emph{not} satisfy the
  Beck-Chevalley condition for all pullback squares in \S.
  % For instance, consider the pullback of the identity inclusion
  % $e:1\to G$ over itself, for any nontrivial group object $G$.
  However, we do have the Beck-Chevalley condition for three important
  classes of pullback squares, shown in Figure~\ref{fig:pullbacks}.
  (For example, this follows from~\cite[B2.5.11]{ptj:elephant}.)  Note
  that these squares are all pullbacks in any category with finite
  products, whether or not it has all pullbacks.

  Only once or twice in this paper will we use the Beck-Chevalley
  condition for fully general pullback squares; in most cases we will
  only need it for these particular ones, along with their transposes,
  and their cartesian products with fixed objects.
  In~\cite{ps:indexed}, we said that an indexed category had
  \textbf{indexed homotopy coproducts} if its restriction functors
  $f^*$ have left adjoints satisfying the Beck-Chevalley condition for
  these pullback squares.  But since $\sAct(\S)$ is the only example
  we will discuss in this paper\footnote{There are other examples
    discussed in~\cite{ps:indexed}, in which the fiber categories
    $\V^X$ are ``homotopy categories''.  However, these examples are
    usually not very interesting to enrich in, both because they tend
    to lack fiberwise limits and colimits, and because the resulting
    enriched categories would be only ``up to unspecified homotopy'',
    not up to \emph{coherent} homotopy.} which fails to have true
  indexed (co)products, we will not bother to add the adjective
  ``homotopy'' everywhere in this paper.  Instead we will merely
  remark on the one or two places where the fully general
  Beck-Chevalley condition is used.
\end{eg}

\begin{eg}
  If \V is an \S-indexed monoidal category and \bD\ a small category,
  define $(\V^\bD)^X=(\V^X)^\bD$, with the pointwise monoidal
  structure.  Then $\V^\bD$ is an \S-indexed monoidal category, which
  inherits closedness, fiberwise limits and colimits, and indexed
  products and coproducts from \V.

  For example, if $G$ is an ordinary monoid, we can consider it as a
  one-object category and obtain an indexed cosmos $\V^G$ of objects in
  \V\ with a $G$-action.
%   The case of enriched group actions will have to wait for
% Section~\ref{sec:monoidal}.
  On the other hand, taking $\bD=\lDelta\op$ to be the simplex
  category and \S a topos of sheaves, we obtain an indexed cosmos
  $\self(\S)^{\lDelta\op}$ of simplicial sheaves.
\end{eg}

\begin{eg}
  Now let \bD\ be a small \emph{monoidal} category and \V\ a fiberwise
  complete and cocomplete \S-indexed cosmos.  Then the category
  $(\V^X)^\bD$ is also closed monoidal under the \emph{Day convolution
    product}; see~\cite{day:closed}.  This gives a different monoidal
  structure to the indexed category $\V^\bD$.  Since the Day monoidal
  structure and internal-homs are constructed out of limits and
  colimits from those in $\V^X$, and $f^*$ preserves all of these, the
  transition functors $f^*$ are also strong and closed with respect to
  the convolution product.  Thus we have a second indexed cosmos
  structure on $\V^\bD$ in this case.

  For example, if we take \Sigma\ to be the category of finite sets
  and permutations and $\V$ to be the indexed cosmos $\sK_*$ of
  sectioned spaces, then $\sK_*^\Sigma$ is the indexed cosmos of
  symmetric sequences of sectioned spaces.  The spheres give a
  canonical commutative monoid $S$ in $(\sK_*^\Sigma)^1$ with respect
  to the convolution product, so by \autoref{eg:modules-mf} we have an
  indexed cosmos $\mod S{(\sK_*^\Sigma)}$ of \emph{parametrized
    (topological) symmetric spectra}.  We will consider parametrized
  orthogonal spectra in \SS\ref{sec:monoidal}.
  % The parametrized version of~\cite{mmss} says that
  % $S(\sK_*)^\Sigma$ should also be equivalent to
  % `$\sFun(\Sigma_S,\sR(\bT))$' for a suitably defined
  % \emph{topological} category $\Sigma_S$.  The parametrized
  % \emph{orthogonal} spectra used in~\cite{maysig:pht} also require a
  % version of $\sFun(\bD,\V)$ allowing \bD to be enriched.
\end{eg}

\begin{eg}\label{eg:factsys}
  Suppose \S is equipped with an orthogonal factorization system
  $(\cE,\cM)$ which is \emph{stable} in that \cE is preserved by
  pullback.  Then defining $\cM(\S)^X = \cM/X$ gives an \S-indexed
  monoidal category.  For instance, if \S is a regular category, then
  $(\cE,\cM)$ could be (regular epi, mono), in which case $\cM(\S)^X$
  is the poset of subobjects of $X$.
\end{eg}

\begin{eg}\label{eg:psh}
  Let \S have pullbacks and \bV be an ordinary monoidal category (such
  as \nSet), and define $\sPsh(\S,\bV)^X = \bV^{(\S/X)\op}$, with the
  pointwise monoidal structure.  The functor $f^*$ is defined by
  precomposition with $f:\S/X \to \S/Y$.  Since \S has pullbacks,
  $f:\S/X \to \S/Y$ has a right adjoint, hence so does $f^*$.  The
  Beck-Chevalley condition follows from that for pullbacks in \S, so
  the resulting \S-indexed monoidal category $\sPsh(\S,\bV)$ has
  indexed products.  It also inherits fiberwise limits and colimits
  from \bV.

  If \S is small and \bV is complete and cocomplete, then
  $\sPsh(\S,\bV)$ is closed and has indexed coproducts, hence is an
  \S-indexed cosmos.  If \S is not small, then $\sPsh(\S,\bV)$ may not
  have these properties, but if \S is \emph{locally} small and \bV is
  complete, cocomplete, and closed (i.e.\ a classical B\'enabou
  cosmos), then it has an important class of them.

  Namely, for any $Z\xto{g} X$, consider the \emph{representable}
  presheaf $F_g \in \bV^{(\S/X)\op}$, defined by $F_g(W\xto{h} X) =
  (\S/X)(h,g) \cdot \I$, a copower of copies of the unit object \I of
  \bV.  Now as in~\cite{dl:lim-smallfr}, we define an object of
  $\bV^{(\S/X)\op}$ to be \textbf{small} if it is a small
  (\bV-weighted) colimit of such representables.  Representable
  objects are closed under restriction, since $f^*(F_g) \cong
  F_{f^*g}$; hence so are small objects.  Similarly, representable
  objects are closed under tensor products, since $F_g \otimes_X F_h
  \cong F_{g\times_X h}$; hence so are small objects.

  Now $f^*:\bV^{(\S/Y)\op}\to \bV^{(\S/X)\op}$ has a partial left
  adjoint $f_!$ defined on small objects, which takes them to small
  objects: we define $f_!(F_g) = F_{fg}$ and extend cocontinuously.
  Similarly, all homs $\uV^{?}(A,B)$ exist when $A$ is small: we
  define $\uV^X(F_g,B)(W\xto{h} X) = B(g\times_X h)$ and extend
  cocontinuously, and construct the other homs from this as usual.
\end{eg}

\section{Small \V-categories}
\label{sec:small-cats}

Let \S be a category with finite products and \V an \S-indexed
monoidal category, with corresponding fibration $\tV\to\S$.  In this
section we describe a notion of ``small \V-category'' which directly
generalizes internal categories and small enriched categories.  (In
\SS\ref{sec:large-cats} we will see that there is also another, less
elementary, notion of ``smallness'' for \V-categories.)

We will use the following notation:
\begin{equation}
  \xymatrix{A \ar[r]^{\phi}\ar@{ |-> }[d] &
    B\ar@{ |-> }[d]\\
    X\ar[r]_-{f} & Y}
\end{equation}
to indicate that $\phi:A\to B$ is a morphism in $\tV$ lying over
$f:X\to Y$ in \S.  Of course, to give such a $\phi$ is equivalent to
giving a morphism $A\to f^* B$ in $\V^X$, but using morphisms in $\tV$
often makes commutative diagrams less busy (since there are fewer
$f^*$'s to notate).

\begin{defn}\label{def:small-vcat}
  A \textbf{small \V-category} $A$ consists of:
  \begin{enumerate}
  \item An object $\eA\in \S$.
  \item An object $\uA \in \V^{\eA\times\eA}$.
    % \[\xymatrix{ \uA \ar@{ |-> }[d] \\
    %   \e A\times \e A }\]
  \item A morphism in \tV:
    \begin{equation}\label{eq:identities}
      \xymatrix{\I_\eA \ar[r]^{\mathrm{ids}}\ar@{ |-> }[d] &
        \uA \ar@{ |-> }[d]\\
        \eA\ar[r]_-{\Delta} & \eA\times \eA}
    \end{equation}
  \item A morphism in \tV:
    \begin{equation}\label{eq:composition}
      \xymatrix{\uA \ten_\eA \uA
        \ar[r]^-{\mathrm{comp}}\ar@{ |-> }[d] &
        \uA \ar@{ |-> }[d]\\
        \eA\times \eA\times \eA \ar[r]_-{\pi_2} &
        \eA\times \eA}
    \end{equation}
    in which the fiberwise tensor product is over the middle copy of
    \eA.
  \item The following diagrams commute:
    \[\xymatrix{\uA \ten_\eA (\uA\ten_\eA \uA) \ar[r]^\iso
      \ar[d]_{\mathrm{comp}} &
      (\uA\ten_\eA \uA) \ten_\eA \uA \ar[r]^<>(.5){\mathrm{comp}} &
      \uA\ten_\eA \uA\ar[d]^{\mathrm{comp}}\\
      \uA\ten_\eA \uA\ar[rr]_{\mathrm{comp}} &&
      \uA}\]
    \[\xymatrix{\uA \ar[r]^<>(.5)\iso\ar@(dr,l)[drr]_{=} &
      \I_\eA \ten_{\eA} \uA \ar[r]^{\mathrm{ids}} &
      \uA \ten_\eA \uA \ar[d]^{\mathrm{comp}} &
      \uA \ten_\eA \I_\eA \ar[l]_{\mathrm{ids}} &
      \uA \ar[l]_<>(.5)\iso \ar@(dl,r)[dll]^{=}\\
      && \uA }\]
  \end{enumerate}
\end{defn}

\begin{eg}\label{eg:enriched}
  If \bV is an ordinary monoidal category, then a small
  $\fam(\bV)$-category is precisely a small \bV-enriched category.  It
  has a set $\eA$, an $(\eA\times\eA)$-indexed family
  \begin{equation}
    (\uA(a,b))_{(b,a)\in \eA\times\eA}\label{eq:enriched-matrix}
  \end{equation}
  of objects of \bV, an identities map with components $\I\to
  \uA(a,a)$, and a composition map with components
  $\uA(b,c)\ten \uA(a,b) \too \uA(a,c)$.
\end{eg}

\begin{rmk}
  In $\eA\times\eA$, we interpret the \emph{first} copy of \eA\ as the
  \emph{codomain} and the \emph{second} copy as the \emph{domain};
  hence the reversal of order in the subscript
  in~(\ref{eq:enriched-matrix}).  This is so that we end up composing
  morphisms in the usual order.
\end{rmk}

\begin{egs}\label{eg:internal-and-pointed}
  If \S has finite limits, then a small $\self(\bS)$-category is
  precisely a category internal to \S.  Similarly, if \cM is the class
  of monomorphisms in \S as in \autoref{eg:factsys}, then a small
  $\sM(\bS)$-category is precisely an internal poset in \bS.

  When \S satisfies the hypotheses of \autoref{eg:ret-mf}, then a
  small $\self_*(\bS)$-category may be called a \emph{pointed
    \bS-internal category}.  It consists of an \S-internal category
  $A_1 \toto A_0$ together with a morphism $A_0\times A_0 \to A_1$
  assigning to every pair of objects a ``zero morphism'' between them,
  which is preserved by composition on each side.
\end{egs}

\begin{eg}
  For \S and \bV as in \autoref{eg:const-mf}, a small
  $\sConst(\S,\bV)$-category consists of an object of \S together with
  a monoid in \bV.
\end{eg}

\begin{eg}
  A small $\sAb(\S)$-category consists of an internal category
  $A_1\toto A_0$ in \S, together with the structure of an abelian
  group on the object $A_1$ of $\S/(A_0 \times A_0)$ which is
  preserved by composition in each variable.
\end{eg}

\begin{egs}
  A small $\mod R\V$-category is a small \V-category with an action of
  $(\pi_{(\eA\times\eA)})^* R$ on $\uA \in \V^{\eA\times \eA}$, which
  is suitably preserved by the composition in each variable.
% \end{eg}
% \begin{eg}
  By contrast, a small $G\V$-category is a small \V-category with an
  action of $(\pi_{(\eA\times\eA)})^* G$ on $\uA$, which is preserved
  by the composition in both variables together.
\end{egs}

\begin{defn}\label{eg:discrete-vcats}
  If \V has indexed coproducts preserved by $\ten$, then for any
  object $X\in\S$, there is a small \V-category $\delta X$ with
  $\e(\delta X) = X$ and $\underline{\delta X} = (\Delta_X)_!\I_X$.
  We call it the \textbf{discrete \V-category on $X$}.
\end{defn}

% It is important to note that \emph{smallness} is built into the
% definition.  Unlike the case for ordinary categories, or categories
% enriched in an ordinary monoidal category, we cannot pass from small
% categories to large ones simply by replacing the set of objects with a
% proper class.  This is to be expected, since the `large' analogue of
% an \emph{internal} category is much less clear.  We will return to
% this in \SS\ref{sec:large}.

\begin{defn}\label{def:small-vfunctor}
  Let $A$ and $B$ be small \V-categories.  A \textbf{\V-functor}
  $f\maps A\to B$ consists of:
  \begin{enumerate}
  \item A morphism $\e f\maps \eA\to \eB$ in \bS.
  \item A morphism in \tV:
    \[\xymatrix@C=3pc{\uA \ar[r]^{f}\ar@{ |-> }[d] & \uB \ar@{ |-> }[d]\\
      \eA\times \eA\ar[r]_{\e f\times \e f} &  \eB\times \eB.}\]
  \item The following diagrams commute:
    \[\vcenter{\xymatrix@C=3pc{
        \I_{\eA} \ar[r]^-{\mathrm{ids}}\ar[d] & \uA \ar[d]^{f}\\
      \I_{\eB}\ar[r]_-{\mathrm{ids}} & \uB}} \qquad\text{and}\qquad
    \vcenter{\xymatrix@C=3pc{\uA\ten_{\eA} \uA \ar[r]^-{\mathrm{comp}}
      \ar[d]_{f\ten f} & \uA \ar[d]^{f}\\
      \uB\ten_\eB \uB \ar[r]_-{\mathrm{comp}} & \uB.}}\]
  \end{enumerate}
\end{defn}

\begin{egs}
  Evidently a $\fam(\bV)$-functor is a \bV-enriched functor, and a
  $\self(\S)$-functor is an \S-internal functor.  The other specific
  examples are similar.
\end{egs}

\begin{rmk}\label{rmk:disc-functors}
  If $\delta X$ is the discrete \V-category on $X\in\S$ as in
  \autoref{eg:discrete-vcats}, then a \V-functor $f:\delta X\to A$ is
  uniquely determined by its underlying morphism $\e f : X\to \eA$ in
  \S (the unit axiom forces its action on homs to be induced by
  $\mathrm{ids}:\I_{\eA} \to \uA$).

  If \V lacks indexed coproducts, in which case $\delta X$ may not
  exist as a \V-category, it is nevertheless often convenient to abuse
  language and allow the phrase ``\V-functor $\delta X \to A$'' to
  refer simply a morphism $X\to\eA$ in \S.
\end{rmk}

\begin{defn}\label{def:small-vnat}
  Let $f,g\maps A\to B$ be \V-functors.  A \textbf{\V-natural
    transformation} $\alpha\maps f\to g$ consists of:
  \begin{enumerate}
  \item A morphism
    \begin{equation}\label{eq:vnat-datum}
      \vcenter{\xymatrix@C=3pc{\I_{\eA} \ar[r]^{\alpha}\ar@{ |-> }[d]
        & \uB \ar@{ |-> }[d]\\
        \eA \ar[r]_-{(\e g,\e f)} & \eB\times \eB.}}
    \end{equation}
  \item The following diagram commutes.
    \begin{equation}
      \vcenter{\xymatrix@C=3pc{\uA \ar[r]^-\iso\ar[d]_\iso
          & \uA\ten_\eA \I_\eA
        \ar[r]^{f\ten \alpha} &
        \uB\ten_\eB \uB \ar[d]^{\mathrm{comp}}\\
        \I_\eA \ten_\eA \uA \ar[r]_{\alpha\ten g} & \uB\ten_\eB \uB
        \ar[r]_-{\mathrm{comp}} & \uB.}}\label{eq:vnat-axiom}
    \end{equation}
  \end{enumerate}
\end{defn}

\begin{egs}
  A $\self(\bS)$-natural transformation is an \bS-internal natural
  transformation, and a $\fam(\bV)$-natural transformation is a
  \bV-enriched one.  The other specific examples are similar.
\end{egs}

\begin{rmk}\label{rmk:disc-vnat}
  If $\delta X$ is a discrete \V-category as in
  \autoref{eg:discrete-vcats} and $f,g:\delta X\to A$ are determined
  by morphisms $\e f,\e g:X\to\eA$ as in \autoref{rmk:disc-functors},
  then a \V-natural transformation $\al:f\to g$ consists solely of a
  morphism~\eqref{eq:vnat-datum} (the axiom~\eqref{eq:vnat-axiom} is
  automatic).
  Thus, as in \autoref{rmk:disc-functors}, we may abuse language by
  referring to a morphism~\eqref{eq:vnat-datum} as a ``\V-natural
  transformation'' between ``\V-functors $\delta X\to A$'' even when
  \V lacks indexed coproducts.
\end{rmk}

\begin{thm}\label{thm:vcat}
  Small \V-categories, \V-functors, and \V-natural transformations
  form a 2-category $\VCat$.
\end{thm}
\begin{proof}
  The composition of \V-functors is obvious.
  The composition of \V-natural transformations $\al:f\to g$ and
  $\be:g\to h$ has components
  \begin{equation}
    \I_\eA \cong \I_\eA \otimes_{\eA} \I_\eA
    \xto{\al\otimes \be} \uB \otimes_{\eB} \uB
    \xto{\mathrm{comp}} \uB.
  \end{equation}
  The ``whiskering'' of a natural transformation on either side by a
  functor is likewise obvious; we leave the verification of the axioms
  to the reader.
\end{proof}

\begin{eg}\label{eg:uf-via-disc}
  When \V has indexed coproducts preserved by $\ten$, the ``discrete
  \V-category'' operation defines a 2-functor $\delta:\S\to\VCat$.
  Thus, any small \V-category $A$ induces an indexed category
  \[ \VCat(\delta-,A) : \S\op\to\nCat.
  \]
  The same construction works formally even if \V lacks indexed
  coproducts, using the conventions of Remarks~\ref{rmk:disc-functors}
  and~\ref{rmk:disc-vnat}.  In \SS\ref{sec:cocuf} we will identify
  this ``underlying indexed category'' with a special case of a
  general ``change of cosmos'' construction.
\end{eg}

% As is well-known, defining a 2-category is not sufficient to have a
% good ``category theory''; we need also a notion of ``profunctor''.

\begin{defn}\label{def:smallprof}
  Let $A$ and $B$ be small \V-categories.  A \textbf{\V-profunctor}
  $H:A\hto B$ consists of:
  \begin{enumerate}
  \item An object $\uH \in \V^{\eA\times \eB}$.
  \item Morphisms in \tV:
    \[\vcenter{\xymatrix@C=3pc{\uA \ten_\eA \uH \ar[r]^-{\mathrm{act}}\ar@{ |-> }[d] &
        \uH \ar@{ |-> }[d]\\
        \eA\times\eA\times \eB\ar[r]_-{1\times \pi\times 1} & \eA\times\eB}}
    \qquad\text{and}\quad
    \vcenter{\xymatrix@C=3pc{\uH \ten_\eB \uB \ar[r]^-{\mathrm{act}}\ar@{ |-> }[d] &
        \uH \ar@{ |-> }[d]\\
        \eA\times\eB\times \eB\ar[r]_-{1\times \pi\times 1} &
        \eA\times\eB}}
    \]
  \item The following diagrams commute.
    \[\vcenter{\xymatrix{\uA \ten_\eA \uA \ten_\eA \uH
        \ar[r]^-{1\ten\mathrm{act}}\ar[d]_{\mathrm{comp}\ten 1} &
        \uA\ten_\eA \uH\ar[d]^{\mathrm{act}}\\
        \uA\ten_\eA \uH\ar[r]_-{\mathrm{act}} & \uH}}
    \qquad
    \vcenter{\xymatrix{\uH \ten_\eB \uB \ten_\eB \uB
        \ar[r]^-{\mathrm{act}\otimes 1}\ar[d]_{1\otimes \mathrm{comp}} &
        \uH\ten_\eB \uB\ar[d]^{\mathrm{act}}\\
        \uH\ten_\eB \uB\ar[r]_-{\mathrm{act}} & \uH}}\]
    \[\vcenter{\xymatrix{\uH \ar[r]^-{\mathrm{ids}}\ar[dr]_{=} &
        \uA\ten_\eA\uH \ar[d]^{\mathrm{act}}\\
        & \uH}}
    \qquad
    \vcenter{\xymatrix{\uH \ar[r]^-{\mathrm{ids}}\ar[dr]_{=} &
        \uH\ten_\eB\uB \ar[d]^{\mathrm{act}}\\
        & \uH}}\qquad
    \vcenter{\xymatrix{\uA\ten_\eA \uH \ten_\eB \uB
      \ar[r]^-{\mathrm{act}\otimes 1}\ar[d]_{1\otimes\mathrm{act}} &
      \uH \ten_\eB \uB \ar[d]^{\mathrm{act}}\\
      \uA\ten_\eA \uH \ar[r]_-{\mathrm{act}} &
      \uH.}}
    \]
  \end{enumerate}
  We have an obvious notion of \textbf{morphism of profunctors},
  yielding a category which we denote $\VProf(A,B)$.
\end{defn}

\begin{egs}
  A $\fam(\bV)$-profunctor is equivalent to a \bV-functor $B\op\ten
  A\to \bV$, which is the classical notion of enriched profunctor; see
  for example~\cite{borceaux:handbook-3}.  Similarly,
  $\self(\S)$-profunctors give the usual notion of internal
  profunctor; see e.g.~\cite[\SS B2.7]{ptj:elephant}.
\end{egs}

\begin{eg}\label{eg:unit-prof}
  For any small \V-category $A$, there is a \textbf{unit}
  \V-profunctor $A\maps A\hto A$ defined by the hom-object \uA.  The
  action maps are simply composition in \uA.
\end{eg}

\begin{rmk}\label{rmk:disc-prof}
  If $A$ or $B$ is a discrete \V-category as in
  \autoref{eg:discrete-vcats}, then the corresponding action map is
  necessarily the unitality isomorphism.  In particular, a profunctor
  $\delta X \hto \delta Y$ is simply an object of $\V^{X\times Y}$,
  while a profunctor $\delta X\hto B$ or $A\hto \delta Y$ is simply an
  object of $\V^{X\times \eB}$ or $\V^{\eA\times Y}$ with a one-sided
  action of $B$ or $A$, as appropriate.
  Thus, as in Remarks~\ref{rmk:disc-functors} and~\ref{rmk:disc-vnat},
  we allow ourselves to abuse language by referring to such data as
  ``\V-profunctors'' even if \V lacks indexed coproducts.
\end{rmk}

\begin{eg}
  If $H:A\hto B$ is a \V-profunctor and $f:A'\to A$ and $g:B'\to B$
  are \V-functors, then there is a \V-profunctor $H(g,f):A'\hto B'$
  defined by
  \[\underline{H(g,f)} = (\e g\times \e f)^*\uH.
  \]
  The action maps are defined by composing the action maps of \uH with
  $f$.

  In particular, from any \V-functor $f:A\to B$ and the unit
  profunctor $B:B\hto B$, we obtain \textbf{representable profunctors}
  $B(1,f):A\hto B$ and $B(f,1):B\hto A$.
\end{eg}

Classically, profunctors can be composed with a `tensor product of
functors'.  In the enriched case, the composite of $H\maps B\op \ten
A\to\bV$ and $K\maps C\op\ten B\to \bV$ is the coend
\begin{multline*}
  (H\odot K)(c,a) = \int^B H(b,a) \ten K(c,b)\\
  = \coeq \left(\coprod_{b_1,b_2\in B} H(b_2,a) \ten B(b_1,b_2) \ten K(c,b_1)
                 \toto \coprod_{b\in B} H(b,a) \ten K(c,b) \right).
\end{multline*}

\begin{rmk}
  We write the composite of profunctors in ``diagrammatic'' order, so
  that $H:A\hto B$ and $K:B\hto C$ yield $H\odot K:A\hto C$.
\end{rmk}

By making these colimits indexed or fiberwise, as appropriate, we obtain:

\begin{lem}\label{thm:smallprofcomp}
  If \V has indexed coproducts preserved by $\ten$ and fiberwise
  coequalizers then any two \V-profunctors $H:A\hto B$ and $K:B\hto C$
  have a \textbf{composite} defined by
  \begin{equation}\label{eq:vdist-comp}
    \underline{H\odot K} =
    \coeq \Big(\uH \ten_{[\eB]} \uB \ten_{[\eB]} \uK
    \;\toto\; \uH \ten_{[\eB]} \uK \Big).
  \end{equation}
  Composition of profunctors is associative up to coherent
  isomorphism, with units as in \autoref{eg:unit-prof}, yielding a
  bicategory \VProf.
\end{lem}
\begin{proof}
  Straightforward.  The compatibility conditions for indexed
  coproducts (including the Beck-Chevalley condition), and the
  preservation of fiberwise coequalizers by restriction, are necessary
  to give $\underline{H\odot K}$ actions by $A$ and $C$, and to show
  associativity.
\end{proof}

In fact, this is a formal consequence of known results.  It is shown
in~\cite{shulman:frbi} that if \V has indexed coproducts preserved by
$\ten$, then it gives rise to a ``framed bicategory'', or equivalently
a ``proarrow equipment'' in the sense
of~\cite{wood:proarrows-i,wood:proarrows-ii}, which we may denote
$\lMat(\V)$.  If \V furthermore has fiberwise coequalizers, then we
can form the further equipment $\lMod(\lMat(\V))$ as defined
in~\cite{shulman:frbi} (which according to~\cite{gs:freecocomp} is the
free cocompletion of $\lMat(\V)$ under tight Kleisli objects).  The
equipment $\lMod(\lMat(\V))$ consists of small \V-categories,
\V-functors, and \V-profunctors, and so we might denote it
$\lprof{\V}$.  Its bicategory of proarrows is precisely \VProf.

If \V lacks indexed coproducts, then a construction analogous to that
in~\cite{shulman:frbi} produces instead a \emph{virtual equipment} in
the sense of~\cite{cs:multicats}, and thereby another virtual
equipment $\lprof{\V}$.  We leave the details to the reader, but we
note the following corollaries.  We refer the proofs to the cited
references, but no knowledge of equipments or framed bicategories will
be necessary for the rest of this paper, so the reader is free to take
these results on faith or to re-prove them by hand (which is not
difficult).

\begin{lem}\label{thm:equip}
  For \V-functors $f,g:A\to B$, there are natural bijections
  \begin{align}
    \VCat(A,B)(f,g) &\cong \VProf(A,B)(B(1,f),B(1,g))\\
    &\cong \VProf(B,A)(B(g,1),B(f,1)).
  \end{align}
  When \V has indexed coproducts preserved by $\ten$ and fiberwise
  coequalizers, the first of these is the action on 2-cells
  of a locally fully faithful pseudofunctor $\VCat \to \VProf$, which
  is the identity on objects and sends $f:A\to B$ to $B(1,f):A\hto B$.
  Furthermore, we have an adjunction $B(1,f) \dashv B(f,1)$ in \VProf.
\end{lem}
\begin{proof}
  The first statement is~\cite[Cor.~7.22]{cs:multicats} applied in the
  virtual equipment \VProf, together
  with~\cite[Prop.~6.2]{cs:multicats} to identify the left-hand side
  with \V-natural transformations as we have defined them.  The rest
  is~\cite[Props.~4.5 and~5.3]{shulman:frbi}.
\end{proof}

\begin{lem}
  If \V is an indexed cosmos, then \VProf is \emph{closed} in that we
  have natural isomorphisms
  \[ \VProf(A,C)(H\odot K, L) \cong \VProf(A,B)(H, K\rhd L)
  \cong \VProf(B,C)(K, L \lhd H).
  \]
\end{lem}
\proof%\begin{proof}
  By~\cite[Theorems~14.2 and~11.5]{shulman:frbi}, or as a direct
  construction using fiberwise equalizers:
  \begin{align}
    \U{K\rhd L} &= \mathrm{eq}\Big( \V^{[\e C]}(\uK,\uL)
    \;\toto\; \V^{[\e C]}(\uK,\V^{[\e C]}(\uC,\uL)) \Big)\\
    \U{L\lhd H} &= \mathrm{eq}\Big( \V^{[\e A]}(\uH,\uL)
    \;\toto\; \V^{[\e A]}(\uH,\V^{[\e A]}(\uA,\uL)) \Big).
    \tag*{\endproofbox}
  \end{align}
%\end{proof}

We also have a couple versions of the Yoneda lemma.

\begin{lem}\label{thm:yoneda1}
  For $f:A\to B$, $H:C\hto B$, and $K:B\hto C$, there are natural bijections
  \begin{align}
    \VProf(B(1,f),H) &\cong \VProf(A,H(f,1)) \mathrlap{\qquad\text{and}}\\
    \VProf(B(f,1),K) &\cong \VProf(A,K(1,f)).
  \end{align}
\end{lem}
\begin{proof}
  By~\cite[Theorems~7.16 and~7.20]{cs:multicats}.
\end{proof}

\begin{lem}\label{thm:yoneda2}
  If \V is an indexed cosmos, $f\maps A\to B$, $H\maps B\hto C$, and $K: C\hto B$
  then we have canonical isomorphisms
  \begin{equation}
    \begin{array}{rcccll}
      H(1,f) &\cong& H \lhd B(f,1) &\cong& B(1,f)\odot H & \text{and}\\
      K(f,1) &\cong& B(1,f) \rhd K &\cong& K \odot B(f,1).&
    \end{array}
  \end{equation}
\end{lem}
\begin{proof}
  By~\cite[Prop.~5.11]{shulman:frbi}.
\end{proof}

In particular, the Yoneda lemma implies the usual sort of hom-object
characterization of adjunctions.

\begin{prop}\label{thm:yoneda-adjn}
  For \V-functors $f\maps A\to B$ and $g\maps B\to A$, there is a
  bijection between
  \begin{enumerate}
  \item Adjunctions $f\adj g$ in \VCat, and
  \item Isomorphisms $B(f,1) \iso A(1,g)$ of \V-profunctors $B\hto A$.
  \end{enumerate}
\end{prop}
\begin{proof}
  Since $B(f,1)$ is always right adjoint to $B(1,f)$, to give $B(f,1)
  \iso A(1,g)$ is equivalent to giving data exhibiting $A(1,g)$ as
  right adjoint to $B(1,f)$.  By \autoref{thm:equip}, this is
  equivalent to data exhibiting $g$ as right adjoint to $f$.  (As
  stated, this argument requires the bicategory \VProf to exist, but
  for general \V we can translate it into the language of the virtual
  equipment $\lprof{\V}$.)
\end{proof}

We postpone further development of the theory of \V-categories until
we have a good notion of \emph{large} \V-category, to minimize
repetition.

\section{Indexed \V-categories}
\label{sec:indexed-cats}

As in \SS\ref{sec:small-cats}, let \S be a category with finite
products and \V an \S-indexed monoidal category, with corresponding
fibration $\tV\to\S$.  In this section we describe a notion of
``indexed \V-category'' which directly generalizes classical indexed
categories.

Recall that if $F\maps \bV\to\bW$ is a lax monoidal functor and $A$ is
a \bV-category, there is an induced \bW-category $F_\bullet A$ with
the same objects as $A$ and hom-objects defined by
$\U{F_\bullet A}(x,y) = F(\uA(x,y))$.
Note moreover that if $F\maps \bV\to\bW$ is
a \emph{closed} monoidal functor, then it can be regarded as a fully
faithful \bW-functor $F_\bullet: F_\bullet\bV\to\bW$.

\begin{defn}\label{def:locally-vcat}
  An \textbf{indexed \sV-category} \sA consists of:
  \begin{enumerate}
  \item For each $X\in \S$, a $\sV^X$-enriched category $\sA^X$.
  \item For each $f\maps X\to Y$ in \S, a fully faithful
    $\sV^X$-enriched functor $f^*\maps (f^*)_\bullet \sA^Y\to \sA^X$.
  \item For each $X\xto{f} Y \xto{g} Z$ in \S, a $\sV^X$-natural
    isomorphism
    \[ (gf)^* \toiso f^* \circ (f^*)_\bullet (g^*)\]
    (where we implicitly identify $(f^*)_\bullet (g^*)_\bullet \sA^Z$
    with $((gf)^*)_\bullet \sA^Z$ in the domains of these functors).
  \item For each $X\in\S$, a $\sV^X$-natural isomorphism $(1_X)^*
    \cong 1_{\sA^X}$.
  \item The following diagrams of isomorphisms commute:
    \[ \xymatrix{
      (hgf)^* \ar[r] \ar[d] &
      f^* \circ (f^*)_\bullet ((hg)^*) \ar[d]\\
      (gf)^* \circ ((gf)^*)_\bullet (h^*) \ar[d] &
      f^* \circ (f^*)_\bullet (g^* \circ (g^*)_\bullet (h^*) ) \ar[d]\\
      f^* \circ (f^*)_\bullet (g^*) \circ ((gf)^*)_\bullet (h^*) \ar[r] &
      f^* \circ (f^*)_\bullet (g^*) \circ (f^*)_\bullet (g^*)_\bullet (h^*).
    }\]
    \[\xymatrix{
      (f 1_X)^* \ar[r] \ar@{=}[dr] &
      (1_X)^* \circ ((1_X)^*)_\bullet (f^*) \ar[d] &
      f^* \circ (f^*)_\bullet ((1_Y)^*) \ar[d] &
      (1_Y f)^* \ar[l] \ar@{=}[dl] \\
      & f^* & f^*
    }\]
    (These are the same coherence conditions as for an ordinary indexed
    category or pseudofunctor, with some $(f^*)_\bullet$'s added to
    make things make sense.)
   \end{enumerate}
  An \textbf{indexed \sV-functor} $F\maps \sA\to\sB$ consists of
  $\sV^X$-enriched functors $F^X\maps \sA^X\to \sB^X$ together with
  isomorphisms
  \begin{equation}\label{eq:locally-v-functor-iso}
    F^X \circ f^* \iso f^* \circ (f^*)_\bullet(F^Y)
  \end{equation}
  such that the following diagrams of isomorphisms commute:
  \[ \xymatrix{
    F^X \circ (gf)^* \ar[r] \ar[d] &
    F^X \circ f^* \circ (f^*)_\bullet (g^*) \ar[r] &
    f^* \circ  (f^*)_\bullet (F^Y) \circ (f^*)_\bullet (g^*) \ar[d] \\
    (gf)^* \circ ((gf)^*)_\bullet(F^Z) \ar[d] &&
    f^* \circ (f^*)_\bullet (F^Y \circ g^*) \ar[d]\\
    f^* \circ (f^*)_\bullet(g^*) \circ (f^*)_\bullet (g^*)_\bullet(F^Z) \ar[rr] &&
    f^* \circ (f^*)_\bullet (g^* \circ (g^*)_\bullet(F^Z))
  }\]
  \[ \xymatrix@C=3pc{
    F^X \circ (1_X)^* \ar[dr] \ar[r] &
    (1_X)^* \circ ((1_X)^*)_\bullet(F^X) \ar[d] \\
    & F^X.
  }\]
  Finally, an \textbf{indexed \sV-natural transformation} consists of
  $\sV^X$-natural transformations $\alpha^X : F^X \to G^X$ such that
  the following diagram of isomorphisms commutes:
  \begin{equation}
    \vcenter{\xymatrix{
        F^X \circ f^*\ar[r]\ar[d] &
        f^* \circ (f^*)_\bullet(F^Y) \ar[d]\\
        G^X \circ f^* \ar[r] &
        f^* \circ (f^*)_\bullet(G^Y).
      }}
  \end{equation}
  We denote the resulting 2-category by \iVCAT.
\end{defn}

The required fully-faithfulness of $f^*\maps (f^*)_\bullet \sA^Y\to
\sA^X$ may seem odd.  The following example should help clarify the
intent.

\begin{eg}\label{eg:enriched-as-indexed}
  Let \bV be an ordinary monoidal category and \bA a (large)
  \bV-enriched category.  We define an indexed $\fam(\bV)$-category
  $\fam(\bA)$ where, for a set $X$, $\fam(\bA)^X$ is the
  $\bV^X$-enriched category of $X$-indexed families of objects of \bA.
  That is, for $(A_x)_{x\in X}$ and $(B_x)_{x\in X}$ with each
  $A_x,B_x \in \bA$, the hom-object in $\bV^X$ is defined by
  \begin{equation}
    \U{\fam(\bA)}^X(A,B)_x = \ubA(A_x,B_x).
  \end{equation}
  For a function $f:Y\to X$, the $\bV^Y$-enriched category
  $(f^*)_\bullet \sA$ has hom-objects in $\bV^Y$:
  \begin{equation}
    \U{\fam(\bA)}^X(A,B)_y = \ubA(A_{f(y)},B_{f(y)}).
  \end{equation}
  Finally, the functor $f^*\maps (f^*)_\bullet \sA^X\to \sA^Y$ sends
  an $X$-indexed family $(A_x)_{x\in X}$ to the $Y$-indexed family
  defined by $(f^*A)_y = A_{f(y)}$.  Thus we have
  \begin{equation}
    \U{\fam(\bA)}^Y(f^*A,f^*B)_y = \ubA((f^*A)_y,(f^*B)_y)
    = \ubA(A_{f(y)},B_{f(y)}) = \U{\fam(\bA)}^X(A,B)_{f(y)}
  \end{equation}
  and hence $\U{\fam(\bA)}^Y(f^*A,f^*B) = f^*(\U{\fam(\bA)}^X(A,B))$,
  so that $f^*\maps (f^*)_\bullet \sA^X\to \sA^Y$ is fully faithful.
  This construction defines a 2-functor
  \[ \fam : \CAT{\bV} \to \icat{\fam(\bV)}.
  \]
  It is not essentially surjective, but it induces an equivalence on
  hom-categories (i.e.\ it is bicategorically fully-faithful).  We
  could characterize its essential image by imposing ``stack''
  conditions.
\end{eg}

Here are a few other important examples.

\begin{eg}\label{eg:v-as-indexed}
  If \V is symmetric and closed as in
  \autoref{thm:three-homs}\ref{item:closed-1}, then we can regard it
  as an indexed \V-category, by regarding each closed symmetric
  monoidal category $\V^X$ as enriched over itself, and each closed
  monoidal functor $f^*: \V^Y \to \V^X$ as a fully faithful
  $\V^X$-enriched functor $(f^*)_\bullet \V^Y \to \V^X$.
\end{eg}

\begin{eg}\label{eg:locintern}
  If \S has finite limits, then an indexed $\self(\S)$-category is
  precisely a \emph{locally internal category} over \S, as defined
  in~\cite{penon:locintern} or~\cite[\SS B2.2]{ptj:elephant}, and
  similarly for functors and transformations.
\end{eg}

\begin{eg}\label{eg:pht}
  Indexed \sK-categories and $\sK_*$-categories, where \sK and $\sK_*$
  are the indexed cosmoi from Examples~\ref{eg:top-mf}
  and~\ref{eg:ret-mf}, are ubiquitous throughout~\cite{maysig:pht}.
\end{eg}

\begin{eg}\label{eg:monads}
  Let \sA be an indexed \V-category, and $T$ a monad on \sA in the
  2-category \iVCAT.  This is easily seen to consist of
  \begin{enumerate}
  \item A $\sV^X$-enriched monad $T^X$ on $\sC^X$, for every $X$, and
  \item Isomorphisms $f^* \circ (f^*)_\bullet(T^Y) \iso T^X \circ f^*$
    which simultaneously make $T$ into an indexed \sV-functor and
    $f^*$ into a morphism of $\sV^X$-enriched monads from
    $(f^*)_\bullet(T^Y)$ to $T^X$.
    % (meaning that the obvious
    % compatibility diagrams with the multiplication and unit
    % transformations commute).
  \end{enumerate}
  We can then form the Eilenberg-Moore object $\sAlg(T)$ in \iVCAT.
  Explicitly, $\sAlg(T)^X = \sAlg(T^X)$, with transition functors
  induced by the above morphisms of monads.

  For instance, if \V is an indexed cosmos and $R$ is a monoid in
  $\sV^1$, then there is a \V-monad on \V defined by $R^X A =
  \pi_X^*R\ten_X A$, whose algebras in $\sV^X$ are $\pi_X^*R$-modules.

  As another example, if \sV is an indexed cartesian cosmos with
  fiberwise countable colimits, then there is a \sV-monad $T$ on \sV
  for which $\sAlg(T)^X$ is the $\sV^X$-enriched category of monoids
  in $\V^X$.  More generally, we can consider algebras for any
  finite-product theory.

  Finally, if \sV is a cartesian cosmos with countable coproducts and
  $C$ is an \emph{operad} in $\sV^1$ as
  in~\cite{may:goils,kelly:operads}, then $\pi_X^*C$ is an operad in
  $\sV^X$, for any $X$.  The induced monad $\widehat{\pi_X^*C}$ on
  $\V^X$ is $\V^X$-enriched, because \sV\ is cartesian, and as $X$
  varies these fit together into a \sV-monad $\Chat$ on \V.  We thus
  obtain a \sV-category of $C$-algebras.  Taking \V to be \sK as in
  \autoref{eg:top-mf}, we obtain \sV-categories of parametrized
  $A_\infty$- and $E_\infty$-spaces.
\end{eg}

\begin{eg}\label{eg:indpsh}
  If \S has finite products and \bV is an ordinary monoidal category
  as in \autoref{eg:psh}, then an indexed $\sPsh(\S,\bV)$-category \sA
  consists of, in particular, for each $X\in\S$, a
  $\bV^{(\S/X)\op}$-enriched category $\sA^X$.  However, a
  $\bV^{(\S/X)\op}$-enriched category is equivalently a functor $\sA^X
  : (\S/X)\op \to \CAT{\bV}$ whose image consists of functors that are
  the identity on objects.

  Moreover, by definition of the functor $f^*:\bV^{(\S/X)\op} \to
  \bV^{(\S/Y)\op}$ and by full-faithfulness of $f^*: (f^*)_\bullet
  \sA^Y \to \sA^X$, for any objects $a,b$ of $\sA^Y$ the hom-object
  $\usA^Y(f)(a,b) \in \bV$ must be isomorphic to
  $\usA^X(1_X)(f^*a,f^*b)$, and similarly for all the category
  structure.  Thus, the \bV-category $\sA^Y(f)$ is completely
  determined by the \bV-category $\sA^X(1_X)$ and the function
  $f^*:\mathrm{ob}(\sA^Y(1_Y)) \to \mathrm{ob}(\sA^X(1_X))$.

  Furthermore, the action on hom-objects of the functors in the image
  of the functor $\sA^X : (\S/X)\op \to \CAT{\bV}$ assemble exactly
  into an extension of this function on objects to a \bV-functor
  $\sA^Y(1_Y) \to \sA^X(1_X)$.  Adding in the pseudofunctoriality
  constraints, we see that an indexed $\sPsh(\S,\bV)$-category is
  equivalently an ordinary pseudofunctor $\S\op \to \CAT{\bV}$.  In
  fact, the 2-category $\icat{\sPsh(\S,\bV)}$ is 2-equivalent to the
  2-category $[\S\op,\CAT{\bV}]$ of pseudofunctors, pseudonatural
  transformations, and modifications.

  In particular, an indexed $\sPsh(\S,\nSet)$-category is merely an
  ordinary \S-indexed category (with locally small fibers).  It is
  known (e.g.~\cite[B2.2.2]{ptj:elephant}) that locally internal
  categories can be identified with indexed categories that are
  ``locally small'' in an indexed sense.  This corresponds to
  identifying indexed $\self(\S)$-categories with indexed
  $\sPsh(\S,\nSet)$-categories whose hom-presheaves $\usA^X(a,b) \in
  \nSet^{(\S/X)\op}$ are all representable.  (The connection with the
  usual definition of ``locally small indexed category'' will be more
  evident in \SS\ref{sec:large-cats}.)
\end{eg}

\begin{rmk}
  In~\cite{gg:fib-rel} indexed \sV-categories are called
  \emph{\sV-enriched fibrations}.  We have chosen a different
  terminology because indexed \sV-categories are more analogous to
  pseudofunctors than to fibrations.  In \SS\ref{sec:v-fibrations} we
  will see a more `fibrational' approach.
\end{rmk}

It is straightforward to define indexed \V-profunctors as well.

\begin{defn}\label{def:ivprof}
  For indexed \V-categories \sA and \sB, an \textbf{indexed
    \V-profunctor} consists of a $\V^X$-enriched profunctor $H^X:\sA^X
  \hto \sB^X$ for each $X$, together with isomorphisms $(f^*)_\bullet
  H^Y \cong H^X$ satisfying evident coherence axioms.  We obtain a
  category $\iVPROF(\sA,\sB)$ of indexed \V-profunctors.
\end{defn}

However, rather than develop the theory of indexed \V-categories any
further here, we will instead move on to a more general notion which
includes both small \V-categories and indexed \V-categories as special
cases.

\section{Large \V-categories}
\label{sec:large-cats}

Continuing with our minimal assumptions that \S has finite products
and \V is an \S-indexed monoidal category, we will now define a
different sort of ``large \V-category'' which more obviously includes
the small ones from \SS\ref{sec:small-cats}.  The relationship of
these large \V-categories with indexed \V-categories is less
immediately clear, but will turn out to be similar to the relationship
between fibrations and classical indexed categories.

\begin{defn}\label{def:large-vcat}
  A \textbf{large \sV-category} \sA\ consists of
  \begin{enumerate}
  \item A collection of objects $x,y,z,\dots$.
  \item For each object $x$, an object $\ex \in \bS$, called its
    \emph{extent}.
  \item For each $x,y$, an object $\usA(x,y)$ of $\V^{\e y\times \e x}$.
  \item For each $x$, a morphism in \tV:
    \begin{equation}\label{eq:identities-large}
      \xymatrix{\I_{\ex} \ar[r]^<>(.5){\mathrm{ids}}\ar@{ |-> }[d] &
        \usA(x,x) \ar@{ |-> }[d]\\
        \e x\ar[r]_<>(.5){\Delta} & \e x\times \e x}
    \end{equation}
  \item For each $x,y,z$, a morphism in \tV:
    \begin{equation}\label{eq:composition-large}
      \xymatrix{(\usA(y,z) \ten_{\ey} \usA(x,y))
        \ar[r]^<>(.5){\mathrm{comp}}\ar@{ |-> }[d] &
        \usA(x,z) \ar@{ |-> }[d]\\
        \ez\times \ey\times \ex \ar[r]_-{\pi_{\ey}} &
        \ez\times \ex}
    \end{equation}
  \item Composition is associative and unital, just as in
    \autoref{def:small-vcat}.
  \end{enumerate}
  If \sA\ and \sB\ are large \sV-categories, a \textbf{\sV-functor}
  $f\maps \sA\to\sB$ consists of:
  \begin{enumerate}
  \item For each object $x$ of \sA, an object $fx$ of \sB\ and a
    morphism $\e f_x\maps \ex \to \e(fx)$ in \bS.
  \item For each pair $x,y$ in \sA, a morphism in \tV:
    \[\xymatrix{\usA(x,y) \ar[rr]^<>(.5){f_{xy}}\ar@{ |-> }[d] &&
      \usB(fx,fy)\ar@{ |-> }[d]\\
      \ey\times \ex\ar[rr]_-{\e f_y \times \e f_x} &&
      \e(fy) \times \e (fx).}\]
  \item Composition and identities are preserved, as in
    \autoref{def:small-vfunctor}.
  \end{enumerate}
  If $f,g\maps \sA\to\sB$ are \sV-functors between large
  \sV-categories, a \textbf{\sV-natural transformation} $\alpha\maps
  f\to g$ consists of, for each object $x$, a morphism
  \[\xymatrix{\I_\ex \ar[rr]^<>(.5){\alpha_x}\ar@{ |-> }[d] &&
    \usB(fx,gx) \ar@{ |-> }[d]\\
    \e x \ar[rr]_-{(\e g_x,\e f_x)} &&
    \e(gx)\times \e(fx)}\]
  such that for all $x,y$ a diagram analogous to that in
  \autoref{def:small-vnat} commutes.
\end{defn}

\begin{lem}
  Large \sV-categories form a 2-category \VCAT, which contains the
  2-category \VCat of small \sV-categories as the full sub-2-category
  of large \V-categories with exactly one object.
\end{lem}
\begin{proof}
  Just like \autoref{thm:vcat}.
\end{proof}

The most obvious non-small example comes from \V itself.

\begin{eg}
  If \V is closed as in \autoref{thm:three-homs}\ref{item:closed-3},
  then we can define a large \V-category whose objects are the objects
  of \tV (that is, the disjoint union of the objects of the categories
  $\V^X$) and whose hom-objects are the external ones $\uV(x,y)$.
  (Recall from \autoref{rmk:weaker-homs} that the external-homs can be
  defined in terms of the fiberwise ones for arbitrary \V.  This
  explicit definition suffices to make them into a large \V-category.)
\end{eg}

\begin{rmk}\label{rmk:disc-large}
  When \V has indexed coproducts preserved by $\ten$, we have the
  ``discrete \V-category'' 2-functor
  \begin{equation}
    \delta:\S\to \VCat \into \VCAT.
  \end{equation}
  As in Remarks~\ref{rmk:disc-functors} and~\ref{rmk:disc-vnat}, we
  can make sense of ``\V-functors $f:\delta X\to \sA$'' and
  ``\V-natural transformations $\al:f\to g$'' between them, even when
  \V lacks indexed coproducts.  Namely, the former is simply a choice
  of an object $a\in\sA$ and a morphism $\e f_a:X\to\ea$ in \S, while
  the latter is simply a morphism $\I_X \to \usA(a,b)$ in \tV lying
  over $(\e g_b,\e f_a)$.  In particular, any large \V-category \sA
  induces an \S-indexed category
  \[ \VCAT(\delta-,\sA) : \S\op\to\nCat. \]
  Again, in \SS\ref{sec:cocuf} we will identify this with a special
  case of ``change of cosmos''.
\end{rmk}

\begin{defn}
  If \ka is a regular cardinal, then we say a large \V-category \sA is
  \textbf{\ka-small} if its collection of objects is a set of
  cardinality $<\ka$.  We say \sA is \textbf{\infty-small} or
  \textbf{set-small} if its collection of objects is a small set (of
  any cardinality).
\end{defn}

Note that if \S is a small ordinary category, then a \V-category \sA
is set-small if and only if for each $X\in\S$, there is a small set
of objects of \sA having extent $X$.

\begin{rmk}
  If we allow \ka to be an ``arity class'' in the sense
  of~\cite{shulman:exsite}, then according to this definition the
  ``small \V-categories'' of \SS\ref{sec:small-cats} may be called
  ``$\{1\}$-small''.
\end{rmk}

\begin{rmk}\label{rmk:large-as-small}
  Recall from \autoref{eg:famv} that an \S-indexed monoidal category
  \V gives rise to a $\nFam(\S)$-indexed monoidal category $\fam(\V)$.
  It is straightforward to identify set-small \V-categories, as
  defined above, with small $\fam(\V)$-categories, as defined in
  \SS\ref{sec:small-cats}.  (Of course, by allowing ``large families''
  we could include all large \V-categories.)

  In particular, for a classical monoidal category \bV, we can
  identify set-small $\sConst(\star,\bV)$-categories with small
  \bV-enriched categories in the classical sense---while we have
  already observed in \autoref{eg:famv} that $\fam(\bV) =
  \fam(\sConst(\star,\bV))$, and in \autoref{eg:enriched} that small
  $\fam(\bV)$-categories (in the sense of \SS\ref{sec:small-cats}) can
  also be identified with small \bV-enriched categories.

  Thus, although on the one hand small \V-categories are evidently a
  special case of large ones, we could equivalently regard large
  \V-categories as a special case of small ones, by the expedient of
  changing \V.
  % However, while this might simplify things formally, it
  % would make the exposition rather more opaque, so we have eschewed it.
\end{rmk}

The theory of profunctors from \SS\ref{sec:small-cats} also
generalizes to large \V-categories, as long as we pay appropriate
attention to size questions.

\begin{defn}
  For large \V-categories \sA and \sB, a \textbf{\V-profunctor}
  $H:\sA\hto \sB$ consists of:
  \begin{enumerate}
  \item For each pair of objects $a$ of \sA and $b$ of \sB, an object
    $\uH(b,a) \in \V^{\e a\times \e b}$.
  \item Morphisms in \tV:
    \[\vcenter{\xymatrix@C=3pc{\uA(a,a') \ten_{\e a} \uH(b,a)
        \ar[r]^-{\mathrm{act}}\ar@{ |-> }[d] &
        \uH(b,a') \ar@{ |-> }[d]\\
        \e a'\times\e a\times \e b
        \ar[r]_-{\pi_{\ea}} &
        \e a'\times\e b}}
    \qquad\text{and}\quad
    \vcenter{\xymatrix@C=3pc{\uH(b,a) \ten_\eb \uB(b',b)
        \ar[r]^-{\mathrm{act}}\ar@{ |-> }[d] &
        \uH(b',a) \ar@{ |-> }[d]\\
        \e a\times\e b\times \eb'\ar[r]_-{\pi_{\eb}} &
        \ea\times\eb}}
    \]
  \item The following diagrams commute.
    \[\vcenter{\xymatrix{
        \uA(a',a'') \ten_{\ea'} \uA(a,a') \ten_\ea \uH(b,a)
        \ar[r]^-{1\ten\mathrm{act}}\ar[d]_{\mathrm{comp}\ten 1} &
        \uA(a',a'')\ten_{\ea'} \uH(b,a')\ar[d]^{\mathrm{act}}\\
        \uA(a,a'')\ten_\ea \uH(b,a)\ar[r]_-{\mathrm{act}}
        & \uH(b,a'')}}
    \]
    \[\vcenter{\xymatrix{
        \uH(b,a) \ten_\eb \uB(b',b) \ten_{\eb'} \uB(b'',b')
        \ar[r]^-{\mathrm{act}\otimes 1}\ar[d]_{1\otimes \mathrm{comp}} &
        \uH(b',a)\ten_{\eb'} \uB(b'',b')\ar[d]^{\mathrm{act}}\\
        \uH(b,a)\ten_\eb \uB(b'',b)\ar[r]_-{\mathrm{act}}
        & \uH(b'',a)}}\]
    \[\vcenter{\xymatrix{
        \uH(b,a) \ar[r]^-{\mathrm{ids}}\ar[dr]_{=} &
        \uA(a,a)\ten_\ea\uH(b,a) \ar[d]^{\mathrm{act}}\\
        & \uH(b,a)}}
    \qquad
    \vcenter{\xymatrix{\uH(b,a) \ar[r]^-{\mathrm{ids}}\ar[dr]_{=} &
        \uH(b,a)\ten_\eb\uB(b,b) \ar[d]^{\mathrm{act}}\\
        & \uH(b,a)}}\]
    \[\vcenter{\xymatrix{
        \uA(a,a')\ten_\ea \uH(b,a) \ten_\eb \uB(b',b)
      \ar[r]^-{\mathrm{act}\otimes 1}\ar[d]_{1\otimes\mathrm{act}} &
      \uH(b,a') \ten_\eb \uB(b',b) \ar[d]^{\mathrm{act}}\\
      \uA(a,a')\ten_\ea \uH(b',a) \ar[r]_-{\mathrm{act}} &
      \uH(b',a).}}
    \]
  \end{enumerate}
  We have an obvious notion of \textbf{morphism of profunctors},
  yielding a category which we denote $\VPROF(A,B)$.
\end{defn}

Of course, when \sA and \sB are small, this reduces exactly to
\autoref{def:smallprof}.

\begin{eg}\label{eg:unitprof}
  Any large \V-category \sA has a \emph{unit profunctor} $\sA:\sA\hto
  \sA$ which is made up of its hom-objects.
\end{eg}

\begin{eg}
  If $H:\sA\hto\sB$ is a \V-profunctor and $f:\sA'\to\sA$ and
  $g:\sB'\to\sB$ are \V-functors, then there is a \V-profunctor
  $H(g,f):\sA'\hto\sB'$ defined by
  \[ \U{H(g,f)}(b,a) = (\e f \times \e g)^* \big(\uH(g b, f a)\big) \]
  In particular, for a \V-functor $f:\sA\to\sB$, we have the
  \emph{representable \V-profunctors} $\sB(1,f):\sA\hto \sB$ and
  $\sB(f,1):\sB\hto\sA$.
\end{eg}

To define the composite of \V-profunctors $H:\sA\hto\sB$ and
$K:\sB\hto \sC$ in general, we need \V to have fiberwise colimits of
the size of the collection of objects of \sB.  Therefore, if this
collection is itself large (in the sense of the ambient set theory),
we cannot expect such composites to exist.  For this reason, it is
useful to introduce the following notion.

\begin{defn}
  For \V-profunctors $H:\sA\hto\sB$, $K:\sB\hto \sC$, and $L:\sA\hto
  \sC$, a \textbf{bimorphism} $\phi:H,K\to L$ consists of
  \begin{enumerate}
  \item For each $a,b,c$, a morphism in \tV:
    \begin{equation}
      \vcenter{\xymatrix{
          \uH(b,a) \otimes_{\eb} \uK(c,b) \ar[r]^-{\phi_{abc}}\ar@{ |-> }[d] &
          \uL(c,a)\ar@{ |-> }[d]\\
          \ea \times \eb\times \ec\ar[r] &
          \ea \times \ec
        }}
    \end{equation}
  \item The following diagrams commute:
    \begin{equation}\label{eq:bimorr}
      \vcenter{\xymatrix@C=3pc{
          \uH(b,a) \otimes_{\eb} \uK(c,b) \otimes_{\ec} \usC(c',c)
          \ar[r]^-{1\otimes\mathrm{act}}\ar[d]_{\phi\otimes 1} &
          \uH(b,a) \otimes_{\eb} \uK(c',b)\ar[d]^{\phi}\\
          \uL(c,a) \otimes_{\ec} \usC(c',c)\ar[r]_-{\mathrm{act}} &
          \uL(c',a)
        }}
    \end{equation}
    \begin{equation}\label{eq:bimorl}
      \vcenter{\xymatrix@C=3pc{
          \usA(a,a') \otimes_{\ea} \uH(b,a) \otimes_{\eb} \uK(c,b)
          \ar[r]^-{\mathrm{act}\otimes 1}\ar[d]_{1\otimes\phi} &
          \uH(b,a') \otimes_{\eb} \uK(c,b)\ar[d]^{\phi}\\
          \usA(a,a') \otimes_{\ea} \uL(c,a)\ar[r]_-{\mathrm{act}} &
          \uL(c,a')
        }}
    \end{equation}
    \begin{equation}\label{eq:bimorm}
      \vcenter{\xymatrix@C=3pc{
          \uH(b,a) \otimes_{\eb} \usB(b',b) \otimes_{\eb'} \uK(c,b')
          \ar[r]^-{1\otimes\mathrm{act}}\ar[d]_{\mathrm{act}\otimes 1} &
          \uH(b,a) \otimes_{\eb} \uK(c,b)\ar[d]^\phi\\
          \uH(b',a) \otimes_{\eb'} \uK(c,b')\ar[r]_-{\phi} &
          \uL(c,a)
        }}
    \end{equation}
  \end{enumerate}
\end{defn}

We write $\VBimor(H,K;L)$ for the set of bimorphisms $H,K\to L$.  It
is evident that such bimorphisms can be composed with ordinary
morphisms of profunctors $L\to L'$, $H'\to H$, and $K'\to K$, yielding
a functor
\[ \VBimor(-,-;-) : \VPROF(\sA,\sB)\op \times \VPROF(\sB,\sC)\op
\times \VPROF(\sA,\sC) \too \nSet
\]
We can now characterize some conditions under which composites exist.

\begin{lem}\label{thm:vproflargecomp}
  If \V has indexed coproducts, fiberwise coequalizers, and fiberwise
  \ka-small coproducts, all preserved by $\ten$, and \sB is \ka-small,
  then for any $H:\sA\hto\sB$ and $K:\sB\hto\sC$, the functor
  $\VBimor(H,K;-)$ is representable by some $H\odot K:\sA\hto\sC$.
\end{lem}
\proof%\begin{proof}
  As in \autoref{thm:smallprofcomp}, define $\underline{H\odot
    K}(c,a)$ to be the coequalizer of the parallel pair
  \begin{equation}
    \coprod_{b,b'\in \mathrm{ob}\sB}
    \big(\uH(b,a) \ten_{[\eb]} \usB(b',b) \ten_{[\eb']} \uK(c,b')\big)
    \toto
    \coprod_{b\in\mathrm{ob}\sB}
    \big(\uH(b,a) \ten_{[\eb]} \uK(c,b)\big).
    \tag*{\endproofbox}
  \end{equation}
%\end{proof}

We also have dual constructions of right and left homs, under
analogous hypotheses.

\begin{lem}\label{thm:vproflargeclosed}
  Suppose \V is symmetric and closed and has indexed products,
  fiberwise equalizers, and fiberwise \ka-small products, and we have
  \V-profunctors $H:\sA\hto \sB$, $K:\sB\hto\sC$, and and $L:\sA\hto
  \sC$.
  \begin{enumerate}
  \item If \sC is \ka-small, then $\VBimor(-,K;L)$ is
    representable by some $K\rhd L : \sA \hto \sB$.\label{item:rhom}
  \item If \sA is \ka-small, then $\VBimor(H,-;L)$ is
    representable by some $L\lhd H: \sB\hto \sC$.\label{item:lhom}
  \end{enumerate}
\end{lem}
\proof%\begin{proof}
  We define $\U{K\rhd L}(b,a)$ to be the fiberwise equalizer of the
  following parallel pair of maps between fiberwise products:
  \begin{equation}\label{eq:largerhd}
    \prod_{c\in\mathrm{ob}\sC}
    \V^{[\e c]}\big(\uK(c,b),\uL(c,a)\big)
    \;\toto\;
    \prod_{c,c'\in\mathrm{ob}\sC}
    \V^{[\e c]}\Big(\uK(c,b),
    \V^{[\e c']}\big(\usC(c',c),\uL(c',a)\big)\Big)
  \end{equation}
  and similarly for $\U{L\lhd H}(c,b)$ we use:
  \begin{equation}
    \prod_{a\in\mathrm{ob}\sA} \V^{[\e a]}\big(\uH(b,a),\uL(c,a)\big)
    \;\toto\;
    \prod_{a,a'\in\mathrm{ob}\sA}
    \V^{[\e a]}\Big(\uH(b,a),
    \V^{[\e a']}\big(\usA(a,a'),\uL(c,a')\big)\Big).
    \tag*{\endproofbox}
  \end{equation}
%\end{proof}

\begin{rmk}\label{rmk:disc-homs}
  If \sC is a discrete \V-category $\delta Z$, then its actions on $K$
  and $L$ are trivial, and so the two morphisms in~\eqref{eq:largerhd}
  are in fact equal.  Thus, in this case we have $\U{K\rhd L}(b,a) =
  \V^{[Z]}(\uK(\star,b),\uL(\star,a))$.  Of course, similar
  observations hold when \sA is discrete, or in
  \autoref{thm:vproflargecomp} when \sB is discrete.
\end{rmk}

\begin{eg}\label{thm:psh-comphom}
  Recall from \autoref{eg:psh} that when $\V=\sPsh(\S,\bV)$, with \S a
  locally small category with pullbacks and \bV a classical cosmos, we
  do not have all indexed coproducts and homs, but only those
  satisfying some smallness conditions.  Let us say that a
  $\sPsh(\S,\bV)$-category \sA is \textbf{locally small} if each
  hom-object $\usA(a,a')$ is small in the sense of \autoref{eg:psh},
  and likewise that a $\sPsh(\S,\bV)$-profunctor $H:\sA\hto\sB$ is
  \textbf{locally small} if each $\uH(b,a)$ is small.  Then the proof
  of \autoref{thm:vproflargecomp} goes through as long as we assume
  additionally that \sB, $H$, and $K$ are locally small.  Likewise,
  \autoref{thm:vproflargeclosed}\ref{item:rhom} holds when \sC and $K$
  are locally small, and~\ref{item:lhom} holds when \sA and $H$ are
  locally small.

  The $\sPsh(\S,\bV)$-categories that arise ``in nature'' are
  generally not locally small.  However, we will see in
  \SS\ref{sec:indexed-limits} that locally small
  $\sPsh(\S,\bV)$-categories and profunctors are useful for describing
  weighted limits and colimits.
\end{eg}

These representing objects actually have a stronger universal
property, which is necessary (for instance) to show that they are
associative.  To express this property, we first observe that for
\V-profunctors $H_i:\sA_i \hto \sA_{i+1}$ and $K:\sA_0 \hto \sA_n$,
there is a more general notion of a \textbf{multimorphism}
$\phi:H_1,\dots,H_n \to K$.  This has components
\begin{equation}
  \vcenter{\xymatrix@C=4pc{
      \uH_1(a_1,a_0) \otimes_{\e a_1} \cdots \otimes_{\e a_n} \uH_n(a_n,a_{n-1})
      \ar[r]^-{\phi_{a_0,\dots,a_n}}\ar@{ |-> }[d] &
      \uK(a_n,a_0)\ar@{ |-> }[d]\\
      \ea_0 \times \cdots \times \ea_n\ar[r] &
      \ea_0 \times \ea_n
      }}
\end{equation}
satisfying axioms similar to~\eqref{eq:bimorr} and~\eqref{eq:bimorl}
for the actions of $\sA_0$ and $\sA_n$, and axioms similar
to~\eqref{eq:bimorm} for the actions of $\sA_1$ through $\sA_{n-1}$.
In the case $n=0$, the components of a multimorphism $\phi:()\to K$
are
\begin{equation}
  \vcenter{\xymatrix@C=4pc{
      \I_{\ea}
      \ar[r]^-{\phi_{a}}\ar@{ |-> }[d] &
      \uK(a,a)\ar@{ |-> }[d]\\
      \ea \ar[r]_-{\Delta} &
      \ea \times \ea
      }}
\end{equation}
and its only axioms are of the form~\eqref{eq:bimorr}
and~\eqref{eq:bimorl}.  We write $\Vmmor(H_1,\dots,H_n;K)$ for the set
of multimorphisms $H_1,\dots,H_n \to K$.  Multimorphisms can obviously
also be composed with ordinary morphisms, and also with each other in
a multicategory-like way, e.g.\ given $\phi:H_1,H_2 \to K_1$ and
$\psi:K_1,K_2 \to L$ we have $\psi(\phi,1):H_1,H_2,K_2 \to L$.

Formally, multimorphisms are the 2-cells of another virtual equipment,
which we denote \Lprof{\V}.  The following
is~\cite[Def.~5.1]{cs:multicats}.

\begin{defn}\label{def:composite}
  A \textbf{composite} of \V-profunctors $H:\sA\hto \sB$ and
  $K:\sB\hto\sC$ is a \V-profunctor $H\odot K : \sA \hto \sC$ and a
  bimorphism $\phi:H,K\to H\odot K$ such that composing with $\phi$
  induces bijections
  \begin{multline}
    \Vmmor(L_1,\dots,L_n,H\odot K,M_1,\dots,M_m) \toiso\\
    \Vmmor(L_1,\dots,L_n,H,K,M_1,\dots,M_m)
  \end{multline}
  for all well-typed $L_i,M_j$.
\end{defn}

When $n=m=0$, this just says that $H\odot K$ represents the functor
$\VBimor(H,K;-)$, so being a composite is a strengthening of that
universal property.  It is easy to verify that the proof of
\autoref{thm:vproflargecomp} actually constructs composites in the
sense of \autoref{def:composite}.

\begin{defn}
  A \textbf{left hom} of \V-profunctors $H:\sA\hto\sB$ and
  $K:\sA\hto\sC$ is a \V-profunctor $K\lhd H:\sB\hto \sC$ and a
  bimorphism $\phi:H,(K\lhd H) \to K$ such that composing with $\phi$
  induces bijections
  \[ \Vmmor(L_1,\dots,L_n;K\lhd H) \toiso
  \Vmmor(H,L_1,\dots,L_n;K)
  \]
  for all well-typed $L_i$.
  Similarly, a \textbf{right hom} of $H:\sB\hto\sC$ and
  $K:\sA\hto\sC$ is a \V-profunctor $H\rhd K:\sA\hto \sB$ and a
  bimorphism $\phi:(H\rhd K),H \to K$ which induces bijections
  \[ \Vmmor(L_1,\dots,L_n;H\rhd K) \toiso
  \Vmmor(L_1,\dots,L_n,H;K).
  \]
\end{defn}

Again, when $n=1$ these definitions reproduce the universal property
stated in \autoref{thm:vproflargeclosed}, whereas the proof of that
lemma produces objects with this stronger universal property.

\begin{rmk}\label{thm:objwise-comp}
  The omitted verification in Lemmas~\ref{thm:vproflargecomp}
  and~\ref{thm:vproflargeclosed} that the given objects do, in fact,
  form a profunctor with the desired property, applies verbatim to
  show that if $H:\sA\hto\sB$ and $K:\sB\hto\sC$ are \V-profunctors
  such that the composite $H(1,a) \odot K(c,1)$ exists for all
  $a\in\sA$ and $c\in\sC$, then the composite $H\odot K$ also exists.
  There is a similar result for homs.
\end{rmk}

% We will say that a large \V-category \sB is \textbf{effectively small}
% if composites exist for all pairs of \V-profunctors $H:\sA\hto \sB$
% and $K:\sB\hto \sC$.  Similarly, we will say \sB is
% \textbf{effectively co-small} if left homs exist for all
% $H:\sB\hto\sA$ and $K:\sB\hto\sC$ and also right homs exist for all
% $H:\sC\hto\sB$ and $K:\sA\hto\sB$.  If \V lacks indexed coproducts,
% then we also include explicitly in these definitions the case when \sA
% and/or \sC is ``discrete'', with profunctors interpreted as in
% \autoref{rmk:disc-large}.  Thus:
% \begin{enumerate}
% \item If \V is an indexed cosmos, then any small \V-category is
%   effectively small and effectively co-small.
% \item If \V is a fiberwise cocomplete indexed cosmos, then any
%   set-small \V-category is effectively small.
% \item Dually, if \V is a fiberwise complete indexed cosmos, then any
%   set-small \V-category is effectively co-small.
% \end{enumerate}

We also note that the unit profunctors $\sA:\sA\hto\sA$ from
\autoref{eg:unitprof} have an analogous universal property.

\begin{lem}\label{thm:large-units}
  For any large \V-category \sA, there is a multimorphism $\phi:() \to
  \sA$ such that composing with $\phi$ induces bijections
  \[ \Vmmor(\vec L,\sA,\vec M;N)\toiso \Vmmor(\vec L,\vec M;N)
  \]
  for all well-typed $\vec L = L_1,\dots,L_n$ and $\vec M =
  M_1,\dots,M_m$ and $N$.
\end{lem}
\begin{proof}
  By~\cite[Prop.~5.5]{cs:multicats}.
\end{proof}

One value of these stronger universal properties is that they
automatically imply associativity and unitality of composites
and homs.

\begin{lem}\label{thm:prof-assoc}
  If all necessary composites and homs exist, then we have the
  following isomorphisms:
  \begin{align}
    (H\odot K)\odot L &\cong H\odot (K\odot L)\label{eq:profassoc}\\
    (H\odot K) \rhd L &\cong H\rhd (K\rhd L)\\
    (H\lhd K) \lhd L &\cong H \lhd (K\odot L)\\
    (H\rhd K) \lhd L &\cong H \rhd (K \lhd L).
  \end{align}
\end{lem}
\begin{proof}
  We prove only~\eqref{eq:profassoc}; the others are similar.  For any
  appropriately typed $M$, we have
  \begin{align}
    \VPROF((H\odot K)\odot L,M)
    &\cong \VBimor((H\odot K), L;M)\\
    &\cong \Vmmor(H,K,L;M)\\
    &\cong \VBimor(H, K\odot L;M)\\
    &\cong \VPROF(H\odot (K\odot L),M).
  \end{align}
  Thus, the Yoneda lemma gives~\eqref{eq:profassoc}.
\end{proof}

\begin{lem}\label{thm:large-unit-comp}
  For $H:\sA\hto \sB$, the composites $H\odot \sB$ and $\sA\odot H$,
  and the homs $\sB\rhd H$ and $H\lhd \sA$, all exist and are
  canonically isomorphic to $H$.
\end{lem}
\begin{proof}
  Just like the previous lemma.
\end{proof}

\begin{rmk}\label{rmk:comphom-abuse}
  We will henceforth abuse language by writing ``$H\odot K\cong L$''
  to mean that the composite $H\odot K$ exists and is canonically
  isomorphic to $L$, and similarly for left and right homs.
\end{rmk}

\begin{lem}\label{thm:coyoneda}
  For a \V-profunctor $H:\sA\hto \sB$ and a \V-functor $f:\sA'\to\sA$,
  we have $\sA(1,f)\odot H \cong H(1,f)$.  Similarly, for
  $g:\sB'\to\sB$, we have $H\odot\sB(g,1) \cong H(g,1)$.
\end{lem}
\begin{proof}
  By~\cite[Theorem~7.16]{cs:multicats}.
\end{proof}

In particular, for $\sA\xto{f}\sB \xto{g}\sC$, we have $\sB(1,f) \odot
\sC(1,g) \cong \sC(1,gf)$, so that representable profunctors are
``pseudofunctorial'', even though in general, \V-profunctors between
large \V-categories do not form a bicategory.  The same proof as in
\autoref{thm:equip} shows that this functor is fully faithful, i.e.\
we have natural bijections
\begin{align}
  \VCAT(\sA,\sB)(f,g)
  &\cong \VPROF(\sA,\sB)(\sB(1,f),\sB(1,g))\\
  &\cong \VPROF(\sB,\sA)(\sB(g,1),\sB(f,1)).
\end{align}
The dual statement for homs is the second Yoneda Lemma, as in
\ref{thm:yoneda2}.

\begin{lem}\label{thm:yoneda2large}
  For a \V-profunctor $H:\sA\hto \sB$ and a \V-functor $f:\sA'\to\sA$,
  we have $H\lhd \sA(f,1)\cong H(1,f)$. Similarly, for $g:\sB'\to\sB$,
  we have $\sB(1,g) \rhd H\cong H(g,1)$.
\end{lem}
\begin{proof}
  This follows immediately from~\cite[Theorem~7.20]{cs:multicats}.
\end{proof}

In particular, although in the statement of \autoref{thm:yoneda2} we
assumed \V to be an indexed cosmos so that the homs $\lhd$ and $\rhd$
would exist \emph{a priori}, this version of it shows that that
assumption is unnecessary; the particular homs in question
\emph{automatically} exist.

The first Yoneda Lemma \ref{thm:yoneda1} follows immediately.

\begin{lem}\label{thm:yoneda1large}
  For any \V-functor $f:\sA\to \sB$ and \V-profunctors $H:\sA\hto \sB$
 and $K:\sB\hto \sA$, there are natural bijections
 \begin{align}
   \VPROF(\sA,\sB)(\sB(1,f),H) &\cong
   \VPROF(\sA,\sA)(\sA,H(f,1)) \qquad\text{and}\\
   \VPROF(\sB,\sA)(\sB(f,1),K) &\cong
   \VPROF(\sA,\sA)(\sA,K(1,f)).
 \end{align}
\end{lem}
\begin{proof}
  For the first, we have
  \begin{align}
    \VPROF(\sA,\sB)(\sB(1,f),H)
    &\cong \VBimor(\sA,\sB(1,f);H)\\
    &\cong \VPROF(\sA,\sA)(\sA,H(f,1))
  \end{align}
  by \autoref{thm:large-unit-comp} and \autoref{thm:yoneda2large}.  The
  second is analogous.
\end{proof}

Finally, to generalize the remaining lemmas about profunctors from
\SS\ref{sec:small-cats} we require the following weakened notion of
adjunction, which does not require all composites to exist.

\begin{defn}\label{def:large-adj}
  An \textbf{adjunction} $H\dashv K$ between \V-profunctors
  $H:\sA\hto\sB$ and $K:\sB\hto\sA$ consists of a composite $H\odot K$
  together with multimorphisms $\eta:()\to H\odot K$ and $\ep:K,H \to
  \sB$ such that the composites
  \begin{gather}
    H \xto{\eta,1} H\odot K,H
    \xto{\hat{\ep}} H \mathrlap{\qquad\text{and}} \\
    K \xto{1,\eta} K,H\odot K
    \xto{\check{\ep}} K
  \end{gather}
  are identities.  Here $\hat{\ep}$ denotes the unique bimorphism such
  that
  \[H,K,H \to H\odot K, H \xto{\hat\ep} H
  \qquad\text{is equal to}\qquad
  H,K,H \xto{1,\ep} H,\sB \to H,
  \]
  and similarly for $\check{\ep}$.
\end{defn}

It is easy to show that such adjunctions have all the same formal
properties as ordinary adjunctions in a bicategory.  We can now
duplicate essentially the same proofs from \SS\ref{sec:small-cats} of
the following.

\begin{lem}
  For any \V-functor $f:\sA\to\sB$ there is an adjunction
  $\sA(1,f)\dashv\sA(f,1)$.
\end{lem}

\begin{lem}\label{thm:yoneda-adj-large}
  For \V-functors $f:\sA\to\sB$ and $g:\sB\to\sA$, there is a
  bijection between adjunctions $f\dashv g$ in \VCAT and isomorphisms
  $\sB(f,1) \iso \sA(1,g)$ of \V-profunctors.
\end{lem}

We also note, for future reference:

\begin{lem}\label{thm:compadj}
  Suppose $H\dashv K$ and $L\dashv M$ are adjunctions as in
  \autoref{def:large-adj}, and that the composites $H\odot L$ and
  $M\odot K$ and $(H\odot L) \odot (M\odot K)$ exist.  Then there is
  an adjunction $H\odot L\dashv M\odot K$.
\end{lem}

\section{\V-fibrations}
\label{sec:v-fibrations}

At first glance, indexed \sV-categories and large \sV-categories
appear very different, but it turns out that both are `loose enough'
notions that they are essentially equivalent.  A starting point for
this equivalence is to recall from \autoref{thm:three-homs} that the
fiberwise-homs, which make \sV\ into an indexed \sV-category, and the
external-homs, which make it into a large \sV-category, are related as
follows:
\begin{align*}
  \uV(x,y) &\iso \uV^{Y\times X}\big(\pi_X^*x,\pi_Y^*y\big)\\
  \uV^X(x,y) &\iso \Delta_X^*\uV(x,y)
\end{align*}
Thus, it is natural to try to extend these operations to compare large
and indexed \sV-categories.  In one direction this is straightforward:
given an indexed \sV-category \sA, we define a large \sV-category
$\Theta\sA$ whose objects are the objects of the fiber
categories $\sA^X$, and whose hom-objects are
\[\U{\Theta\sA}(x,y) =  \usA^{Y\times X}\big(\pi_X^*x,\pi_Y^*y\big)\]
(where $x\in\sA^X$, $y\in\sA^Y$).  It is easy to verify that this
does, in fact, give a large \sV-category.  Similarly, if $F\maps
\sA\to\sB$ is an indexed \sV-functor, we define $\Theta F$ to take
each object $x\in\sA^X$ to $Fx\in\sB^X$, with $\e (\Theta F)_x =
1_X$ and the obvious action on hom-objects.  We leave to the reader
the definition of \Theta\ on natural transformations and the
verification that it defines a 2-functor
\[\Theta\maps \iVCAT\to\VCAT.\]
from indexed \V-categories to large \V-categories.

Now \Theta\ is nothing like a 2-equivalence of 2-categories.  Any
large \sV-category of the form $\Theta\sA$ has lots of objects, and
the number of objects is clearly preserved by isomorphisms in \VCAT.
In particular, no \emph{small} \sV-category can be isomorphic to
anything in the image of \Theta.  Similarly, the induced functors
\[\Theta\maps \iVCAT(\sA,\sB)\to \VCAT(\Theta\sA,\Theta\sB)\]
are clearly not isomorphisms, since every functor of the form $\Theta
F$ preserves extents strictly (that is, the maps $\e(\Theta F)_x$ are
identities).

However, as we will prove shortly, \Theta is nevertheless a
\emph{biequivalence}.  We will approach this by trying to construct an
inverse to \Theta.  One obvious place to start, given a large
\sV-category \sB, is to try to define an indexed \sV-category
$\Lambda\sB$ as follows.  We take the objects of $\Lambda\sB^X$ to be
the objects of \sB\ of extent $X$, and set
\[\U{\Lambda\sB}^X(x,y) = \Delta_X^*\usB(x,y).\]
It is easy to check that this defines a $\sV^X$-enriched category
$\Lambda\sB^X$, but when we come to try to define the reindexing
functors $f^*$ we are stuck.  Just because $x$ is an object of
\sB\ with extent $Y$ and we have an arrow $f\maps X\to Y$ in \bS,
there need not be any object at all with extent $X$; this is
glaringly obvious when \sB\ is a small \sV-category.

% One way to see how to remedy this difficulty is the following.  Given
% an indexed \sV-category \sA, we can define an underlying ordinary
% indexed category $\sA_0$ (that is, a pseudofunctor $\bS\op\to\nCat$) by
% setting $(\sA_0)^X = (\sA^X)_0$, with the reindexing functors
% $f^*_0\maps \sA^Y_0\to\sA^X_0$ induced by those of \sA.  Similarly,
% given a large \sV-category \sB, we can define a category $\sB_\bbz$
% \emph{over} \bS\ whose objects are the objects of \sB, lying
% over their extents, and whose morphisms $x\to y$ over $f\maps \ex\to
% \ey$ are the morphisms
% \[\xymatrix{I_\ex \ar[r]\ar@{ |-> }[d] & \usA(x,y) \ar@{ |-> }[d]\\
%   \ex\ar[r]_<>(.5){(f,1)} & \ey\times \ex.}\]

% However, we know that in order to achieve an equivalence between
% pseudofunctors $\bS\op\to\nCat$ and categories over \bS, we need to
% require the categories over \bS\ to be \emph{fibrations}, which
% $\sB_\bbz$ is generally not.  Thus, it makes sense to look for
% fibrational conditions on large \sV-categories which would make them
% equivalent to indexed ones.

This should be regarded as similar to the problem we might encounter
when trying to define a inverse to the classical ``Grothendieck
construction'' which makes a pseudofunctor $\sA:\S\op\to\nCat$ into a
functor $\tot\sA \to \S$.  In that case, the answer is that we need
the input functor $\bA\to\S$ to be a \emph{fibration}; thus it makes
sense to look for ``fibrational'' conditions on large \V-categories.

If $x$ is an object of a large \V-category \sA, we will write $x$ also
for the \V-functor $\delta (\e x) \to \sA$ induced by $x$ and $1_{\e
  x}:\e x \to \e x$.  (Recall from \autoref{rmk:disc-large} that we
can make sense of this even if \V lacks indexed coproducts.)

\begin{defn}\label{def:restriction}
  Let \sA be a large \sV-category, $x$ an object of \sA, and $f\maps
  Y\to \ex$ a morphism in \bS.  A \textbf{restriction} of $x$ along
  $f$ is an object $f^*x$ of \sA such that $\e(f^*x) = Y$, together
  with an isomorphism between the \V-functors
  \begin{equation}
    \delta Y \xto{f^*x} \sA \qquad\text{and}\qquad
    \delta Y \xto{\delta f} \delta X \xto{x} \sA.
  \end{equation}
  % \begin{equation}
  %   \sA(1,f^*x) \iso \sA(1,x)(1,\delta f)\label{eq:restiso1}
  % \end{equation}
  % of \V-profunctors $\delta Y\hto \sA$.
\end{defn}

Of course, by the Yoneda lemma, this can equivalently be expressed by
isomorphisms of profunctors
\begin{align}
  \sA(1,f^*x) &\cong \sA(1,\delta f \circ x) \qquad\text{or}\\
  \sA(f^*x,1) &\cong \sA(\delta f\circ x,1)
\end{align}
If we note that $\sA(1,\delta f \circ x) \cong \sA(1,x)(1,\delta f)$
and similarly, and evaluate these profunctors at some $y\in \sA$, we
obtain isomorphisms
\begin{align}
  \usA(y,f^*x) &\cong (f\times 1)^* \usA(y,x) \qquad\text{and}\\
  \usA(f^*x,y) &\cong (1\times f)^* \usA(x,y).
\end{align}
In the case $\sV=\self(\bS)$, this idea is due
to~\cite{cplt-locintern}, where restrictions are called
\emph{substitutions}.

% Of course, $\sA(1,x)(1,\delta f) \cong \sA(1,\delta f \circ x)$, so by
% the Yoneda lemma, an isomorphism~\eqref{eq:restiso1} is equivalent to
% an isomorphism of \V-functors $f^* x \cong \delta f \circ x$.
% However,~\eqref{eq:restiso1} makes sense even if \V lacks indexed
% coproducts and $\delta (\e x)$ doesn't exist as a \V-category.  This
% observation does, however, make it clear that restrictions should be
% self-dual.

% \begin{prop}\label{thm:restriction-selfdual}
%   $f^*x$ is a restriction of $x$ along $f\maps Y\to \ex$ if and only
%   if we have an isomorphism
%   \begin{equation}
%     \sA(f^*x,1) \iso \sA(x,1)(\delta f,1)\label{eq:restiso2}
%   \end{equation}
%   of \V-profunctors $\sA\hto \delta Y$.
% \end{prop}
% \begin{proof}
%   By \autoref{thm:coyoneda}, we have
%   \begin{align}
%     \sA(1,x)(1,\delta f)
%     &\cong \delta(\ex)(1,\delta f) \odot \sA(1,x) \qquad\text{and}\\
%     \sA(x,1)(\delta f,1)
%     &\cong \sA(x,1)\odot \delta(\ex)(\delta f,1).
%   \end{align}
%   Thus by \autoref{thm:compadj}, each side of~\eqref{eq:restiso2} is
%   the right adjoint, in the sense of \autoref{def:large-adj}, of the
%   corresponding side of~\eqref{eq:restiso1}.  The result follows by
%   uniqueness of adjoints.
% \end{proof}

We can now characterize the large \sV-categories in the image of
\Theta.

\begin{defn}
  A \textbf{\sV-fibration} is a large \sV-category such that for each
  object $x$ and each $f\maps Y\to \ex$, there exists a
  restriction $f^*x$.
\end{defn}

\begin{rmk}
  The phrase ``\V-fibration'' is, of course, motivated by the remarks
  above comparing $\Theta$ to the Grothendieck construction.
  Moreover, just as ordinary fibrations replace the ``algebraic''
  reindexing functors of an ordinary indexed category by cartesian
  arrows with a universal property, \V-fibrations replace the
  reindexing functors of an indexed \V-category by ``restrictions'' as
  defined above, which are objects with a sort of universal property.
  In particular, it is no longer necessary to specify the coherence
  isomorphisms in \autoref{def:locally-vcat}; they follow
  automatically from the universal property.
  
  However, the analogy is just an analogy:
  there is no \S-indexed monoidal category \V such that large
  \V-categories can be identified with arbitrary functors into \S.
  (\autoref{eg:indpsh} might suggest that $\sPsh(\S,\nSet)$ should
  have this property, but it does not).  Large \V-categories contain
  more data than an arbitrary functor into \S, which as we will see is
  in fact sufficient to characterize a corresponding \V-fibration.
\end{rmk}

\begin{rmk}
  On the other hand, \V-fibrations share the virtue of ordinary
  fibrations that for fixed \V, they are an elementary (first-order)
  notion, as contrasted with indexed \V-categories and classical
  indexed categories which are not.
\end{rmk}

% By an appropriate version of the Yoneda lemma, a map
% \[\usB(-,f^*x) \to(f\times 1)^*\usB(-,x)\]
% is determined by a parametrized map
% \[\xymatrix{I_Y \ar[r]\ar@{ |-> }[d] & \usB(f^*x,x) \ar@{ |-> }[d]\\
%   Y \ar[r]_<>(.5){(f,1)} & \ex\times Y}\]
% which is exactly a map $f^*x \to x$ in $\sB_\bbz$.  (Note, though,
% that the dual characterization in \autoref{thm:restriction-selfdual}
% does \emph{not} correspond to an arrow $x\to f^*x$ in $\sB_\bbz$.)  We
% say that a map $f^*x \to x$ in $\sB_\bbz$ is \textbf{\sV-cartesian} if
% the corresponding map $\usB(-,f^*x) \to (f\times 1)^*\usB(-,x)$ is an
% isomorphism, or equivalently the map $\usB(-,f^*x) \to \usB(-,x)$ is
% cartesian in \bV\ over $(f\times 1)$.

% Thus, the assertion that \sB\ is a \sV-fibration can be rephrased by
% saying that all \sV-cartesian liftings exist.  Clearly any
% \sV-cartesian arrow is also cartesian, so $\sB_\bbz$ is an ordinary
% fibration whenever \sB\ is a \sV-fibration.  This gives us hope that
% the following theorem may be true.

\begin{prop}\label{thm:iso-theta-fib}
  A large \sV-category is isomorphic to one of the form $\Theta\sA$ if
  and only if it is a \sV-fibration.
\end{prop}
\begin{proof}
  If \sA is an indexed \V-category with transition functors
  $f^*:\sA^X\to\sA^Y$, then it is easy to check that for any
  $x\in\sA^X$, the object $f^*x\in\sA^Y$ is a restriction of $x$ along
  $f$ in $\Theta\sA$.  Thus $\Theta \sA$, and anything isomorphic to
  it, is a \sV-fibration.

  Conversely, given a \sV-fibration \sB, we complete the above
  construction of an indexed \sV-category $\Lambda\sB$ as follows.  We
  choose, for every $x$ and $f$, a restriction $f^*x$, and define the
  functor $f^*\maps (f^*)_\bullet\Lambda\sB^Y\to\Lambda\sB^X$ to take
  $x$ to $f^*x$.  The definition of restriction ensures that this can
  be extended to a fully faithful $\V^Y$-functor, and the essential
  uniqueness of restrictions ensures that they are coherent.  Finally,
  it is straightforward to check that $\Theta\Lambda\sB\iso\sB$ in
  \VCAT.
\end{proof}

As remarked above, in order to fully characterize the image of \Theta,
we will also need to limit the functors we consider.

\begin{defn}
  A \sV-functor $f:\sA\to\sB$ between large \sV-categories is called
  \textbf{indexed} if $\e f_x$ is an identity for all $x$.
\end{defn}

If we now let \VFib denote the sub-2-category of \VCAT consisting of
the \sV-fibrations, the indexed \sV-functors between them, and all the
\sV-natural transformations between those, then it is easy to extend
$\Lambda$ to a 2-functor $\VFib\to\iVCAT$.  Our double use of the word
`indexed' is unproblematic because of the following result.

\begin{thm}
  The 2-functors \Theta\ and \Lambda\ are inverse 2-equivalences
  between \iVCAT\ and \VFib.
\end{thm}
\begin{proof}
  Left to the reader.
\end{proof}

Since 2-equivalences preserve all 2-categorical structure, we can use
indexed \sV-categories and \sV-fibrations interchangeably, just as we
do for ordinary fibrations and pseudofunctors, and we will rarely
distinguish notationally between them.

\begin{rmk}\label{rmk:ivprof-eqv}
  This 2-equivalence also extends to profunctors.  We could define a
  virtual equipment of indexed \V-profunctors and show it is
  equivalent to the restriction of \Lprof{\V} to the \V-fibrations and
  indexed functors.  However, for our purposes it will suffice to note
  that for \V-fibrations \sA and \sB, we have an equivalence of
  categories
  \begin{equation}
    \iVPROF(\sA,\sB) \simeq \VPROF(\sA,\sB)
  \end{equation}
  connecting indexed \V-profunctors, as in \autoref{def:ivprof}, to
  \V-profunctors as considered in \SS\ref{sec:large-cats}.  This
  equivalence is constructed just as for the hom-objects of
  categories, by restricting along diagonals and projections.
\end{rmk}

% \begin{rmk}
%   It is perhaps worth mentioning, though, that just as in the case of
%   ordinary fibrations and pseudofunctors, passing from the
%   `nonalgebraic' or fibrational notion to the `algebraic' notion
%   involves making lots of choices---in fact, a large number of them
%   (in the technical sense).  Thus, we need a strong enough axiom
%   of choice for the equivalence to go through.
% \end{rmk}

In contrast to the classical case, however, it turns out that by
including the \emph{non-indexed} \sV-functors, we can put back in the
large \sV-categories that \emph{aren't} \sV-fibrations and still
maintain a biequivalence.  We first observe the following.

\begin{prop}\label{thm:fib-indexed-ok}
  If \sB\ is a \sV-fibration, then any \sV-functor $F\maps \sA\to\sB$
  is naturally isomorphic to an indexed one.
\end{prop}
\begin{proof}
  Given $F$, with components $\e F_x\maps \ex\to \e(Fx)$, we define
  $F'\maps \sA\to\sB$ by choosing $F'x$ to be a restriction $(\e
  F_x)^*(Fx)$ of $Fx$ along $\e F_x$.  It is easy to check that $F'$
  is an indexed \sV-functor, and that the isomorphisms $(\e F_x)^*(Fx)
  \cong (Fx) \circ \delta(\e F_x)$ from \autoref{def:restriction}
  assemble into a natural isomorphism $F\cong F'$.
\end{proof}

It follows that \VFib, while not a full sub-2-category of \VCAT, is a
`full sub-bicategory' in the sense that the inclusions
\[\VFib(\sA,\sB)\into \VCAT(\sA,\sB)\]
are equivalences of categories.  Thus, to prove that $\VFib\into\VCAT$
is a biequivalence, it suffices to check that every large \sV-category
is equivalent, in \VCAT, to a \sV-fibration.  This is included in the
following theorem.

\begin{thm}\label{thm:fibrepl}
  The (non-full) inclusion $\VFib\into\VCAT$ has a right 2-adjoint
  \Gamma; this means that for a \sV-fibration \sA\ and a large
  \sV-category \sB, we have natural isomorphisms of hom-categories
  \[\VFib(\sA,\Gamma\sB) \iso \VCAT(\sA,\sB).\]
  Moreover, the unit and counit $\sA\to \Gamma\sA$ and
  $\Gamma\sB\to\sB$ are internal equivalences, so this 2-adjunction is
  actually a biequivalence.
\end{thm}
\begin{proof}
  We define the objects of $\Gamma\sB$ to be `formal
  restrictions' $f^*x$, where $x$ is an object of \sB\ with
  extent $X$ and $f\maps Y\to X$ is a map in \bS.  Of course, we set
  $e(f^*x)=Y$, and we define $\U{\Gamma\sB}(f^*x,g^*y)$ to be
  $(g\times f)^*\usB(x,y)$.  We leave it to the reader to define the
  rest of the structure and check that $\Gamma\sB$ is a
  \sV-fibration.

  Now, an indexed \sV-functor $\sA\to\Gamma\sB$ sends each
  object $a$ of \sA\ to a formal restriction $f^*x$ in
  $\Gamma\sB$, where $f\maps \e a\to\ex$ is a map in \bS.  On the
  other hand, a non-indexed \sV-functor $\sA\to\sB$ sends each
  object $a$ to an object $x$ and chooses a map $f\maps \e a\to
  \ex$, so at this level the bijection is obvious.  It is easy to
  check that it carries over to the action on hom-objects and to
  natural transformations, so that \Gamma\ defines a right 2-adjoint
  to the inclusion.

  Now since the inclusion is \emph{bicategorically} fully faithful, it
  follows automatically that the unit $\sA\to\Gamma\sA$ is an
  \emph{equivalence} (though not an isomorphism).  Thus, it remains to
  check that the counit $\ep\maps \Gamma \sB\to\sB$ is an equivalence
  for any large \sV-category \sB.

  Of course, the counit $\ep\maps \Gamma\sB\to\sB$ sends $f^*x$ to
  $\ep(f^*x) = x$ with $\e (\ep_{f^*x}) = f$.  We define a \sV-functor
  $\xi\maps \sB\to\Gamma\sB$ by sending each object $x$ of \sB\ to its
  formal restriction $1_{\ex}^*x$.  Clearly $\ep\xi$ is the identity
  on \sB.  The composition $\xi\ep$ sends the formal restriction
  $f^*x$ to $1_{\ex}^*x$, with $\e(\xi\ep)_{f^*x} = f$.  It suffices
  to show that $\xi\ep\iso\Id_{\Gamma \sB}$ in \VCAT, which we can do
  by assembling the isomorphisms from \autoref{def:restriction}, as we
  did in \autoref{thm:fib-indexed-ok}.
  %
  % We define natural transformations $\alpha\maps \xi\ep\to\Id_{\Gamma
  %   \sB}$ and $\beta\maps\Id_{\Gamma \sB}\to \xi\ep$ as follows.
  % \alpha\ is defined by, for each $x$ and $f\maps Y\to \ex$, a
  % parametrized map
  % \[\xymatrix{\pi^* I \ar[rr]^<>(.5){\alpha}\ar@{ |-> }[d] &&
  %   \U{\Gamma\sB}(x,f^*x) \ar@{ |-> }[d]\\
  %   Y \ar[rr]_<>(.5){(1_Y,f)} &&
  %   Y\times \ex.}\]
  % By definition we have $\U{\Gamma\sB}(x,f^*x) = (f\times
  % 1)^*\usB(x,x)$, so we can define \alpha\ by restricting
  % $\pi_{\ex}^*I \to \usB(x,x)$ along $f$.  Similarly, \beta\ is
  % defined by maps
  % \[\xymatrix{\pi^* I \ar[rr]^<>(.5){\alpha}\ar@{ |-> }[d] &&
  %   \U{\Gamma\sB}(f^*x,x) \ar@{ |-> }[d]\\
  %   Y \ar[rr]_<>(.5){(f,1_Y)} &&
  %   \ex\times Y,}\]
  % and $\U{\Gamma\sB}(f^*x,x) = (1\times f)^*\usB(x,x)$, so a
  % similar argument applies.  It is easy to check that $\alpha\beta$
  % and $\beta\alpha$ are identities, so we have $\xi\ep\iso\Id_{\Gamma
  %   \sB}$ as desired.
\end{proof}

\begin{eg}
  If \bV is a classical monoidal category and \bC is a small
  \bV-enriched category, regarded as a small $\fam(\bV)$-category as
  in \autoref{eg:enriched}, then $\Gamma\bC$ is the
  $\fam(\bV)$-category $\fam(\bC)$ constructed in
  \autoref{eg:enriched-as-indexed}.
\end{eg}

\begin{eg}
  If $A$ is an \S-internal category regarded as a small
  $\self(\S)$-category, then $\Gamma A$ is the locally
  internal category classically associated to $A$.
\end{eg}

\begin{rmk}
  On the other hand, if we write $\VCAT_\mathrm{ind}$ for the
  sub-2-category of \VCAT\ containing all the \sV-categories but only
  the indexed \sV-functors, then the non-full inclusion
  $\VCAT_\mathrm{ind}\into\VCAT$ is \emph{not} a biequivalence.
%   Homotopically speaking, we can consider the \sV-fibrations to be the
%   `fibrant objects' in $\VCAT_\mathrm{ind}$; if we defined a suitable
%   notion of `weak equivalence' we could expect to recover \VCAT\ as
%   the bicategory of fractions of $\VCAT_\mathrm{ind}$.
%
%   On the other hand, 2-categorically speaking, we can say that the
%   functor \Gamma\ defines a \emph{pseudo-idempotent 2-monad}
%   (see~\cite{kl:property-like}) on $\VCAT_\mathrm{ind}$ whose algebras
%   are the \sV-fibrations.  This can be contrasted with the case of
%   ordinary fibrations over \bS, which are instead the algebras for an
%   \emph{oplax}-idempotent 2-monad on $\nCat/\bS$.
% \end{rmk}
% \begin{rmk}
  This is relevant because if we were to restrict ourselves to data
  contained in the \emph{bicategory} constructed from \V (rather than
  the whole equipment), then the indexed \sV-functors would be the
  only type of morphism available.  For this reason, the authors
  of~\cite{bcsw:variation-enr,cplt-locintern,chbase-locintern,desc-locintern}
  had to impose extra conditions at least as strong as being a
  \sV-fibration in order to obtain an equivalence with indexed
  \sV-categories (in the case $\sV=\self(\bS)$, which is the only one
  they considered).
\end{rmk}

Most \sV-categories which arise ``in nature'' are either small or are
\sV-fibrations.  We can regard the other large \V-categories as a
technical tool which makes it easier to relate these two most
important types.  (As we will see in \SS\ref{sec:limits-colimits},
set-sized \sV-categories are also convenient to use as diagram
shapes.)

% We will generally work with large \sV-categories, but when convenient
% we will feel free to assume that they are \sV-fibrations, if necessary
% by applying the functor \Gamma.  Thus, if \sA\ is a large
% \sV-category, it has both external-homs $\usA(x,y)$ and fiberwise-homs
% $\usA^X(x,y)$ in each fiber.

\begin{rmk}
  We can also define an anologue of the general hom-functors from the
  end of \SS\ref{sec:indexed-moncats} for any \sV-fibration \sA; we
  set
  \begin{align*}
    \usA^{Y,[W]}(B,C) &= \pi_{W*}\Delta_{Y\times W}^*\usA(B,C)\\
    &\iso \pi_{W*} \usA^{X\times Y\times Z\times W}(\pi_X^*B,
    \pi_Z^*C).
  \end{align*}
  When we consider tensors, cotensors, and monoidal structures for
  \sV-categories, we will also find analogues for \sV-fibrations of
  the various types of monoidal structure on \sV.
\end{rmk}

\section{Change of cosmos and underlying indexed categories}
\label{sec:cocuf}

If \V is an \S-indexed monoidal category and \sW is a \bT-indexed one,
then by a \textbf{lax monoidal morphism} $\Phi:\V\to\sW$ we mean a
commutative square
\begin{equation}
  \xymatrix{\tV \ar[r]^{\Phi}\ar[d] & \tW \ar[d]\\
    \bS\ar[r]_{\Phi} & \bT}\label{eq:mon-mor-fib}
\end{equation}
such that $\Phi:\bS\to\bT$ preserves finite products (hence is strong
cartesian monoidal), $\Phi:\tV\to\tW$ is lax monoidal and preserves
cartesian arrows, and the square commutes in the 2-category of lax
monoidal functors.  If $\Phi:\bS\to\bT$ is an identity, as is often
the case, we say that $\Phi$ is a morphism \textbf{over \S.}

In this situation, we have induced operations $\Phi_\bullet$ from
small, large, and indexed \V-categories to the corresponding sort of
\sW-categories, and similarly for functors, transformations,
profunctors, multimorphisms, and so on, which we call \textbf{change
  of cosmos}.  Formally, $\Phi_\bullet$ is a (normal, lax) equipment
functor $\Lprof{\V} \to \Lprof{\sW}$, which in particular induces
2-functors $\cat\V \to \cat\sW$, $\CAT\V \to \CAT\sW$, and so on.

If $\Phi$ is strong monoidal and preserves indexed coproducts,
fiberwise coequalizers, and fiberwise coproducts of the appropriate
cardinalities, then $\Phi_\bullet$ preserves composition of
profunctors.  Similarly, if $\Phi$ is closed monoidal and preserves
indexed products, fiberwise equalizers, and fiberwise products of the
appropriate cardinalities, then $\Phi_\bullet$ preserves right and
left homs of profunctors.

We can furthermore assemble the operations $\Lprof{(-)}$ and
$(-)_\bullet$ into a 2-functor from a 2-category of indexed monoidal
categories as in~\cite{shulman:frbi} into the 2-category
$\mathit{v}\cE\mathit{quip}$ of~\cite[7.6]{cs:multicats}.  This can be
decomposed into the 2-functor $\mathbb{F}\mathrm{r}$
of~\cite[14.9]{shulman:frbi} (suitably generalized to the virtual
case) followed by a many-object version of the 2-functor
$\mathbb{M}\mathsf{od}$ of~\cite[3.9]{cs:multicats}.  In particular,
any monoidal adjunction between indexed monoidal categories induces an
adjunction between 2-categories (or equipments) of enriched indexed
categories.

We omit the details of all of this since we will not need them here;
instead we merely mention some important special cases.

\begin{eg}
  Any lax monoidal functor $\bV\to\bW$ gives rise to a lax monoidal
  morphism $\fam(\bV) \to \fam(\bW)$.  When we identify \bV-enriched
  categories with certain indexed $\fam(\bV)$-categories as in
  \autoref{eg:enriched-as-indexed}, the induced operations from
  $\fam(\bV)$-categories to $\fam(\bW)$-categories agree with the
  classical change-of-enrichment operations.
\end{eg}

\begin{eg}
  Any pullback-preserving functor $F\maps \bS\to\bT$ gives rise to a
  strong monoidal morphism $\self(\bS)\to\self(\bT)$.  The induced
  operations on internal categories and locally internal categories
  agree with the obvious ones.
\end{eg}

\begin{eg}\label{thm:free-monfib}
  If $\V$ has indexed coproducts preserved by $\ten$, then there is a
  strong monoidal morphism $\Si:\self(\bS) \too \V$, which takes an
  object $A\xto{a} X$ of $\self(\bS)^X = \bS/X$ to the object
  $a_!\I_A$ of $\V^X$.  The pseudonaturality isomorphism $\Si(f^*A)
  \cong f^*(\Si A)$ is just the Beck-Chevalley condition for indexed
  coproducts in \V, together with the fact that $f^*\I_Y \cong \I_X$.
  And for a pullback square
  \begin{equation}
    \vcenter{\xymatrix{
        C\ar[r]^q\ar[d]_p \ar[dr]|c &
        B\ar[d]^b\\
        A\ar[r]_a &
        X
      }}
  \end{equation}
  in \S, regarded as the fiberwise product $C=A\times_X B$ in
  $\self(\S)^X$, the comparison isomorphism $\Si(A) \otimes_X \Si(B)
  \toiso \Si(C)$ is the composite
  \begin{alignat}{2}
    a_! \I_A \otimes_X b_! \I_B
    &\cong a_!(\I_A \otimes_A a^* b_! \I_B)
    &\qquad\text{(since $\ten$ preserves indexed coproducts)}\\
    &\cong a_!(\I_A \otimes_A p_! q^* \I_B)
    &\qquad\text{(by the Beck-Chevalley condition)}\\
    &\cong a_!p_!(p^*\I_A \otimes_{C} q^* \I_B)
    &\qquad\text{(since $\ten$ preserves indexed coproducts)}\\
    &\cong c_!(\I_{C} \otimes_{C} \I_{C})
    &\qquad\text{(since $p^*$ and $q^*$ are strong monoidal)}\\
    &\cong c_!(\I_{C}).
  \end{alignat}
  It is evident that $\Si$ also preserves indexed coproducts.

  We say that the induced functor $\CAT{\self(\S)} \to \CAT{\V}$
  builds the \textbf{free \V-category} on a $\self(\S)$-category, and
  write it as $\V[-]$.  As a special case, the discrete \V-category
  $\delta X$ is the free \V-category on the discrete \S-internal
  category on $X$.

  (This is the primary place in this paper where we use the fully
  general Beck-Chevalley condition for \V, as opposed to the limited
  version described in \autoref{eg:actions}.  Thus we cannot expect to
  build free $\sAct(\S)$-categories.)
\end{eg}

\begin{eg}\label{eg:lsfib}
  Assuming \S is locally small, the fiberwise Yoneda embedding gives a
  strong monoidal morphism $\self(\S) \to \sPsh(\S,\nSet)$.  Since
  this morphism is fiberwise fully faithful, the induced functor from
  $\self(\S)$-categories to $\sPsh(\S,\nSet)$-categories is
  2-fully-faithful.  Hence, we can regard $\self(\S)$-categories as
  $\sPsh(\S,\nSet)$-categories---that is, classical \S-indexed
  categories with locally small fibers---with the property that all
  their hom-presheaves are representable.

  When expressed in terms of the fiberwise-homs $\usA^X(x,y)$ of an
  indexed $\sPsh(\S,\nSet)$-category, this requires that for any
  $x,y\in\sA^X$ the functor
  \begin{equation}
    \begin{array}{rcl}
      \S/X &\too& \nSet\\
      (Z\xto{f} X) &\mapsto& \sA^Z(f^*x,f^*y)
    \end{array}
  \end{equation}
  is representable.  And when expressed equivalently in terms of the
  external-homs $\usA(x,y)$ of a large $\sPsh(\S,\nSet)$-category, it
  requires that for any $x\in \sA^X$ and $y\in\sA^Y$, the functor
  \begin{equation}
    \begin{array}{rcl}
      \S/(X\times Y) &\too& \nSet\\
      (Z \xto{(f,g)} X\times Y) &\mapsto& \sA^Z(f^*x,g^*y)
    \end{array}
  \end{equation}
  is representable.  But this latter condition is exactly the usual
  notion of when an indexed category is ``locally small''.  Thus, we
  recover the theorem indentifying locally small indexed categories
  with ``locally internal categories'' (which, recall, are the same as
  indexed $\self(\S)$-categories).
\end{eg}

\begin{eg}\label{eg:underlying}
  In fact, for \emph{any} \V with locally small fibers, there is a lax
  monoidal morphism $\V \to \sPsh(\S,\nSet)$, which takes $A\in\V^X$
  to the functor
  \begin{equation}
    \begin{array}{rcl}
      (\S/X)\op &\too& \nSet\\
      (Y\xto{f} X) &\mapsto& \V^X(\I_Y,f^* A).
    \end{array}\label{eq:underlying}
  \end{equation}
  In the case $\V=\self(\S)$ this reduces to the fiberwise Yoneda
  embedding.  In general, it implies that any indexed \V-category
  $\sA$ has an underlying ordinary \S-indexed category, which we
  denote $\ord\sA$.  We call $\ord\sA$ the \textbf{underlying
    indexed category} of \sA; it should be regarded as an indexed
  version of the classical ``underlying ordinary category of an
  enriched category''.  In particular, if \V is closed, then by
  applying this construction to the \V-category \V we recover the
  original \S-indexed category \V.

  Tracing through the identification of indexed
  $\sPsh(\S,\nSet)$-categories with classical indexed categories, we
  see that the fiber category $(\ord\sA)^X$ over $X$ is the underlying
  ordinary category of the classical $\V^X$-enriched category $\sA^X$,
  generally denoted $\ord{(\sA^X)}$.  Thus, there is no ambiguity in
  writing $\ord{\sA^X}$ for this fiber.

  If \sA is a (large or small) \V-category that is not a \V-fibration,
  then it still has an underlying \S-indexed category, namely
  $\ord{(\Gamma \sA)}$ where $\Gamma$ is the functor from
  \autoref{thm:fibrepl}.  It is easy to see that this is precisely the
  underlying indexed category constructed in \autoref{eg:uf-via-disc}
  and \autoref{rmk:disc-large}.  However, if \sA \emph{is} a
  \V-fibration, then $\ord{(\Gamma \sA)}$ is rather larger than
  $\ord\sA$ (though still equivalent to it).
\end{eg}

\begin{eg}\label{eg:comprehension}
  If \V has the property that each functor~\eqref{eq:underlying} is
  representable, then the morphism $\V \to \sPsh(\S,\nSet)$ factors
  through $\self(\S)$, and so $\ord\sA$ is a $\self(\S)$-category for
  any \V-category \sA.  When \V is cartesian, this property of \V is
  precisely the \textbf{comprehension schema}
  of~\cite{lawvere:comprehension}, so we will henceforth extend that
  terminology to the non-cartesian case.

  In particular, if \V is closed, then the \V-category \V itself has
  an underlying $\self(\S)$-category.  It is easy to see that this
  means the \S-indexed category \V is locally small.  Indeed, for a
  closed \V, local smallness is equivalent to the comprehension
  schema.

  If \V satisfies the comprehension schema and also has indexed
  coproducts preserved by $\ten$, then the ``underlying'' morphism
  $\V\to\self(\S)$ is right adjoint to the ``free'' morphism
  $\self(\S)\to\V$ from \autoref{thm:free-monfib}.  Thus, the induced
  functors on enriched indexed categories are likewise adjoint.

  For instance, when $\V=\self_*(\S)$, we see that every pointed
  internal category (as in \autoref{eg:internal-and-pointed}) has an
  underlying ordinary internal category, and that this operation has a
  left adjoint which ``adjoins a disjoint section'' to the
  hom-objects.  Similarly, if \S satisfies the hypotheses of
  \autoref{eg:ab-mf}, then any $\sAb(\bS)$-category has an underlying
  $\self(\S)$-category, and this operation has a left adjoint which
  builds free abelian group objects on the hom-objects.
\end{eg}

\begin{eg}\label{eg:discrete-objects-adjn}
  Here is an example in which the base category changes.  Let \bS\
  have finite limits; then there is a lax monoidal morphism
  $\self(\bS)\to\fam(\bS)$, which takes $X\in\S$ to the set $\bS(1,X)$
  and an object $A\in\bS/X$ to its family of fibers.  Thus, any
  \bS-internal category gives rise to a small \bS-enriched category.
  For example, any internal topological category gives rise to a
  topologically enriched category by forgetting the topology on the
  set of objects.

  If \bS has small coproducts preserved by pullback, then this
  morphism has a strong monoidal left adjoint, which sends a set $X$
  to $\coprod_{x\in X} 1$ and an $X$-indexed family $\{A_x\}$ of
  objects of \bS to $\coprod_{x\in X} A_x$.  The corresponding
  operation on small \bS-enriched categories regards them as
  \bS-internal categories whose object-of-objects is ``discrete''
  (i.e.\ a coproduct of copies of the terminal object).
\end{eg}

\section{Limits and colimits}
\label{sec:limits-colimits}

We now begin studying limits and colimits for \V-categories.  Here is
where the equipment-theoretic machinery of profunctors is most
helpful, because it automatically gives us a general definition of
weighted limit with many good formal properties.  In this section we
recall this definition and some of these good properties; in the next
section we translate some examples into more concrete terms for
indexed \V-categories.

\begin{defn}
  Let $J\maps K\hto A$ be a \sV-profunctor and $f\maps A\to \sC$ a
  \sV-functor.  A \textbf{$J$-weighted colimit of $f$} consists of a
  \sV-functor $\ell\maps K\to \sC$ together with an isomorphism
  \begin{equation}
    \sC(\ell,1) \cong J\rhd \sC(f,1)\label{eq:colimiso}
  \end{equation}
  of profunctors $\sC \hto K$.  (Recall \autoref{rmk:comphom-abuse}.)

  If instead $J\maps A\hto K$, then a \textbf{$J$-weighted limit of
    $f$} consists of a \sV-functor $\ell\maps K\to X$ together with
  an isomorphism
  \begin{equation}
    \sC(1,\ell) \cong \sC(1,f) \lhd J.\label{eq:limiso}
  \end{equation}
  of profunctors $K\hto \sC$.
\end{defn}

In general, $K$, $A$, and \sC could be any large \V-categories.
However, in our examples, most often $K$ and $A$ will be small or
set-small, while \sC will be a \V-fibration (hence our choice of
typefaces).

% Of course, weighted limits and colimits are unique up to unique
% isomorphism when they exist.

We will consider many examples in \SS\ref{sec:indexed-limits}, but we
should at least mention the following one here, to clarify why this is
a reasonable definition of ``weighted limit''.

\begin{eg}
  If $\V=\fam(\bV)$ and we take $K$ to be the unit \sV-category
  $\delta 1$ and $A$ a small \V-category, then $J$ is simply a diagram
  on $A$.  If \sC is the indexed \V-category constructed from a
  (possibly large) \bV-enriched category \bC as in
  \autoref{eg:enriched-as-indexed}, then the above definitions reduce
  to the usual notion of weighted limit and colimit.  Specifically, if
  we assume $\ell$ to be an indexed \V-functor, then it is just an
  object of $\bC$, and (in the colimit case) the
  isomorphism~\eqref{eq:colimiso} means that
  \[\bC(\ell,x) \iso [A\op,\bC](J,\,\bC(f-,x))
  \]
  for all $x\in\bC$, which is the usual definition of $\ell$ being a
  $J$-weighted colimit of $f$.
  % Dually, the limit isomorphism $\sC(1,\ell) \cong
  % \sC(1,f) \lhd J$ in this case means we have
  % \[\sC(x,\ell) \iso [A,\bC](J,\sC(x,f-))\]
  % which is the usual definition of weighted limit.
\end{eg}

By contrast with the above classical situation, in general it turns
out to be very useful to allow $K$ to be an arbitrary \V-category.
Here is one example of what such more general ``limits'' include.

\begin{eg}\label{eg:limcomp}
  Let $j\maps K\to A$ and take $J = A(1,j)$.  Then for
  any $f\maps A\to \sC$, we have
  \[ A(1,j) \rhd \sC(f,1) \cong \sC(f j,1) \] by
  \autoref{thm:yoneda2large}.  Hence $fj$ is always a
  $A(1,j)$-weighted colimit of $f$.  Dually, $fj$ is also always an
  $A(j,1)$-weighted limit of $f$.

  In particular, if $j$ is the identity functor of $A$, then
  $A(1,1)=A$ is the identity profunctor of $A$, and $f$ is its own
  $A$-weighted (co)limit.
\end{eg}

One real advantage of allowing arbitrary profunctors as weights is
that a given profunctor can be used both as a weight for limits and a
weight for colimits.  This symmetry is necessary in order to even
\emph{state} the following fact, which will be very useful.

\begin{prop}\label{thm:limcolim-adj}
  For a \V-profunctor $J\maps A\hto B$ and \V-functors $f\maps B\to
  \sC$ and $g\maps A\to \sC$ which have respectively a $J$-weighted
  colimit
  \(A\xto{\colim^J f} \sC\)
  and a $J$-weighted limit
  \(B\xto{\lim^Jg} \sC\), we have
  a natural isomorphism
  \[\VCAT(\colim^J f, g) \cong \VCAT(f, \textstyle\lim^J g). \]
\end{prop}
\proof%\begin{proof}
  We calculate
  \begin{align*}
    \VCAT(\colim^J f, g)
    &\iso \VPROF(\sC(g,1), \sC(\colim^J f,1))\\
    &\iso \VBimor(\sC(g,1),J; \sC(f,1))\\
    &\iso \VBimor(J, \sC(1,f); \sC(1,g))\\
    &\iso \VPROF(\sC(1,f), \sC(1,{\textstyle\lim^J g}))\\
    &\iso \VCAT(f, \textstyle\lim^J g).
    \tag*{\endproofbox}
  \end{align*}
%\end{proof}

We immediately obtain a description of Kan extensions as particular 
weighted limits and colimits.

\begin{cor}\label{thm:pointwise-kan}
  Let $j\maps A\to K$ and $f\maps A\to X$.  Then any $K(j,1)$-weighted
  colimit of $f$ is an internal left extension of $f$ in \VCAT, and
  any $K(1,j)$-weighted limit of $f$ is an internal right extension of
  $f$ in \VCAT.
\end{cor}
\begin{proof}
  Internal left extension is defined to be a (partial) left adjoint to
  composition, and we observed in \autoref{eg:limcomp} that
  $\uK(j,1)$-weighted limits are composites with $j$.
\end{proof}

In general, not all 2-categorical left extensions have the stronger
universal property of a $K(1,j)$-weighted colimit.  Sometimes
extensions with this additional property are called \emph{pointwise}
(see~\cite[\SS X.5]{maclane}), but we
follow~\cite[Ch.~4]{kelly:enriched} in reserving the simple term
\emph{Kan extension} for the pointwise ones.

\begin{defn}
  A \textbf{left Kan extension} of $f\maps A\to \sC$ along $j\maps A\to
  K$ is a $K(j,1)$-weighted colimit of $f$.  Dually, a \textbf{right
    Kan extension} of $f$ along $j$ is a $K(1,j)$-weighted limit of
  $f$.
\end{defn}

The following theorem, which will also be very useful, also requires
the use of arbitrary profunctors as weights.

\begin{thm}\label{thm:composing-limits}
  Let $J_2\maps A\hto B$ and $J_1\maps B\hto C$ be weights and let
  $f\maps C\to \sC$ be a \sV-functor.  Suppose that $\ell_1$ is a
  $J_1$-weighted colimit of $f$ and $\ell_2$ is a $J_2$-weighted
  colimit of $\ell_1$, and that the composite $J_2\odot J_1$ exists.
  Then $\ell_2$ is a $(J_2\odot J_1)$-weighted colimit of $f$.
\end{thm}
\proof%\begin{proof}
  For any $\vec H = H_1,\dots,H_n$, we have
  \begin{align}
    \Vmmor(\vec{H},J_2\odot J_1;\sC(f,1))
    &\cong \Vmmor(\vec{H},J_2,J_1;\sC(f,1))\\
    &\cong \Vmmor(\vec{H},J_2;\sC(\ell_1,1))\\
    &\cong \Vmmor(\vec{H};\sC(\ell_2,1)).
    \tag*{\endproofbox}
  \end{align}
%\end{proof}

% This needs monoidal structures.
% \begin{cor}[Interchange of colimits]
%   If $J_1\maps K_1\hto A_1$ and $J_2\maps K_2\hto A_2$ are weights
%   and $f\maps A_1\ten A_2\to X$ is a \sV-functor, then we have
%   \[\colim^{J_1} \colim^{J_2} f \iso \colim^{J_2} \colim^{J_1} f.\]
% \end{cor}
% \begin{proof}
%   Of course, by $\colim^{J_2} f$ we mean $\colim^{U_{A_1}\ten J_2} f$,
%   and similarly.  The result then follows from
%   \autoref{thm:composing-limits} and the observation that
%   \[(J_1\ten U_{K_2})\odot (U_{A_1}\ten J_2) \iso
%   (U_{K_1}\ten J_2)\odot (U_1\ten U_{A_2})
%   \]
% \end{proof}

Here is one example of the usefulness of
\autoref{thm:composing-limits}.  We say that a \V-functor
$f:\sA\to\sB$ is \textbf{fully faithful} if each morphism $\sA(x,y)
\to \sB(f(x),f(y))$ is cartesian over $\e f_x \times \e f_y$.  It is
easy to see that this is equivalent to the induced morphism $\sA \to
\sB(f,f)$ of profunctors $\sA\hto\sA$ being an isomorphism, or
equivalently that $\sB(1,f) \odot \sB(f,1) \cong \sA$.

\begin{cor}\label{thm:ff-extn-are-honest}
  Left and right Kan extensions along fully faithful \sV-functors are
  honest extensions.  In other words, if $j\maps A\to K$ is fully
  faithful and $\ell\maps K\to \sC$ is a left Kan extension of $f\maps
  A\to \sC$ along $j$, then $\ell j\iso f$.
\end{cor}
\begin{proof}
  Left Kan extensions are $K(j,1)$-weighted colimits, while
  precomposition with $j$ is a $K(1,j)$-weighted colimit.  Thus, by
  \autoref{thm:composing-limits}, $\ell j$ is a $(K(1,j)\odot
  K(j,1))$-weighted colimit.  However, since $j$ is fully faithful,
  $(K(1,j)\odot K(j,1))$ is isomorphic to the identity profunctor, for
  which a weighted colimit of $f$ is just $f$ itself.  The case of
  right Kan extensions is dual.
\end{proof}

% We will see yet more applications of \autoref{thm:composing-limits} in the sequel.

We now consider what it means for a \V-functor to preserve or reflect
limits.  Let $J\maps K\hto A$ be a weight and $d\maps A\to \sC$ a
\sV-functor, and suppose given a bimorphism
\begin{equation}
  \sC(\ell,1),J\xto{\psi} \sC(d,1)\label{eq:to-be-pres-lim}
\end{equation}
Let $f\maps \sC \to \sD$ be a \sV-functor, and consider the unique morphism
\begin{equation}
  \sC(f\ell,1),J \too \sC(fd,1)\label{eq:pres-lim-2cell}
\end{equation}  
whose composite with the universal bimorphism $\sC(f,1),\sC(\ell,1)
\xto{\phi}\sC(f\ell,1)$ is
\begin{align*}
  \sC(f,1), \sC(\ell,1), J \xto{1,\psi}
  \sC(f,1), \sC(d,1) \xto{\phi}
  \sC(fd,1).
\end{align*}

\begin{defn}
  In the above situation, if~\eqref{eq:to-be-pres-lim} exhibits $\ell$
  as a $J$-weighted colimit of $d$, we say that $f$ \textbf{preserves}
  this colimit if~\eqref{eq:pres-lim-2cell} exhibits $f\ell$ as a
  $J$-weighted colimit of $fd$.
  Similarly, if~\eqref{eq:pres-lim-2cell} exhibits $f\ell$ as a
  $J$-weighted colimit of $fd$, we say that $f$ \textbf{reflects} this
  colimit if~\eqref{eq:to-be-pres-lim} exhibits $\ell$ as a
  $J$-weighted colimit of $d$.
\end{defn}

Dually, we define what it means for a \sV-functor to preserve and
reflect a weighted limit.  The following observations are expected.

\begin{prop}
  If $f\maps \sC\to \sD$ is a left adjoint, then $f$ preserves
  any colimits which exist in $\sC$.  Dually, right adjoints
  preserve all limits.
\end{prop}
\begin{proof}
  Recall that an adjunction $f\adj g$ implies an isomorphism $\sD(f,1)
  \iso \sC(1,g)$.  Therefore, if $\ell\maps K\to \sC$ is a colimit of
  $d\maps A\to \sC$ weighted by $J\maps K\hto A$, then for any
  well-typed $\vec H = H_1,\dots,H_n$, we have
  \begin{align*}
    \Vmmor(\vec H, J; \sD({fd},1))
    &\iso \Vmmor(\vec H, J, \sD(1,d); \sD(f,1))\\
    &\iso \Vmmor(\vec H, J, \sD(1,d); \sC(1,g))\\
    &\iso \Vmmor(\sC(g,1),\vec H, J; \sD(d,1))\\
    &\iso \Vmmor(\sC(g,1),\vec H; \sD(\ell,1))\\
    &\iso \Vmmor(\vec H,\sD(1,\ell);\sC(1,g))\\
    &\iso \Vmmor(\vec H,\sD(1,\ell);\sD(f,1))\\
    &\iso \VPROF(M, \sD(f\ell,1)).
  \end{align*}
  The case of right adjoints is dual.
\end{proof}

\begin{prop}
  A fully faithful \sV-functor reflects all limits and colimits.
\end{prop}
\begin{proof}
  Let $f\maps \sC\to \sD$ be fully faithful and let $f\ell$ be a
  $J$-weighted colimit of $fd$.  Suppose we have a multimorphism $\vec
  H, J\to \sC(d,1)$; we want to show that it factors uniquely through
  a multimorphism $\vec H \to \sC(\ell,1)$.  We can compose on the
  left with $\sD(f,1)$ to obtain a multimorphism
  \[\sD(f,1) , \vec H , J \too \sD(f,1)\odot \sC(d,1) \iso \sD({fd},1),\] and
  since $\sD({f\ell},1) \iso J\rhd \sD({fd},1)$, this factors uniquely
  through $\sD({f\ell},1)$ via a multimorphism $\sD(f,1), \vec H\to
  \sD({f\ell},1)$.  Now composing on the left with $\sD(1,f)$ gives a
  multimorphism
  \[\sD(1,f)\odot \sD(f,1) , \vec H \too
  \sD(1,f) \odot \sD(f,1) \odot \sC(\ell,1).
  \]
  But since $f$ is fully faithful, we have $\sD(1,f)\odot \sD(f,1)\iso
  \sC$, so this is equivalent to a multimorphism $\vec H\to
  \sC(\ell,1)$.  We leave it to the reader to verify that this is the
  desired factorization, and that it is unique.  The case of limits is
  analogous.
\end{proof}

If $i\maps \sD\to \sC$ has a left adjoint $r\maps \sC\to \sD$ whose counit
$\ep\maps ri\to 1$ is an isomorphism, we say that $i$ exhibits $\sD$ as
a \textbf{reflective sub-\sV-category} of $\sC$.

\begin{prop}
  If $\sD$ is a reflective sub-\sV-category of $\sC$, then $\sD$ admits all limits
  and colimits which $\sC$ does.
\end{prop}
\begin{proof}
  For the case of colimits, if $J\maps K\hto A$ is a \V-profunctor and
  $f\maps A\to \sD$ a \V-functor, we can consider the composite $if$.
  If $\sC$ admits the $J$-weighted colimit of $if$, then since $r$ is
  a left adjoint, it preserves this colimit; thus $r(\colim^J if)$ is
  a $J$-weighted colimit of $rif\iso f$.

  For the case of limits, let $J\maps K\hto A$ again be a weight and
  $g\maps K\to \sD$ a \V-functor, and suppose that $\ell\maps A\to
  \sC$ is a $J$-weighted limit of $ig$; thus we have
  \begin{align*}
    \sC(1,\ell) &\iso \sC(1,{ig}) \lhd J \\
    &\iso (\sD(1,g) \odot \sC(1,i)) \lhd J\\
    &\iso (\sD(1,g) \odot \sD(r,1)) \lhd J\\
    &\iso (\sD(1,r) \rhd \sD(1,g)) \lhd J\\
    &\iso \sD(1,r) \rhd (\sD(1,g) \lhd J)\\
    &\iso (\sD(1,g) \lhd J) \odot \sD(r,1).
  \end{align*}
  It follows that
  \begin{align*}
    \sD(1,{r\ell}) &\iso \sC(1,\ell) \odot \sD(1,r)\\
    &\iso (\sD(1,g) \lhd J) \odot \sD(r,1) \odot \sD(1,r)\\
    &\iso (\sD(1,g) \lhd J) \odot \sC(1,i) \odot \sD(1,r)\\
    &\iso \sD(1,g) \lhd J
  \end{align*}
  since $ri\iso 1$.  Thus $r\ell$ is a $J$-weighted limit of $g$.
  Note that since $i$ is a right adjoint, it preserves all limits, so
  we can say more strongly that $\sD$ is `closed under limits' in
  $\sC$.
\end{proof}

\section{Limits in indexed \V-categories}
\label{sec:indexed-limits}

All the definitions and theorems in \SS\ref{sec:limits-colimits} make
sense in any (virtual) equipment.  Now, however, we truly specialize
to the case of enriched indexed categories, considering several
special types of limits and colimits and their relationship to more
familiar ones.  We will see that many such limits and colimits can be
described as `fiberwise' limits together with conditions ensuring the
limits are (1) compatible with the enrichment and (2) preserved by
restriction.  For classical indexed categories, the conditions (1)
tend to be automatic and the conditions (2) are significant, while for
classical enriched categories, the situation is reversed.  For general
\sV, both conditions will be nontrivial.

As mentioned in the introduction, in the case $\V=\self(\S)$ this
reduction of the abstract limit-notions to familiar indexed ones is
due to~\cite{cplt-locintern}; while the combination of indexed and
enriched universal properties for a general \V was first considered
(in an \emph{ad hoc} manner) in~\cite{gg:fib-rel}.

For simplicity, in this section we generally assume \V to be an
\S-indexed cosmos.  Most of the results could be rephrased with some
care under weaker assumptions on \V (in particular, symmetry is never
really necessary), but we mostly leave this to the interested reader.

We will also assume that the \V-category \sC in which we consider
limits is a \V-fibration.  In particular, by
\autoref{thm:fib-indexed-ok} this implies that up to isomorphism, we
may as well assume that any \V-functor with codomain \sC is indexed,
and we will generally do so.

First, suppose that $K=\delta X$ and $A=\delta Y$ are small discrete
\sV-categories.  In this case, a weight $J\maps \delta X\hto \delta Y$
is simply an object of $\V^{X\times Y}$.  Since (indexed) \sV-functors
$\delta X\to\sC$ are equivalent to objects of the fiber $\sC^X$,
$J$-weighted colimits take the fiber $\sC^X$ to the fiber $\sC^Y$, and
$J$-weighted limits take $\sC^Y$ to $\sC^X$.

% In particular, if $X=Y$, they take $\sC^X$ to itself.

For such a $J\in\sV^{X\times Y}$, we call a $J$-weighted colimit of
$x\in\sC^X$ a \textbf{global \sV-tensor} of $x$ with $J$, and write it
as $J\oast_{[X]} x$.  (We use the word `global' to distinguish these
tensors from the `fiberwise' ones that we will consider later.)
Invoking the definition of colimits and \autoref{rmk:disc-homs}, we
find that $J\oast_{[X]} x$ is an object of $\sC^Y$ characterized by an
isomorphism of profunctors $\sC\hto \delta Y$:
\[\usC\big(J\oast_{[X]} x, 1\big) \iso \usV^{[X]}\big(J,\usC(x,1)\big).\]

Dually, we call a $J$-weighted limit of $y\in\sC^Y$ a \textbf{global
  \sV-cotensor} with $J$ and write it as $\cten{J}{y}{[Y]}$.  By
definition and \autoref{rmk:disc-homs}, it is an object of $\sC^X$
characterized by an isomorphism of profunctors $\delta X\hto \sC$:
\[\usC\big(1, \cten{J}{y}{[Y]}\big) \iso \usV^{[Y]}\big(J,\usC(1,y)\big).\]

\begin{eg}\label{eg:enr-global-tensors}
  Suppose $\V=\fam(\bV)$ and $\sC=\fam(\bC)$ as in
  \autoref{eg:enriched-as-indexed}.  Then when $X=Y=1$, global
  \V-tensors are the same as classical tensors in enriched category
  theory.

  For general $X$ and $Y$, global tensors combine classical tensors
  with coproducts.  Namely, if $J = (J_{y,x})_{(y,x)\in Y\times X}$ is
  an object of $\bC^{Y\times X}$ and $M= (M_x)_{x\in X}$ is an object
  of $\sD^X$ for a \bC-enriched category \sD, then we have
  \[(J \oast_{[X]} M)_y \iso \coprod_{x\in X} J_{y,x} \oast M_x.
  \]
  In particular, if $Y=X$ and $J_{y,x} =\emptyset$ is an initial
  object for $y\neq x$ (which is to say that $J=\Delta_! J'$ for some
  $J'\in\sD^X$), then $(J \oast_{[X]} M)_x = J_{x,x}\oast M_x$
  involves no coproducts.  This is also the tensor of $M$ with $J'$
  in the $\bV^X$-enriched category $\bC^X$.

  On the other hand, if we have a function $g\maps X\to Y$ and we
  define
  \[J_{y,x} =
  \begin{cases}
    \I & \text{if } f(x)=y\\
    \emptyset & \text{otherwise}
  \end{cases},
  \]
  then $(J \oast_{[X]} M)_y = \coprod_{f(x)=y} M_x$ involves
  \emph{only} coproducts.
\end{eg}

This example suggests that for general \sV\ we may also profitably
split up the study of global \sV-tensors into those of the form
$\Delta_!J'$ and those induced by morphisms in \bS.  We start by
considering the latter.

Suppose that $f\maps X\to Y$ is a morphism in \bS.  Then it gives rise
to profunctors $Y(1,f): \delta X\hto \delta Y$ and $Y(f,1) :\delta
Y\hto \delta X$.  Explicitly, we have
\begin{align}
  \U{Y(1,f)} &= (f\times 1)^*(\Delta_Y)_!\I_Y
  \mathrlap{\qquad\text{and}}\\
  \U{Y(f,1)} &= (1\times f)^*(\Delta_Y)_!\I_Y.
\end{align}
By \autoref{thm:coyoneda} and \autoref{thm:yoneda2large}, for $M \in
\VProf(\delta Y,\delta Z) = \V^{Y\times Z}$, we have natural
isomorphisms
\[ Y(1,f) \odot M \;\cong\; (f\times 1)^* M
\;\cong\; M \lhd Y(f,1).
\]
Similarly, for $N \in \VProf(\delta X,\delta Z) = \V^{X\times Z}$ we
have
\begin{align}
  Y(f,1) \odot N &\cong (f\times 1)_! N\\
  N \lhd Y(1,f) &\cong (f\times 1)_* N.\label{eq:lhd-coext}
\end{align}
and symmetrically in all cases.

Now by \autoref{eg:limcomp}, for any $y\in\sC^Y$, the
$Y(1,f)$-weighted colimit and the $Y(f,1)$-weighted limit of the
corresponding indexed \V-functor $y:\delta Y\to \sC$ are just the
composite $f y : \delta X \to \sC$.  (If we demand the (co)limit to be
an \emph{indexed} \V-functor, then it must be precisely a
\emph{restriction} of $y$ along $f$ as in \autoref{def:restriction}.)

% \begin{rmk}\label{rmk:map-tensors}
%   Consider briefly the virtual equipment $\VDIST_\mathrm{ind}$
%   where the vertical arrows consist only of the \emph{indexed}
%   \sV-functors.  In this case, $f\maps Y\to X$ does not define a map
%   in $\VDIST_\mathrm{ind}$, and so precompositions with $f$ (that is,
%   $X_f$-weighted limits or ${}_f X$-weighted colimits) do not
%   necessarily exist.  It is not difficult to show that a
%   precomposition with $f$ is the same as a \emph{restriction} along
%   $f$ in the sense of \SS\ref{sec:large-vs-locally}.  Thus, another
%   definition of a \sV-fibration is one which admits these particular
%   limits in $\VDist_\mathrm{ind}$.  This is the approach taken (in the
%   case $\sV=\sA(\bS)$)
%   in~\cite{cplt-locintern,chbase-locintern,desc-locintern}.
% \end{rmk}

\begin{thm}\label{thm:vbc-cbc}
  If \sC is an indexed \V-category admitting $Y(f,1)$-weighted
  colimits for all $f: X\to Y$ in \S, then the \S-indexed category
  $\ord\sC$ has \S-indexed coproducts.  Dually, if \sC\ admits
  $Y(1,f)$-weighted limits for all $f\maps X\to Y$, then $\ord\sC$ has
  \S-indexed products.
\end{thm}
\begin{proof}
  By \autoref{thm:limcolim-adj} and the above observation,
  $Y(f,1)$-weighted colimits define a left adjoint $f_!\maps
  \ord{\sC^X}\to \ord{\sC^Y}$ to the restriction functor $f^*$, and
  dually.  So it remains only to check the Beck-Chevalley condition.
  Thus, suppose that
  \[\xymatrix{W \pullback \ar[r]^f\ar[d]_g & X \ar[d]^k\\
    Y \ar[r]_h & Z}\]
  is a pullback square in \bS; the question is whether the canonical
  transformation
  \[\colim^{Y(g,1)} \colim^{X(1,f)} \too
  \colim^{Z(1,h)} \colim^{Z(k,1)}
  \]
  is an isomorphism.  By \autoref{thm:composing-limits}, this can be
  reduced to the question of whether the canonical transformation
  \begin{equation}
    Y(g,1)\odot X(1,f) \too Z(1,h)\odot Z(k,1)\label{eq:cbciso}
  \end{equation}
  is an isomorphism.  However, by the above remarks, the functors
  \begin{gather}
    Y(g,1)\odot X(1,f) \odot - \mathrlap{\qquad\text{and}}\\
    Z(1,h)\odot Z(k,1) \odot -
  \end{gather}
  are naturally isomorphic to $(g\times 1)_!(f\times 1)^*$ and
  $(h\times 1)^*(k\times 1)_!$, respectively, and under these
  isomorphisms~\eqref{eq:cbciso} is identified with the Beck-Chevalley
  morphism in \sV.  Since \V has indexed coproducts, this
  transformation is an isomorphism; hence by the bicategorical Yoneda
  lemma for \VProf, so is~\eqref{eq:cbciso}.  The case of indexed
  products is dual.
\end{proof}

Of course, if \V only satisfies the Beck-Chevalley condition for some
pullback squares in \S, as in \autoref{eg:actions}, then $\ord\sC$
only inherits the Beck-Chevalley condition for those same pullback
squares.

\begin{defn}
  We say that \sC\ has \textbf{\S-indexed \sV-coproducts} if it admits
  $Y(f,1)$-weighted colimits for all $f\maps X\to Y$ in \S.
  Similarly, we say it has \textbf{\S-indexed \sV-products} if it
  admits all $Y(1,f)$-weighted limits.
\end{defn}

Just as a limit in the underlying category of a classical enriched
category need not be an enriched limit, it is not necessarily true
that indexed (co)products in $\ord\sC$ imply indexed \V-(co)products
in \sC.  We need to also require that the adjunction $f_!\adj f^*$ or
$f^*\adj f_*$ is ``enriched'' in a suitable sense.  To explain this
condition, suppose we have an adjunction $f_!\adj f^*$ relating
$\ord\sC^X$ and $\ord\sC^Y$; then for $x\in\sC^X$ and $y\in\sC^Y$ we
have a transformation
\begin{equation}\label{eq:indvcoprod-mate}
  f^* \left(\usC^Y(f_!x,y)\right) \toiso
  \usC^X(f^*f_!x,f^*y) \to
  \usC^X(x,f^*y)
\end{equation}
in which the second map is precomposition with the unit $x\to f^* f_!
x$ of the adjunction $f_!\dashv f^*$.  The mate
of~\eqref{eq:indvcoprod-mate} under the adjunction $f^*\dashv f_*$ in
\V is a transformation
\begin{equation}\label{eq:enr-indexed-coprod}
  \usC^Y(f_!x,y) \too f_*\left(\usC^X(x,f^*y)\right).
\end{equation}

% We will prove this in a moment, after first
% establishing a lemma relating all three base change functors and the
% fiberwise internal-homs.

% \begin{lem}\label{lem:three-functor-closed}
%   If $J\in\sV^X$, $f\maps X\to Y$, and $M\in\sV^{Y\times X}$, then we
%   have
%   \[(\pi_X)_*\usV^{Y\times X}\Big((f,1)_!J, M\Big)
%   \iso f_*\usV^X\Big(J,(f,1)^* M\Big).
%   \]
%   In particular, if $Y=X$ and $f=1$ we have
%   \[\pi_* \usV^{X\times X}\Big(\Delta_!J, M\Big)
%   \iso \usV^X\Big(J, \Delta^* M\Big).
%   \]
% \end{lem}
% \begin{proof}
%   The proof is a simple Yoneda lemma argument; for any $N\in\sV^Y$ we
%   have
%   \begin{align*}
%     \sV^Y\Big(N,(\pi_X)_*\usV^{Y\times X}\big((f,1)_!J, M\big)\Big)
%     &\iso \sV^{YX}\Big((\pi_X)^*N, \usV^{Y\times X}\big((f,1)_!J, M\big)\Big)\\
%     &\iso \sV^{YX}\big((\pi_X)^*N \ten_{Y\times X} (f,1)_!J, M\big)\\
%     &\iso \sV^{YX}\big((f,1)_!((f,1)^*(\pi_X)^*N \ten_X J), M\big)\\
%     \intertext{(Because $f_!a \ten b \iso f_!(x\ten f^*y)$ for any
%       closed monoidal $f^*$ with a left adjoint $f_!$)}
%     &\iso \sV^{YX}\big((f,1)_!(f^*N \ten_X J), M\big)\\
%     &\iso \sV^X \big(f^*N\ten_X J, (f,1)^*M\big)\\
%     &\iso \sV^X \Big(f^*N, \usV^X\big(J, (f,1)^*M\big)\Big)\\
%     &\iso \sV^Y \Big(N, f_*\usV^X\big(J, (f,1)^*M\big)\Big).
%   \end{align*}
% \end{proof}

\begin{thm}\label{thm:indexed-coproducts}
  A \V-fibration \sC\ has indexed \sV-coproducts if and only if
  $\ord\sC$ has indexed coproducts and every
  map~\eqref{eq:enr-indexed-coprod} is an isomorphism.  In this case,
  $f_!$ extends to a $\sV^Y$-enriched functor $(f_*)_\bullet
  \sC^X\to\sC^Y$.

  Dually, \sC\ has indexed \sV-products if and only if $\ord\sC$ has
  indexed products and every canonical map
  \begin{equation}\label{eq:enr-indexed-prod}
    \usC^Y(y,f_*x) \too f_*\left(\usC^X(f^*y,x)\right)
  \end{equation}
  is an isomorphism, in which case $f_*$ extends to a $\sV^Y$-enriched
  functor $(f_*)_\bullet \sC^X\to\sC^Y$.
\end{thm}
\begin{proof}
  % For $x\in\sC^X$ and $f\maps x\to y$, the indexed \sV-coproduct $f_!x
  % = Y(f,1) \oast_{[X]} x$ is characterized by an isomorphism
  % \begin{align}
  %   \usC(f_!x,z) &\iso \usC^{[X]}(Y(f,1),\usC(x,z))\\
  %   &\iso Y(f,1) \rhd \usC(x,z)\\
  %   &\iso (1\times f)_*\usC(x,z)\label{eq:indexed-coprod-char}
  % \end{align}
  First, assume that \sC\ has indexed \sV-coproducts.  We have already
  shown that then $\ord\sC$ has indexed coproducts.  And for $x\in\sC^X$
  and $y\in\sC^Y$ we have
  \begin{alignat}{2}
    \usC^Y(f_!x,y)
    &\iso \Delta_Y^* \usC(f_!x,y)\\
    &\iso \Delta_Y^* \big(Y(f,1) \rhd \usC(x,y)\big)
    &\quad\text{(by definition of indexed \V-coproducts)}\\
    &\iso \Delta_Y^* (1\times f)_* \usC(x,y)
    &\quad\text{(by the dual of~\eqref{eq:lhd-coext})}\\
    &\iso f_* \Delta_X^* (f\times 1)^*\usC(x,y)
    &\quad\text{(by the Beck-Chevalley condition)}\\
    &\iso f_* \Delta_X^* \usC(x,f^*y)
    &\quad\text{(since $f^*y$ is a restriction of $y$)}\\
    &\iso f_* \usC^X(x,f^*y).
  \end{alignat}
  We leave it to the reader to check that this isomorphism is in fact
  the canonical map~\eqref{eq:enr-indexed-coprod}.  Now using this
  isomorphism, we can define a morphism
  \[f_*\usC^X(x,w) \too f_*\usC^X(x,f^*f_!w) \iso \usC^Y(f_!x,f_!w)\]
  and check that it makes $f_!$ into a $\sC^Y$-enriched functor
  $(f_*)_\bullet \sC^X\to\sC^Y$, and the
  isomorphism~\eqref{eq:enr-indexed-coprod} into an enriched
  adjunction.

  For the other direction, suppose that $\ord\sC$ has indexed
  coproducts and that~\eqref{eq:enr-indexed-coprod} is always an
  isomorphism.  Then for any $z\in\sC$, we have
  \begin{alignat}{2}
    \usC(f_!x,z) &\iso \usC^{Z\times Y}(\pi_Z^*f_!x, \pi_Y^*z)\\
    &\iso \usC^{Z\times Y}\big((1\times f)_!\pi_Z^* x, \pi_Y^*z\big)
    &\quad\text{(by the Beck-Chevalley condition)}\\
    &\iso (1\times f)_*\usC^{Z\times X}\big(\pi_Z^* x, (1\times f)^*\pi_Y^*z\big)
    &\quad\text{(by~\eqref{eq:enr-indexed-coprod} for $1\times f$)}\\
    &\iso (1\times f)_*\usC^{Z\times X}\big(\pi_Z^* x, \pi_X^*z\big)\\
    &\iso (1\times f)_*\usC(x,z)\\
    &\iso Y(f,1) \rhd \usC(x,z)
    &\quad\text{(by the dual of~\eqref{eq:lhd-coext})}.
  \end{alignat}
  Since this isomorphism is suitably natural in $z$, by definition
  $f_! x$ is an indexed \V-coproduct.  The case of $f_*$ is similar.
\end{proof}

\begin{eg}
  If $\V=\fam(\bV)$ and $\sC=\fam(\bC)$ as in
  \autoref{eg:enriched-as-indexed}, then \sC has indexed
  $\fam(\bV)$-coproducts just when \bC\ has small \bV-coproducts.
  Here the Beck-Chevalley condition is automatic,
  but~\eqref{eq:enr-indexed-coprod} gives the coproducts their
  \bV-enriched universal property, beyond merely being coproducts in
  $\ord\bC$.
\end{eg}

\begin{eg}\label{thm:indcoprod-lsfib}
  On the other hand, for an indexed $\self(\S)$-category, one can show
  that condition~\eqref{eq:enr-indexed-coprod} is automatic; here all
  the content is in the Beck-Chevalley condition.  This is also true
  for indexed $\sPsh(\S,\nSet)$-categories (that is, ordinary indexed
  categories).
  % although some care is required since $\sPsh(\S,\nSet)$ is not a cosmos.
  More generally, for indexed $\sPsh(\S,\bV)$-categories, we need the
  adjunctions $f_!\dashv f^*$ to be \bV-enriched, in addition to the
  Beck-Chevalley condition.
\end{eg}

We now consider the other principal type of global tensors.  If
$J\in\sV^X$ and $x\in\sC^X$, we write $J\oast_X x$ for the
$\Delta_!J$-weighted colimit of $x$.  By definition, it is
characterized by an isomorphism
\begin{alignat}{2}
  \usC(J\oast_X x, z) &\iso \usV^{[X]}(\Delta_{X!} J, \usC(x,z))\\
  &\iso \pi_{X*}\Delta_X^*\big(\usV(\Delta_{X!}J, \usC(x,z))\big)\\
  &\iso \pi_{X*}\Delta_X^*(1\times \Delta_{X})_*\big(\usV(J, \usC(x,z))\big)
  &\quad\text{(by \autoref{thm:bc-enriched-ext})}\\
  &\iso \pi_{X*}\Delta_{X*}(\Delta_X\times 1)^*\big(\usV(J, \usC(x,z))\big)
  &\quad\text{(by Beck-Chevalley)}\\
  &\iso (\Delta_X\times 1)^*\big(\usV(J, \usC(x,z))\big)\\
  &\iso \usV^X(J, \usC(x,z)).\label{eq:fibtenchar}
\end{alignat}
If we choose $z$ such that $z\in\sC^X$, and apply $\Delta_X^*$ to the
above isomorphism, we obtain as a special case
\begin{align}
  \usC^X(J\oast_X x, z) &=
  \Delta_X^* \big(\usC(J\oast_X x,z)\big)\\
  &\iso \Delta_X^*(\Delta_X\times 1)^*\big(\usV(J, \usC(x,z))\big)\\
  &\iso \Delta_X^*(1\times \Delta_X)^*\big(\usV(J, \usC(x,z))\big)\\
  &\iso \Delta_X^* \big( \usV(J, \Delta_X^*\usC(x,z))\big)\\
  &\iso \usV^X\left(J, \usC^X(x,z)\right).\label{eq:fibteniso}
\end{align}
By definition, this isomorphism (natural in $z$) says that $J\oast_X
x$ is a $\sV^X$-enriched tensor of $x$ by $J$ in the fiber category
$\sC^X$.  (This is what we expected, from
\autoref{eg:enr-global-tensors}.)

\begin{defn}
  We say \sC has \textbf{fiberwise \sV-tensors} if it admits all
  colimits with weights of the form $\Delta_! J$.  Dually, we say \sC
  has \textbf{fiberwise \sV-cotensors} if it has global cotensors with
  all weights of the form $\Delta_!J$.
\end{defn}

Fiberwise cotensors are characterized by an isomorphism
\[\usC(z,\cten{J}{x}{X}) \iso \usV^X(J,\usC(z,x)).\]
and are, in particular, cotensors in the fibers.

Of course, if the fiber categories $\sC^X$ have $\sV^X$-enriched
tensors, it doesn't necessarily follow that \sC\ has fiberwise
\sV-tensors.  Here what is missing is not the enrichment (which is
already there in the definition of tensors) but stability under
restriction (which corresponds to the Beck-Chevalley condition for
indexed coproducts).

\begin{thm}\label{thm:fiberten-as-frbiten}
  \sC\ has fiberwise \sV-tensors if and only if each fiber $\sC^X$ has
  $\sV^X$-enriched tensors, and for any $f\maps X\to Y$, $J\in\sV^Y$,
  and $y\in \sC^Y$ the canonical map
  \[f^*J \oast_X f^*y \too f^*(J\oast_Y y)
  \]
  is an isomorphism.  Dually, \sC\ has fiberwise \sV-cotensors if and
  only if each fiber $\sC^X$ has $\sV^X$-enriched cotensors preserved
  by restriction.
\end{thm}
\begin{proof}
  We have already shown that if \sC\ has fiberwise \sV-tensors, then
  the fibers have tensors, so for the `only if' direction it suffices
  to show that they are preserved by restriction.  Like
  \autoref{thm:vbc-cbc}, this follows from
  \autoref{thm:composing-limits}, the fact that $f^*$ is given by a
  $Y(1,f)$-weighted colimit, and the composite isomorphism
  \[Y(1,f) \odot \Delta_{Y!}J \iso
  (f\times 1)^* \Delta_{Y!} J \iso
  (1\times f)_! \Delta_{X!} f^* J \iso 
  \Delta_{X!}{f^*J} \odot Y(1,f).
  \]
  Conversely, if we suppose that the fibers have tensors preserved by
  restriction, then we can calculate
  \begin{align*}
    \usC(J\oast_X x, z)
    &\iso \usC^{Z\times X}(\pi_Z^*(J\oast_X x), \pi_X^*z)\\
    &\iso \usC^{Z\times X}(\pi_Z^*J \oast_{X\times Z} \pi_Z^*x, \pi_X^*z)\\
    &\iso \usV^{X\times Z}\Big(\pi_Z^*J, \usC^{Z\times X}(\pi_Z^*x, \pi_X^*z)\Big)\\
    &\iso \usV^{X\times Z}\big(\pi_Z^*J, \usC(x,z)\big)\\
    &\iso \usV^X(J,\usC(x,z))
  \end{align*}
  which is~\eqref{eq:fibteniso}.  The case of cotensors is dual.
\end{proof}

\begin{eg}
  The $\fam(\bV)$-category $\fam(\bC)$ has fiberwise
  $\fam(\bV)$-tensors exactly when \bC has \bV-enriched tensors in the
  usual sense.
\end{eg}

\begin{eg}\label{eg:fibtens-lsfib}
  On the other hand, if a $\self(\bS)$-category \sC has indexed
  coproducts (hence indexed $\self(\bS)$-coproducts, by
  \autoref{thm:indcoprod-lsfib}), then it automatically has fiberwise
  $\self(\bS)$-tensors.  Namely, if $x\in\sC^X$ and $p\maps J\to X$ is
  an object of $\self(\bS)^X=\bS/X$, then $p_!p^*x$ is a fiberwise
  tensor of $x$ by $J$.  For if $z\in\sC^Z$, we have
  \begin{align}
    \U{\self(\bS)}^X(J, \usC(x,z))
    &\iso (1\times p)_*(1\times p)^*\usC(x,z)\\
    &\iso (1\times p)_*\usC(p^*x,z)\\
    % &\quad\text{(since $p^*x$ is a restriction of $x$)}\\
    &\iso X(p,1) \rhd \usC(p^*x,z)\\
    &\iso \usC(p_!p^*x,z).
  \end{align}
  Dually, if \sC has indexed products, it has fiberwise
  $\self(\bS)$-cotensors.
\end{eg}

\begin{eg}\label{eg:small-tensors}
  For size reasons, it is unreasonable to expect an indexed
  $\sPsh(\S,\bV)$-category to have all fiberwise tensors or cotensors.
  However, we can ask for fiberwise tensors by \emph{small} objects in
  the sense of \autoref{eg:psh}.  As in \autoref{eg:fibtens-lsfib}, a
  fiberwise tensor of $x\in \sC^X$ by the representable object $F_g$
  is given by $g_!g^* x$, and fiberwise tensors by small objects are
  \bV-weighted colimits of these preserved by restriction.  In
  particular, fiberwise tensors by $V\otimes F_{1_X}$, for $V\in\bV$,
  are tensors by $V$ in the \bV-enriched category $\sC^X$ which are
  preserved by restriction.
\end{eg}

% \begin{rmk}
%   One may be tempted to think that if \sC\ has indexed \sV-products,
%   then $f^*$ should preserve all tensors automatically since it is an
%   `enriched left adjoint'.  Unfortunately, this is not true because of
%   the need to apply $f_*$ `outside' in the adjunction
%   isomorphisms~\eqref{eq:enr-indexed-coprod}
%   and~\eqref{eq:enr-indexed-prod}.  The argument works if the functor
%   $f_*$ in \sV\ reflects isomorphisms, but not in general.
% \end{rmk}

\begin{rmk}
  Of course, our use of $\oast_{[X]}$ for global tensors and $\oast_X$
  for fiberwise tensors is not an accident.  The canceling product is,
  in fact, a global tensor in the \sV-category \sV, while the fiberwise
  product is a fiberwise tensor.  Dually, the canceling hom is a
  global cotensor and the fiberwise hom is a fiberwise cotensor.  (We
  could also, if we wished, define `external' tensors and cotensors in
  arbitrary \sV-categories.)
\end{rmk}

% We can use \autoref{thm:fiberten-as-frbiten} to prove a more general
% theorem relating limits in enriched fiber categories to general
% weighted limits, but first we need to set up some definitions.

For our next example, suppose that $\bA$ is a $\sV^X$-enriched
category, in the classical sense.  Then we can construct a large
\sV-category $X[\bA]$ whose objects are those of $\bA$, all with
extent $X$, and with hom-objects
\[\U{X[\bA]}(a,b) = \Delta_{X!} \big(\ubA(a,b)\big).
\]
Then for any \sV-fibration \sC, indexed \sV-functors $X[\bA]\to\sC$
are equivalent to $\sV^X$-enriched functors $\bA\to\sC^X$.

Now suppose additionally that $J\maps \bA\op\to\sV^X$ is a
$\sV^X$-enriched functor.  Then we can talk about $J$-weighted
colimits in any $\sV^X$-enriched category, and in particular in
$\sC^X$.  On the other hand, we can build a \sV-profunctor $X[J]\maps
\delta X\hto X[\bA]$ by setting $X[J](\star,a) = \Delta_{X!} (Ja)$,
and ask about $X[J]$-weighted colimits.  It should no longer be
surprising that $X[J]$-weighted colimits will turn out to be
$J$-weighted colimits which are preserved by restriction.

To make the latter precise in this case, let $f:Y\to X$ be a morphism
in \S, and observe that any $\V^X$-enriched functor $d:\bA\to\sC^X$
gives rise to a $\V^Y$-enriched functor
\[(f^*)_\bullet \bA
\xto{(f^*)_\bullet d} (f^*)_\bullet \sC^X
\xto{f^*} \sC^Y
\]
which we denote \dhat.  Similarly, $J\maps \bA\op\to\sV^X$ gives rise
to a $\V^Y$-enriched functor
\[(f^*)_\bullet \bA\op
\xto{(f^*)_\bullet J} (f^*)_\bullet \V^X
\xto{f^*} \V^Y
\]
which we denote \Jhat, and any $J$-weighted cocone under $d$ in
$\sC^X$ induces a \Jhat-weighted cocone under \dhat in $\sC^Y$.

\begin{thm}\label{thm:fiberlim-as-frbilim}
  In the above situation, a \V-fibration \sC admits $X[J]$-weighted
  colimits if and only if the fiber $\sC^X$ admits $J$-weighted
  colimits and moreover for any $f:Y\to X$, the functor
  $f^*:(f^*)_\bullet\sC^X \to \sC^Y$ takes $J$-colimiting cocones to
  \Jhat-colimiting ones.
\end{thm}
\begin{proof}%\subsection*{Sketch of proof.}
  For simplicity we assume that \bA is \ka-small and that \V has
  fiberwise \ka-small products, so that homs over $X[\bA]$ can
  be constructed as in \autoref{thm:vproflargeclosed}.  Then by
  definition, an $X[J]$-weighted colimit $\ell$ of $d:X[\bA] \to \sC$ is
  characterized by an equalizer
  \begin{equation}
    \usC(\ell, z)
    \too \prod_{a\in \bA} \uV^{[X]}(\Delta_{X!} Ja, \usC(da,z))
    \;\toto\;
    \prod_{a,a'\in \bA} \uV^{[X]}(\Delta_{X!} Ja \otimes_{[X]}
    \Delta_{X!} \bA(a,a'), \usC(da',z)).
  \end{equation}
  As in~\eqref{eq:fibtenchar}, this is equivalent to an equalizer
  \begin{equation}
    \usC(\ell, z)
    \too \prod_{a\in \bA} \uV^{X}(Ja, \usC(da,z))
    \;\toto\;
    \prod_{a,a'\in \bA} \uV^{X}(Ja \otimes_{X} \bA(a,a'), \usC(da',z)).
  \end{equation}
  Again, choosing $z\in\sC^X$ and applying $\Delta_X^*$ yields an
  equalizer
  \begin{equation}
    \usC^X(\ell, z)
    \too \prod_{a\in \bA} \uV^{X}(Ja, \usC^X(da,z))
    \;\toto\;
    \prod_{a,a'\in \bA} \uV^{X}(Ja \otimes_{X} \bA(a,a'), \usC^X(da',z))
  \end{equation}
  whence $\ell$ is the $J$-weighted colimit of $d:\bA\to\sC^X$.
  Preservation by restriction follows exactly as in the proof of
  \autoref{thm:fiberten-as-frbiten}, as does the converse.
\end{proof}

In particular, if \bA\ is an unenriched category and $J=\Delta 1$ the
standard conical weight $\bA\op\to\nSet$, then for any $X\in\bS$ we
can first take the free $\sV^X$-enriched category $\sV^X[\bA]$ and
weight $\sV^X[J]$ and then apply this construction.  We refer to the
resulting $X[\sV^X[J]]$-weighted limits and colimits as
\textbf{fiberwise \sV-limits and colimits}.  Thus we have notions of
\emph{fiberwise \sV-equalizers}, \emph{fiberwise \sV-products}, and so
on.  These can all be expressed more explicitly; for example, the
fiberwise \sV-product of two objects $x,y\in\sC^X$ is an object
$z\in\sC^X$ together with an isomorphism
\[\sC(1,z) \iso \sC(1,x)\times\sC(1,y)
\]
of profunctors $\delta X \hto \sC$.  Of course,
$\sC(1,x)\times\sC(1,y)$ denotes the profunctor whose value at $z$ is
$\sC(z,x)\times\sC(z,y)$, the product taking place in $\sV^{\ez\times
  X}$.

Since $f^*:\V^X \to \V^Y$ is a strong monoidal left adjoint, it
commutes with the free enriched category construction up to
isomorphism: $(f^*)_\bullet \V^X[\bA] \cong \V^Y[\bA]$.  Thus, to give
any sort of fiberwise conical limit is equivalently to give a
$\V^X$-enriched conical limit in $\sC^X$ which is preserved by
restriction, in the appropriate sense.

\begin{eg}
  Clearly, the indexed $\fam(\bV)$-category $\fam(\bC)$ has fiberwise
  $\fam(\bV)$-limits and colimits just when \bC has the relevant
  \bV-enriched ones.  Here again the preservation by restriction is
  automatic, but the fact that the limits are enriched in each fiber
  is crucial.
\end{eg}

\begin{eg}\label{eg:psh-fiblim}
  As usual, by contrast, in an indexed $\self(\S)$-category or
  $\sPsh(\S,\nSet)$-category, it is the preservation by restriction
  which contains the content.  Once we know that limits exist in
  fibers and are preserved by restriction, the fact that they are
  enriched in each fiber follows automatically.  Similarly, fiberwise
  $\sPsh(\S,\bV)$-limits are \bV-enriched limits in fibers preserved
  by restriction.
\end{eg}

Finally, it is well-known in classical enriched category theory that
if cotensors exist, then the distinction between enriched and
unenriched ordinary limits disappears
(see~\cite[\SS3.8]{kelly:enriched}).  The analogue of this for
\sV-categories is the following.

\begin{thm}\label{thm:coten-enriched-colim}
  Let \sC\ be a \sV-fibration with fiberwise \V-cotensors.
  \begin{enumerate}
  \item \sC\ has fiberwise \sV-colimits of a given (conical) type if
    and only if $\ord\sC$ has fiberwise colimits of that
    type.\label{item:coten-enr-colim}
  \item \sC\ has indexed \sV-coproducts if and only if $\ord\sC$ has
    indexed coproducts.\label{item:coten-enr-indcoprod}
  \end{enumerate}
\end{thm}
Of course, there is a dual result for tensors and limits.
\begin{proof}
  Both `only if' statements have already been proven.  By
  \autoref{thm:fiberlim-as-frbilim}, for~\ref{item:coten-enr-colim} it
  suffices to show that fiberwise colimits in $\ord\sC^Y$ are actually
  $\sV^Y$-enriched for any $Y$.  But this follows from the classical
  version of this theorem, since $\sC^X$ is $\sV^X$-enriched and
  cotensored.

  For~\ref{item:coten-enr-indcoprod}, by
  \autoref{thm:indexed-coproducts} it suffices to show
  that~\eqref{eq:enr-indexed-coprod} is an isomorphism.  For $f\maps
  X\to Y$, $x\in\sC^X$, $y\in\sC^Y$, and any $J\in\sV^X$, we compute
  \begin{align*}
    \sV^Y\left(J, \usC^Y(f_!x, y)\right)
    &\iso \sC^Y\left(f_!x, \cten{J}{y}{Y}\right)\\
    &\iso \sC^X\left(x, f^* \cten{J}{y}{Y}\right)\\
    &\iso \sC^X\left(x, \cten{f^*J}{f^*y}{X}\right)\\
    &\iso \sV^X\left(f^*J, \usC^X(x,f^*y)\right)\\
    &\iso \sV^Y\left(J, f_* \usC^X(x,f^*y)\right)
  \end{align*}
  so the desired isomorphism~(\ref{eq:enr-indexed-coprod}) follows by
  the Yoneda lemma in $\sV^Y$.
\end{proof}

We now turn to the question of constructing general \V-limits and
colimits out of basic ones such as those we have just studied.  Our
first observation is that the generality in allowing $K$ to be a
non-discrete or large \sV-category comes for free.

\begin{thm}\label{thm:parametrized-limits}
  Let $J\maps \sK\hto A$ be a \V-profunctor, let $f\maps A\to \sC$ be
  a \sV-functor and suppose that the $J(1,k)$-weighted colimit of $f$
  exists for all objects $k$ of \sK.  Then the $J$-weighted colimit of
  $f$ also exists, and agrees with $\colim^{J(1,k)} f$ upon restriction
  to $\delta(\e k)$ for each $k$.
\end{thm}
\begin{proof}
  We must define a \V-functor $\colim^J f : \sK \to \sC$.  Its action
  on objects is fixed: we send $k$ to $\colim^{J(1,k)} f$.  What remains
  of the data is morphisms
  \[\usK(k,k') \to \usC(\colim^{J(1,k)} f, \colim^{J(1,k')} f).\]
  But if we consider this to be a morphism of profunctors $\delta(\e
  k') \hto \delta(\e k)$, then we can obtain it by passing across the
  isomorphism
  \begin{multline}
    \VProf(\usK(k,k'), \sC(\colim^{J(1,k)} f, \colim^{J(1,k')} f))\\
    \begin{aligned}
      &\cong \Vmmor( \sC(\colim^{J(1,k')} f,1), \usK(k,k') ;
      \sC(\colim^{J(1,k)} f,1) )\\
      &\cong \Vmmor( \sC(\colim^{J(1,k')} f,1), \usK(k,k'), J(1,k) ;
      \sC(f,1) )
    \end{aligned}
  \end{multline}
  from the composite multimorphism
  \[ \sC(\colim^{J(1,k')} f,1), \usK(k,k'), J(1,k) \too
  \sC(\colim^{J(1,k')} f,1), J(1,k') \too
  \sC(f,1)
  \]
  built out of the action of \sK on $J$ and the universal bimorphism
  of $\colim^{J(1,k')} f$.  It is straightforward to show that this
  defines a \V-functor, and since isomorphisms of profunctors $\sC\hto
  \sK$ are detected at each object of \sK, this functor must be the
  $J$-weighted colimit of $f$.
\end{proof}

\begin{cor}
  For a fixed \sV-category $A$, if \sC\ admits all colimits with
  weights $J\maps \delta X\hto A$, then it admits all colimits with
  weights $J\maps \sK\hto A$.\endproof
\end{cor}

Next, we observe that the two special cases of tensors considered
above actually suffice to reconstruct all global \sV-tensors.

\begin{thm}
  If \sC\ is a \sV-category with indexed \sV-coproducts and fiberwise
  \sV-tensors, then \sC\ admits all global tensors.  Dually, if \sC\
  has indexed \sV-products and fiberwise \sV-cotensors, then it admits
  all global cotensors.
\end{thm}
\begin{proof}
  Let $J\in\sV^{Y\times X}$ be a weight for a global tensor.  Then we
  can define
  \[ J' = \Delta_{(Y\times X)!}J\in\sV^{Y\times X\times Y\times X},\]
  which we can regard as a profunctor $J'\maps \delta (Y\times X)\hto
  \delta (Y\times X)$.
  % Recall that $\pi_Y\maps Y\times X\to X$ and
  % $\pi_X\maps Y\times X\to Y$ denote the projections.
  Now since
  $(\pi_X\times\pi_Y)\circ \Delta_{Y\times X}$ is the identity, we have
  \begin{align*}
    J &\iso (\pi_X\times \pi_Y)_!(J')\\
    &\iso Y(\pi_X,1) \odot J' \odot X(1, {\pi_Y}).
  \end{align*}
  Thus, by \autoref{thm:composing-limits}, $J$-weighted colimits can
  be built from $X(1, {\pi_Y})$-weighted colimits, $J'$-weighted
  colimits, and $Y(\pi_X,1)$-weighted colimits.  However, $X(1,
  {\pi_Y})$-weighted colimits are restrictions, $J'$-weighted colimits
  are fiberwise \sV-tensors, and $Y(\pi_X,1)$-weighted colimits are
  indexed \sV-coproducts.
\end{proof}

Finally, we have an enriched indexed version of the classical
construction of colimits out of tensors, coproducts, and coequalizers.

\begin{thm}
  If \sC\ admits global \sV-tensors, fiberwise \sV-coequalizers, and
  fiberwise \sV-coproducts of the size of the set of objects of $A$,
  then it admits all colimits with weights $J\maps K\hto A$.
\end{thm}
\begin{proof}
  By \autoref{thm:parametrized-limits} it suffices to assume that $K$
  is a discrete small \sV-category $\delta Y$.  We then have to show
  that for any $f\maps A\to \sC$ there is an object $\colim^J
  f\in\sC^X$ with an isomorphism
  \[\usC(\colim^J f, 1) \iso J \rhd \usC(f,1)\]
  of profunctors $\sC\hto \delta Y$.  Now, by the construction of
  $\rhd$ in \autoref{thm:vproflargeclosed}, we have
  \[J\rhd \usC(f,-) \iso \eq \left(\prod_{a\in A} \uJ(\star,a)\rhd \usC(fa,-)
    \toto \prod_{a,b\in A} (\uJ(\star,b)\odot \uA(a,b)) \rhd \usC(fa,-)
  \right).
  \]
  Since \sC\ admits global \sV-tensors, for each $a\in A$ we have an
  object $x_a\in\sC^X$ such that
  \[\usC(x_a,-)\iso \uJ(\star,a)\rhd \usC(fa,-).\]
  Similarly, for each pair $a,b\in A$ we have an object
  $y_{a,b}\in\sC^X$ such that
  \[\usC(y_{a,b},-) \iso (\uJ(\star,b)\odot \uA(a,b)) \rhd \usC(fa,-).\]
  Therefore, we have
  \[J\rhd \usC(f,-) \iso \eq \left(\prod_{a\in A} \usC(x_a,-)
    \toto \prod_{a,b\in A} \usC(y_{a,b},-)
  \right).
  \]
  where the two maps are induced by canonical morphisms $y_{a,b}\to
  x_a$ and $y_{a,b}\to x_b$ in $\sC^X$.  Similarly, since \sC\ has
  $\kappa$-sized fiberwise \sV-coproducts, there is an object $z$ such
  that
  \[\usC(z,-)\iso \prod_{a\in A} \usC(x_a,-),\]
  and an object $w$ such that
  \[\usC(w,-)\iso \prod_{a,b\in A} \usC(y_{a,b},-).\]
  Thus we have
  \[J\rhd \usC(f,-) \iso \eq \Big(\usC(z,-)\toto \usC(w,-)\Big).\]
  Finally, because \sC\ has fiberwise \sV-coequalizers, there is an
  object $\colim^J f$ such that 
  \[\usC(\colim^J f,-) \iso \eq \Big(\usC(z,-)\toto \usC(w,-)\Big),\]
  which completes the proof.
\end{proof}

\begin{cor}\label{thm:ggcomplete-frbicomplete}
  If \sC admits indexed \sV-coproducts, fiberwise \sV-tensors, and
  fiberwise \sV-coequalizers, then it admits all \sV-colimits with
  weights $J\maps K\hto A$ where $A$ is small in the sense of
  \SS\ref{sec:small-cats}.\endproof
\end{cor}

% \begin{cor}
%   If \sC\ is a \bS-fibration such that $\sC_0$ is a BC bifibration
%   with fiberwise coequalizers, then \sC\ admits all \bS-colimits with
%   weights $J\maps K\hto A$ where $A$ is an \bS-internal category.
% \end{cor}
% \begin{proof}
%   We observed in the course of \SS\ref{sec:special-limits} that the
%   given conditions ensure indexed \bS-coproducts, fiberwise
%   \bS-tensors, and fiberwise \bS-coequalizers.
% \end{proof}

In~\cite{gg:fib-rel}, an indexed \V-category satisfying the hypotheses
of \autoref{thm:ggcomplete-frbicomplete} was called \emph{cocomplete}.
% (They then proved that such a \sC admits all Kan extensions, a special
% case of \autoref{thm:ggcomplete-frbicomplete}.)
Unfortunately,
however, the converse of \autoref{thm:ggcomplete-frbicomplete} fails,
since fiberwise \sV-coequalizers are not a small \sV-colimit: the
relevant \sV-category $X[\sV^X[\bA]]$ has two objects.  But if we add
at least finite fiberwise coproducts, we do get an equivalence.

\begin{cor}\label{thm:constr-of-colim}
  Let $\kappa$ be an infinite regular cardinal.
  % such that \sV is $\kappa$-cocomplete.
  The following are equivalent for a \sV-category \sC.
  \begin{enumerate}
  \item \sC\ admits indexed \sV-coproducts, fiberwise \sV-tensors, and
    fiberwise \sV-colimits of cardinality $<\kappa$.
  \item \sC\ admits all colimits with weights $J\maps K\hto A$ where
    $A$ is \kappa-small.\endproof
  \end{enumerate}
\end{cor}

In the situation of \autoref{thm:constr-of-colim}, we will say that
\sC is \textbf{\ka-cocomplete}.  The two most important cases are when
$\kappa=\om$, since \om-cocompleteness is an elementary condition, and
when $\kappa=\infty$ is the size of the universe, in which case we
simply say \sC is \textbf{cocomplete}.

Finally, combining \autoref{thm:constr-of-colim} with
\autoref{thm:coten-enriched-colim}, we see that when \sC\ is tensored
and cotensored, it suffices to construct colimits in the underlying
fibration.

\begin{cor}
  If \sC\ has fiberwise \sV-tensors and \V-cotensors, and $\ord\sC$ has
  indexed coproducts and fiberwise \kappa-small colimits, then \sC is
  \kappa-cocomplete.
\end{cor}

Of course, everything we have proven applies dually to limits as well,
and we have a notion of \textbf{\ka-completeness}.

\begin{eg}\label{eg:psh-cplt}
  The same proofs show that for a locally small \S with pullbacks and
  a classical cosmos \bV, an indexed $\sPsh(\S,\bV)$-category admits
  colimits with all weights $J:K\hto A$, where $A$ is set-small and
  locally small (in the sense of \autoref{thm:psh-comphom}) and $J$ is
  locally small (in the same sense), if and only if it has indexed
  coproducts as in \autoref{thm:indcoprod-lsfib}, fiberwise tensors by
  small objects as in \autoref{eg:small-tensors}, and small fiberwise
  colimits as in \autoref{eg:psh-fiblim}.  This is equivalent to
  asking for \bV-enriched adjunctions $f_!\dashv f^*$ satisfying the
  Beck-Chevalley condition, plus that each \bV-enriched category
  $\sC^X$ is \bV-cocomplete, with colimits preserved by the
  restriction functors $f^*$.
\end{eg}

\section{Presheaf \V-categories}
\label{sec:psh}

Our goal is now to define presheaf \V-categories.  Here again it is
convenient to use the machinery of profunctors.  In particular, for
bicategory-enrichment the construction was already done
by~\cite{street:enr-cohom}.  We will refine it slightly, so as to
simultaneously give a notion of ``small-presheaf category'' as
in~\cite{dl:lim-smallfr} and a version suitable for an elementary
context.

Let \S have finite products, and let \V be an \S-indexed cosmos which
is \ka-complete and \ka-cocomplete for some chosen regular cardinal
\ka.  We also include the case ``$\ka=\infty$'' for which every
(small) set is \ka-small.  (As in \SS\ref{sec:indexed-limits}, with
appropriate care we could weaken these assumptions.  In particular,
symmetry is not really needed.)  The cases of most interest are
$\ka=\om$ (for when we care about being first-order) and $\ka=\infty$.

\begin{defn}\label{def:smprof}
  A \V-profunctor $H:\sB\hto \sA$ is \textbf{\ka-small} if for every
  $b\in\sB$, there exist a \ka-small \V-category $A'$, a
  % fully faithful
  \V-functor $i:A'\to\sA$, and a \V-profunctor $H':\delta(\eb)\hto A'$
  such that $H(1,b)\cong H' \odot \sA(1,i)$.
\end{defn}

Note that since \V is \ka-cocomplete and $A'$ is \ka-small, the
composite $H' \odot \sA(1,i)$ automatically exists.

\begin{eg}
  The unit profunctor $\sA:\sA\hto\sA$ is always \ka-small; for
  each $a\in\sA$ we may take $A'=\delta(\ea)$, $H'$ the unit
  profunctor, and $i:\delta(\ea)\to\sA$ the inclusion.
\end{eg}

\begin{eg}\label{thm:smallcat-smallprof}
  If \sA is itself \ka-small, then every profunctor $\sB\hto\sA$ is
  \ka-small, as for any $b$ we may take $A'=\sA$ and $i=1_{\sA}$.
\end{eg}

\begin{eg}
  If $H:\sB\hto\sA$ is \ka-small, then so is $H(1,f)$ for any
  $f:\sC\to\sB$, since $H(1,f)(1,c) = H(1,f(c))$.
\end{eg}

\begin{eg}
  Finally, the connection with the most classical case is a little bit
  surprising.  Suppose $\V=\fam(\bV)$ and $\sA = \fam(\bA)$ as in
  \autoref{eg:enriched-as-indexed}, and that $H$ is likewise induced
  from a \bV-enriched profunctor $\bB\hto\bA$.  Note that a \ka-small
  $\fam(\bV)$-category is equivalently a small \bV-enriched category
  with a partition of its objects into a family of sets with a
  \ka-small indexing set (but no cardinality restrictions on the
  individual sets in the partition).  Profunctors between such
  categories are, up to equivalence, just \bV-enriched profunctors,
  and functors are those that respect the partitions.  It follows that
  $H$ is \ka-small in the sense of \autoref{def:smprof} if and only if
  it is small in the sense of~\cite{dl:lim-smallfr}---in particular,
  the cardinal \ka is completely irrelevant!

  This makes more sense if we realize that \ka does not exactly
  measure the ``size'' of a \V-category or profunctor, but rather its
  ``departure from elementarity'', the case $\ka=\om$ being the purely
  elementary one.  The point is that for \nSet-indexed categories of
  families, sets of arbitrary cardinality are already built into the
  indexing and have ``become elementary''.
  % (A similar ``collapse'' happens in general if \S is a \ka-ary
  % extensive category, and \V and \sA are stacks for its \ka-extensive
  % topology.)
\end{eg}

\begin{rmk}\label{rmk:smprof-ffind}
  In \autoref{def:smprof} we are free to assume that $i$ is fully
  faithful and indexed.  For if not, define a new \ka-small
  \V-category $A''$ with one object $\xhat$ for every object $x\in
  A'$, and with $\e\xhat = \e(ix)$ and $\U{A''}(\xhat,\yhat) =
  \usA(ix,iy)$.  Then $i$ factors as $A' \xto{k} A'' \xto{j} \sA$,
  where $j$ is fully faithful and indexed, and we have
  \[ H' \odot \sA(1,i) \cong (H' \odot A''(1,k)) \odot \sA(1,j), \]
  the composite $H' \odot A''(1,k)$ existing since $A'$ is \ka-small.
\end{rmk}

We also observe that when the domain of a small profunctor is a small
category, then the decompositions of \autoref{def:smprof} can be
assembled into a single one.

\begin{lem}\label{thm:smprof-assemble}
  If $H:B\hto \sA$ is a \ka-small profunctor, where $B$ is a \ka-small
  \V-category, then there exists a \ka-small \V-category $A'$, a
  \V-profunctor $H':B\hto A'$, and a \V-functor $i:A'\to \sA$ such that
  $H\cong H'\odot \sA(1,i)$.
\end{lem}
\begin{proof}
  By assumption, for each $b\in B$ we have a \ka-small \V-category
  $A_b$, a profunctor $H_b:\delta(\eb)\hto A_b$, and a functor
  $i_b:A_b \to A$ with $H(1,b) \cong H_b \odot \sA(1,i_b)$.  By
  \autoref{rmk:smprof-ffind} we may assume each $i_b$ to be fully
  faithful and indexed.

  Define $A'$ to have as objects the disjoint union of the objects of
  the $A_b$, for all $b$.  This is a \ka-small set since each $A_b$ is
  \ka-small and so is $B$.  We let these objects inherit their extents
  from $A_b$ (and hence from \sA), and take their hom-objects to be
  \begin{equation}
    \U{A'}(a,a') =
    \begin{cases}
      \U{A_b}(a,a') &\quad\text{if } a,a' \in A_b\\
      \emptyset &\quad\text{if } a\in A_b \text{ and } a'\in A_{b'}
      \text{ with } b\neq b'
    \end{cases}
  \end{equation}
  where $\emptyset$ denotes a fiberwise initial object.  It is easy to
  check that this defines a \V-category and that we have functors
  $i:A'\to \sA$ and $j_b:A_b \to A'$ with $i j_b = i_b$.

  Similarly, we define $H':B\hto A'$ by
  \begin{equation}
    \U{H'}(a,b) =
    \begin{cases}
      \U{H_b}(a,\star) &\quad\text{if } a \in A_b\\
      \emptyset &\quad\text{otherwise}.
    \end{cases}
  \end{equation}
  Then $H'(1,b) = H_b \odot A'(1,j_b)$.  Thus, the isomorphisms $H_b
  \odot \sA(1,i_b) \toiso H(1,b)$ assemble into a morphism $H' \odot
  \sA(1,i) \to H$, which restricts to an isomorphism at each $b$ and
  hence is itself an isomorphism.
\end{proof}

The functor $i:A'\to \sA$ constructed in the proof of
\autoref{thm:smprof-assemble} is not fully faithful, but we may apply
the argument of \autoref{rmk:smprof-ffind} to make it so.

Note also that the converse of \autoref{thm:smprof-assemble} is
universally valid: if $A'$ is \ka-small, then for any $i:A'\to \sA$
and $H:\sB\hto A'$, the composite $H\odot \sA(1,i)$ is \ka-small,
since $(H\odot \sA(1,i))(1,b) \cong H(1,b) \odot \sA(1,i)$.

\begin{lem}\label{thm:smprof-comp}
  If $K:B\hto C$ is a \ka-small \V-profunctor, $B$ is a \ka-small
  \V-category, and $H:A\hto B$ is any \V-profunctor, then the
  composite $H\odot K$ is \ka-small.
\end{lem}
\begin{proof}
  Write $K = K' \odot C(1,i)$ as in \autoref{thm:smprof-assemble}.
  Then $H\odot K \cong (H\odot K')\odot C(1,i)$.
\end{proof}

\begin{lem}\label{thm:colim-smwgt}
  If a \V-category \sC is \ka-cocomplete (i.e.\ admits all colimits
  with weights $J:K\hto A$ where $A$ is \ka-small) then it admits all
  colimits with \ka-small weights.  Similarly, any \V-functor that
  preserves \ka-small colimits preserves all colimits with \ka-small
  weights.
\end{lem}
\begin{proof}
  By \autoref{thm:parametrized-limits}, it suffices to show that \sC
  admits $J$-weighted colimits for any \ka-small profunctor $J:\delta
  X \hto A$.  But then we have $J \cong J' \odot A(1,i)$ for some
  $J':\delta X\hto A'$ and $i:A'\to A$ with $A'$ being \ka-small, and
  thus for any $f:A\to \sC$,
  \[ \colim^J f \cong \colim^{J' \odot A(1,i)} f
  \cong \colim^{J'} \colim^{A(1,i)} f
  \cong \colim^{J'} f i
  \]
  (using \autoref{thm:composing-limits}), which exists because $A'$ is
  \ka-small.  The second statement follows immediately.
\end{proof}

\begin{lem}\label{thm:smprof-hom}
  For any \ka-small profunctor $H:\sB\hto\sA$ and any profunctor
  $K:\sC\hto\sA$, the hom $H\rhd K$ exists.
\end{lem}
\begin{proof}
  By \autoref{thm:objwise-comp}, it suffices to show that $H(1,b) \rhd
  K(1,c)$ exists for all $b\in\sB$ and $c\in\sC$.  Fixing such, let
  $H(1,b) \cong H' \odot \sA(1,i)$, for $i:A'\to \sA$ and
  $j:A''\to\sA$ with $A'$ \ka-small.  Then for any well-typed $\vec L
  = L_1,\dots,L_n$, we have
  \begin{align}
    \VBimor(\vec L, H(1,b); K(1,c))
    &\cong \Vmmor(\vec L, H', \sA(1,i); K(1,c))\\
    &\cong \Vmmor(\vec L, H'; K(1,c) \odot \sA(i,1))\\
    &\cong \Vmmor(\vec L; H' \rhd (K(1,c) \odot \sA(i,1))).
  \end{align}
  The composite $K(1,c) \odot \sA(i,1)$ exists by
  \autoref{thm:coyoneda}, and the hom $H' \rhd (K(1,c) \odot
  \sA(i,1))$ exists because $A'$ is \ka-small.
\end{proof}

Finally, we are ready to define presheaf \V-categories.

\begin{defn}\label{def:presheaves}
  Let $A$ be any \V-category.  Then there is a \V-fibration
  $\cPk A$ defined as follows.
  \begin{enumerate}
  \item Its objects are \ka-small \V-profunctors $H:\delta X \hto A$,
    for some $X\in\S$.
  \item The extent of $H:\delta X \hto A$ is $X$.
  \item The hom-object $\U{\cPk A}(H,K)\in\V^{X\times Y}$ is $H\rhd K$
    (which exists by \autoref{thm:smprof-hom}).
  \item The unit morphism $\I_X \to H\rhd H$ is adjunct to the
    identity $H\to H$.
  \item The composition morphism $(K\rhd L) \odot (H\rhd K) \to (H\rhd
    L)$ is adjunct to
    \[ (K\rhd L) \odot (H\rhd K) \odot H \to (K\rhd L)\odot K \to L. \]
  \item The restriction of $H:\delta X \hto A$ along $f:Y\to X$ is
    $H(1,f):\delta Y\hto A$.
  \end{enumerate}
  % There is a similar \V-fibration $\cPd A$ whose objects are
  % \V-profunctors $A\hto \delta X$, and whose hom-objects are $K\lhd
  % H$.
\end{defn}

\begin{rmk}
  Recall that if $A$ is \ka-small, then so is every profunctor into
  it.  In particular, if $\ka'\ge\ka$ then $\cP_{\ka'} A = \cPk A$.
  Again we see that $\ka$ measures not the ``size'' of the
  cocompletion \emph{per se}, but its non-elementariness, and once it
  is above the level of $A$ no further change takes place.  When \ka
  has ``stabilized'' in this sense, we may write merely $\cP A$.
\end{rmk}

\begin{eg}
  If $\V=\fam(\bV)$ and $\sA=\fam(\bA)$, then we have remarked that a
  \V-profunctor into \sA is small just when it is induced by a small
  \bV-enriched profunctor into \bA.  This makes it easy to identify
  $\cPk \sA$ with $\fam(\cP \bA)$, where $\cP \bA$ is the category of
  small \bV-enriched presheaves from~\cite{dl:lim-smallfr}.
\end{eg}

\begin{eg}
  If $\V=\self(\S)$ and $A$ is an \S-internal category, regarded as a
  small \V-category, then every profunctor into $A$ is \ka-small, and
  $\cPk A$ is the usual locally internal category of internal
  presheaves on $A$.
\end{eg}

\begin{eg}\label{eg:pshvs-over-enriched}
  If $\V=\self(\S)$ and \bA is a \ka-small category \emph{enriched}
  over \S with its cartesian monoidal structure, then we can regard
  \bA as a \ka-small $\self(\S)$ category all of whose objects have
  extent $1$.  In this case, $\cPk \bA$ is a locally internal category
  of \S-enriched presheaves on \bA, with $(\cPk \bA)^X =
  (\S/X)^{\bA\op}$.

  More generally, for any \V we may perform the same construction with
  \bA being a $\V^1$-enriched category, obtaining a \V-fibration $\cPk
  \bA$ with $(\cPk \bA)^X = (\V^X)^{\bA\op}$.  In particular, \bA
  might be freely generated by an \S-internal category as in
  \autoref{thm:free-monfib}.
\end{eg}

These presheaf categories have a universal property relating to the
following universal profunctors.

\begin{defn}
  For any $A$, there is a \V-profunctor $Y_A:\cPk A \hto A$ defined
  for $a\in A$ and $H:\delta X\hto A$ by
  \[ \uY(a,H) = \uH(a,1) \in \V^{X \times \ea}. \]
  Its action by $A$ is determined by the action of $A$ on the $H$'s,
  while its action by $\cPk A$ is determined by adjunction from the
  universal property of the homs $H\rhd K$.
  Since each $H\in\cPk A$ is \ka-small, so is $Y_A$.
  % Similarly, there is a profunctor $\Yd_A:A\hto \cPkd A$ defined for
  % $a\in A$ and $H:A\hto \delta X$ by
  % \[ \uYd(H,a) = \uH(1,a) \in\V^{\ea \times X}. \]
\end{defn}

Thus, any \V-functor $f:\sB\to\cPk A$ induces a \ka-small profunctor
$Y_A(1,f):\sB\hto A$.  Since $Y_A(1,f) \cong \cPk A(1,f) \odot Y_A$,
we have a canonical multimorphism
\[ \cPk A(1,g), \cPk A(f,1), Y_A(1,f) \too \cPk A(1,f), Y_A
\too Y_A(1,f).
\]
And since $\cPk A(f,g) \cong \cPk A(1,g) \odot \cPk A(f,1)$, this
multimorphism factors uniquely through a bimorphism
\begin{equation}\label{eq:pshfbimor}
  \cPk A(f,g), Y_A(1,f) \too Y_A(1,g).
\end{equation}

\begin{prop}\label{thm:pshf-ff}
  For any \V-functors $f:\sB\to \cPk A$ and $g:\sC\to\cPk A$, the
  bimorphism~\eqref{eq:pshfbimor} exhibits an isomorphism
  \begin{equation}\label{eq:pshf-homs}
    \cPk A(f,g) \cong Y_A(1,f) \rhd Y_A(1,f).
  \end{equation}
\end{prop}
\begin{proof}
  Follows directly from the definition of hom-objects in $\cPk A$ and
  \autoref{thm:objwise-comp}.
\end{proof}

\begin{prop}\label{thm:pshf-eso}
  For any \V-category \sB, the functor
  \begin{equation}\label{eq:wk-pshf-obj-functor}
    \begin{array}{rcl}
      \VCAT(\sB,\cPk A) &\too& \VPROF(\sB,A)\\
      \big[f: \sB\to \cPk A\big] &\mapsto&
      \big[Y_A(1,f)\maps \sB\hto A\big]
    \end{array}
  \end{equation}
  is fully faithful, and its image consists of the \ka-small
  profunctors.
\end{prop}
\begin{proof}
  Invoking \autoref{thm:pshf-ff}, for $f,g:\sB\to\cPk A$ we have
  \begin{align}
    \VCAT(\sB,\cPk A)(f,g)
    &\cong \VPROF(\sB, \cPk A(f,g))\\
    &\cong \VBimor(\sB, Y_A(1,f); Y_A(1,g))\\
    &\cong \VPROF(Y_A(1,f); Y_A(1,g)).
  \end{align}
  It is straightforward to verify that this isomorphism is the action
  of~\eqref{eq:wk-pshf-obj-functor} on homs;
  thus~\eqref{eq:wk-pshf-obj-functor} is fully faithful.

  Now, this functor certainly takes values in \ka-small profunctors.
  Conversely, suppose $H: \sB\hto A$ is \ka-small.  Of course, $H$
  consists of objects $\uH(a,b)\in \V^{\ea \times \eb}$, for each pair
  of objects $a\in A$ and $b\in\sB$, together with action maps
  \begin{align}
    \uH(a,b)\odot \uA(a',a) &\to \uH(a',b)
    \mathrlap{\quad\text{and}}\label{eq:vdiag-is-pshf-act-1}\\
    \usB(b,b')\odot \uH(a,b) &\to \uH(a,b').
    \label{eq:vdiag-is-pshf-act-2}
  \end{align}
  The maps~\eqref{eq:vdiag-is-pshf-act-1} make each $\uH(-,b)$ into a
  profunctor $\delta(\eb) \hto A$, which is \ka-small since $H$ is;
  thus it is an object of $\cPk A$ with extent \eb.  And the
  maps~\eqref{eq:vdiag-is-pshf-act-2} have adjuncts
  \[\usB(b,b') \too \uH(a,b) \rhd \uH(a,b') = \U{\cPk A}(\uH(-,b), \uH(-,b')),\]
  so we can define an indexed \sV-functor $\sB\to \cPk A$ sending each
  $b$ to $\uH(-,b)$.  We leave it to the reader to check that this
  works.
\end{proof}

% Dually, for $f:\sB\to \cPkd A$ and $g:\sC\to\cPkd A$, we have
% \[ \cPkd A(f,g) \cong \Yd_A(g,1) \lhd \Yd_A(f,1) \]
% and an equivalence $\VCAT(\sB,\cPkd A) \simeq \VPROF(A,\sB)$.

% For most of what we do in this section, we will need only the
% conclusions of \autoref{thm:pshf-ff} and \autoref{thm:pshf-eso}, not
% the construction given in \autoref{def:presheaves}.  In particular,
% everything we do makes sense in an arbitrary virtual equipment with
% ``presheaf objects'' satisfying these two properties.

% \begin{rmk}
%   If for some \V-category $A$ there exists a \V-category $\cPk A$ and a
%   profunctor $Y_A:\cPk A \hto A$ satisfying \autoref{thm:pshf-ff} and
%   \autoref{thm:pshf-eso}, and also $\cPkd A$ and $\Yd_A$ satisfying
%   their duals, then $A$ must be effectively co-small.  For then any
%   profunctor into or out of $A$ is classified by a functor into $\cPk
%   A$ or $\cPkd A$, whereas for these the right or left homs always
%   exist by~\eqref{eq:pshf-homs}.
% \end{rmk}

Since the unit \V-profunctor $A\maps A\hto A$ is \ka-small, it has a
classifying \V-functor $y_A\maps A\to \cPk A$, which we call the
\textbf{Yoneda embedding}.  Thus, by definition, we have $Y_A(1,y_A)
\cong A$.  On the other hand, we can recover $Y_A$ from $y_A$, since
by \autoref{thm:pshf-ff}
\begin{equation}\label{eq:Yfromy}
  \cPk(y_A,1) \;\cong\; Y_A(1,y_A) \rhd Y_A
  \;\cong\; A \rhd Y_A
  \;\cong\; Y_A.
\end{equation}

% Now supposing that $\cPk A$ is a presheaf object, if we take $f=y$ and
% $g=1_{\cPk A}$ we have
% \begin{align*}
%   \cPk A(y,1)
%   &\iso Y(1,y) \rhd Y(1,1)\\
%   &\iso A \rhd Y\\
%   &\iso Y
% \end{align*}
% as we hoped for.

\begin{lem}\label{thm:yoneda-ff}
  The Yoneda embedding is fully faithful.
\end{lem}
\proof%\begin{proof}
  Taking $f=g=y_A$ in \autoref{thm:pshf-ff} yields
  \begin{equation}
    \cPk A(y_A,y_A) \;\iso\; Y_A(1,y_A) \rhd Y_A(1,y_A) 
    \;\iso\; A \rhd A
    \;\iso\; A.
    \tag*{\endproofbox}
\end{equation}

We now move on to familiar completeness properties of presheaf
categories.

\begin{thm}\label{thm:pshf-cplt}
  For any $B$, the \V-category $\cPk B$ is \ka-cocomplete.
\end{thm}
\begin{proof}
  Suppose $A$ is \ka-small, and let $J\maps K\hto A$ be a weight and
  $f\maps A\to \cPk B$ a functor.  Then $f$ classifies a \ka-small
  profunctor $Y_B(1,f) \maps A\hto B$, and we define $\ell\maps K\to
  \cPk B$ to be the classifying map of the composite $J\odot
  Y_B(1,f)$.  This composite exists because $A$ is \ka-small, and the
  composite is itself \ka-small by \autoref{thm:smprof-comp}.  We then
  have
  \begin{align*}
    \cPk B({\ell},1) &\iso Y_B(1,\ell) \rhd Y_B
    & \text{(by~(\ref{eq:pshf-homs}))}\\
    &\iso \big(J\odot Y_B(1,f)\big) \rhd Y_B\\
    &\iso J\rhd \big(Y_B(1,f) \rhd Y_B\big)\\
    &\iso J\rhd \cPk B(f,1) &\text{(by~(\ref{eq:pshf-homs}) again)},
  \end{align*}
  so \ell\ is a $J$-weighted colimit of $f$.
\end{proof}

\begin{cor}\label{thm:pshf-colim-repr}
  For any $f: B\to \cPk A$ we have
  \[f \iso \colim^{Y_A(1,f)} y_A,
  \]
  In other words, every presheaf is a colimit of representables.
\end{cor}
\begin{proof}
  By the construction of colimits in \autoref{thm:pshf-cplt}, we have
  \begin{equation}
    Y_A\Big(1,\colim^{Y_A(1,f)} y_A\Big)
    \;\iso\; Y_A(1,f) \odot Y_A(1,y_A)
    \;\iso\; Y_A(1,f) \odot A
    \;\iso\; Y_A(1,f).
  \end{equation}
  Thus, by the essential uniqueness of classifying arrows,
  $f\iso \colim^{Y_A(1,f)} y_A$.
\end{proof}

\begin{cor}\label{thm:yoneda-dense}
  The Yoneda embedding is dense, i.e.\ $1_{\cPk A}$ is the left Kan
  extension of $y_A$ along itself.
\end{cor}
\begin{proof}
  Take $f=1_{\cPk A}$ in \autoref{thm:pshf-colim-repr} and
  use~\eqref{eq:Yfromy}.
\end{proof}

% The next two theorems are also as expected.

\begin{thm}
  The Yoneda embedding $y\maps B\to \cPk B$ preserves all limits.
\end{thm}
\begin{proof}
  Let $J\maps A\hto K$ be any weight, $g\maps A\to B$ a functor, and
  $\ell\maps K\to B$ a $J$-weighted limit of $g$; thus $B(1,\ell)\iso
  B(1,g) \lhd J$.  We then have
  \begin{multline}
    \cPk B(1,y \ell)
    \;\iso\; Y_B \rhd Y_B(1,y\ell)
    \;\iso\; Y_B \rhd Y_B(1,y)(1,\ell)\\
    \;\iso\; Y_B \rhd B(1,\ell)
    \;\iso\; Y_B \rhd \Big(B(1,g) \lhd J\Big)
    \;\iso\; \Big(Y_B \rhd B(1,g)\Big) \lhd J\\
    \;\iso\; \Big(Y_B \rhd Y_B(1,yg)\Big) \lhd J
    \;\iso\; \cPk B(1,yg) \lhd J,
  \end{multline}
  so $y\ell$ is a $J$-weighted limit of $yg$.
\end{proof}

The following is a generalization of~\cite[Prop.~3.2]{dl:lim-smallfr}.

\begin{thm}
  For \ka-small $A$ and any $J:A\hto K$, a functor $f:A\to \cPk B$ has
  a $J$-weighted limit if and only if the profunctor $Y_B(1,f)\lhd J:
  A\hto B$ is \ka-small.
\end{thm}
\begin{proof}
  Note that $Y_B(1,f) \lhd J$ exists since $A$ is \ka-small.  Consider
  a functor $\ell\maps K\to \cPk B$; we will show that $\ell$ is a
  $J$-weighted limit of $f$ if and only if it is a classifying map for
  $Y_B(1,f)\lhd J$.  On the one hand, we have
  \begin{align}
    \cPk B(1,\ell)
    &\iso Y_B \rhd Y_B(1,\ell)\label{eq:pshcplt1}
  \end{align}
  while on the other we have
  \begin{align}
    \cPk B(1,f) \lhd J
    &\iso \Big(Y_B \rhd Y_B(1,f)\Big) \lhd J\\
    &\iso Y_B \rhd \Big(Y_B(1,f) \lhd J\Big).\label{eq:pshcplt2}
  \end{align}
  Now if $\ell$ classifies $Y_B(1,f)\lhd J$, then by definition
  $Y_B(1,\ell) \cong Y_B(1,f) \lhd J$, and thus~\eqref{eq:pshcplt1}
  and~\eqref{eq:pshcplt2} are isomorphic; hence $\ell$ is a
  $J$-weighted limit of $f$.  Conversely, if~\eqref{eq:pshcplt1}
  and~\eqref{eq:pshcplt2} are isomorphic, we have
  \begin{multline}
    Y_B(1,\ell)
    \cong B \rhd Y_B(1,\ell)
    \cong Y_B(1,y) \rhd Y_B(1,\ell)\\
    \cong \big(Y_B \rhd Y_B(1,\ell)\big)(y,1)
    \cong \Big(Y_B \rhd \big(Y_B(1,f) \lhd J\big)\Big)(y,1)\\
    \cong \big(\cPk B(1,f) \lhd J\big)(y,1)
    \cong \cPk B(y,f) \lhd J
    \cong Y_B(1,f) \lhd J
  \end{multline}
  so that $\ell$ classifies $Y_B(1,f) \lhd J$.
\end{proof}

\begin{cor}
  If $B$ is \ka-small, then $\cPk B$ is \ka-cocomplete.\endproof
\end{cor}

Finally, we prove the familiar theorem that presheaf objects are free
cocompletions.  For \sV-categories $B$ and $C$, we write
$\VCAT_{\ka\text-\mathrm{colim}}(B,C)$ for the full subcategory of
$\VCAT(B,C)$ determined by the \V-functors which preserve \ka-small
colimits.

\begin{thm}\label{thm:pshf-free-cocompletion}
  If $A$ is any \V-category and \sB is a \ka-cocomplete \V-category,
  then composition with $y_A\maps A\to \cPk A$ defines an equivalence
  of categories
  \begin{equation}\label{eq:free-cocompletion-eqv}
    \VCAT_{\ka\text-\mathrm{colim}}(\cPk A,\sB) \too \VCAT(A,\sB).
  \end{equation}
  In other words, $\cPk A$ is the free \ka-small cocompletion
  of $A$.
\end{thm}
\begin{proof}
  Since \sB is \ka-cocomplete, by \autoref{thm:colim-smwgt} any
  functor $f:A\to \sB$ admits a $Y_A$-weighted colimit, which is to
  say a left Kan extension along $y_A:A\to \cPk A$.
  Thus we have a functor
  \[\lan_y\maps \VCAT(A,\sB) \to \VCAT(\cPk A,\sB).\]
  We claim that for any $f\maps A\to \sB$, the functor $\lan_{y_A}
  f:\cPk A \to \sB$ preserves \ka-small colimits.  Suppose that
  $J\maps K\hto C$ is a weight, where $C$ is \ka-small, and $d\maps
  C\to \cPk A$ is a functor.  Let $\ell = \colim^J d\maps K\to \cPk
  A$.  Using again the construction of colimits in $\cPk A$ in
  \autoref{thm:pshf-cplt}, we compute
  \begin{align*}
    \sB\big({(\lan_{y_A} f)\ell},1\big)
    &\iso \sB({\lan_{y_A} f},1) \odot \cPk A(\ell,1) \\
    &\iso \cPk A(1,\ell) \rhd \sB({\lan_{y_A} f},1)\\
    &\iso \Big(J\odot \cPk A(1,d)\Big) \rhd \sB(\lan_{y_A} f,1)\\
    &\iso J\rhd \Big(\cPk A(1,d)\rhd \sB(\lan_{y_A} f,1)\Big)\\
    &\iso J\rhd \sB((\lan_{y_A} f)d,1)
  \end{align*}
  as desired.  Therefore, we have an induced functor
  \[\lan_{y_A}\maps \VCAT(A,\sB) \to \VCAT_{\ka\text-\mathrm{colim}}(\cPk A,\sB),
  \]
  which we claim is an inverse equivalence
  to~(\ref{eq:free-cocompletion-eqv}).

  On the one hand, since $y_A$ is fully faithful by
  \autoref{thm:yoneda-ff}, by \autoref{thm:ff-extn-are-honest} we have
  $(\lan_{y_A} f) y_A \iso f$.  On the other hand, by
  \autoref{thm:yoneda-dense}, $1_{\cPk A}$ is the left Kan extension
  of $y$ along itself.  Therefore, if $g\maps \cPk A\to \sB$ preserves
  \ka-small colimits, hence (by \autoref{thm:colim-smwgt}) also
  colimits with \ka-small weights, it must preserve left Kan
  extensions along $y_A$.  Thus we must have $g\iso \lan_{y_A} (g
  y_A)$.
\end{proof}

\begin{rmk}
  In the case $\V=\sPsh(\S,\bV)$, we can repeat all the above
  arguments but adding ``local smallness'' conditions (in the sense of
  \autoref{thm:psh-comphom}) to all \ka-small categories and
  profunctors.  This yields a free cocompletion of any pseudofunctor
  $\S\op\to\CAT{\bV}$ under fiberwise \bV-enriched colimits and
  \S-indexed \bV-coproducts.
\end{rmk}

As usual, we can also find more general free cocompletions inside
$\cPk A$ by closing up the image of $y_A$ under various types of
colimits.  This is the subject of~\cite{bunge:bddcplt}.

\section{Monoidal \sV-categories and iterated enrichment}
\label{sec:monoidal}

In previous sections we have assumed for simplicity that \V was
symmetric, but so far everything could also be done in the
non-symmetric case, simply by keeping more careful track of right
versus left homs.  Now, however, we consider monoidal structures on
\V-categories, for which we do need a symmetry on \V (or at least a
braiding, although we will not consider that case).

Thus, let \S have finite products and \V be, for now, an \S-indexed
symmetric monoidal category.  This enables us to define opposites and
tensor products of \V-categories.

\begin{defn}
  For a \V-category \sA, its \textbf{opposite} $\sA\op$ has the same
  objects and extents as \sA, with $\U{\sA\op}(x,y) = \ks^*\usA(y,x)$,
  where $\ks:\ex\times \ey\toiso \ey\times\ex$ is the twist
  isomorphism.  Its identities are obvious, and its composition
  morphism is
  \begin{equation}
    \ks^*\usA(z,y) \otimes_{\ey} \ks^*\usA(y,x) \toiso
    \ks^*\big(\usA(y,x) \otimes_{\ey} \usA(z,y)) \xto{\mathrm{comp}}
    \ks^*\usA(z,x)
  \end{equation}
  using the symmetry of \V and the composition morphism of \sA.
\end{defn}

If \sA is a \V-fibration, then so is $\sA\op$, and we have $(\sA\op)^X
= (\sA^X)\op$ (the latter opposite being as a $\V^X$-enriched
category).  On the other hand, for tensor products this may not be the
case.  In fact, we have two seemingly different tensor products for
\V-categories.

\begin{defn}
  For \V-categories \sA and \sB, their \textbf{tensor product}
  $\sA\otimes\sB$ has as objects pairs $(a,b)$ where $a$ is an object
  of \sA and $b$ is an object of \sB, with $\e(a,b) = \ea \times \eb$,
  and hom-objects 
  \begin{equation}
    \U{\sA\otimes\sB}\big((a,b),(a',b')\big) =
    \ks^*\big(\usA(a,a') \otimes \usB(b,b')\big)
  \end{equation}
  where $\ks$ is the isomorphism
  \[ \ea' \times \eb' \times \ea\times\eb
  \toiso \ea' \times \ea \times \eb' \times \eb.
  \]
  Its identities are induced by those of \sA and \sB in an obvious
  way, while its composition morphism is
  \begin{multline}
    \ks^*\big(\usA(a',a'') \otimes \usB(b',b'')\big)
    \otimes_{\ea' \times \eb'}
    \ks^*\big(\usA(a,a') \otimes \usB(b,b')\big)\\
    \toiso
    \ks^*\Big(
    \big(\usA(a',a'') \otimes_{\ea'} \usA(a,a')\big)
    \otimes
    \big(\usB(b',b'') \otimes_{\eb'} \usB(b,b')\big)
    \Big)\\
    \xto{\mathrm{comp}\otimes\mathrm{comp}}
    \ks^*\big(\usA(a,a'') \otimes \usB(b,b'')\big).
  \end{multline}
\end{defn}

\begin{defn}
  For indexed \V-categories \sA and \sB, their \textbf{indexed tensor
    product} $\sA\otimes_\S \sB$ is defined by
  \begin{equation}
    (\sA\otimes_\S\sB)^X = \sA^X \otimes_X \sB^X,
  \end{equation}
  the right-hand side being the tensor product of $\V^X$-enriched
  categories.  Thus, the objects of $(\sA\otimes_\S\sB)^X$ are pairs
  $(a,b)$ with $a\in\sA^X$ and $b\in\sB^X$.
\end{defn}

The tensor product $\otimes$ makes \VCAT into a symmetric monoidal
2-category, with unit object $\delta 1$, while $\otimes_\S$ makes
\iVCAT into a symmetric monoidal 2-category, with a unit object \cI
having $\cI^X$ the unit $\V^X$-enriched category for all $X$.  The two
tensor products are different, but as in \SS\ref{sec:v-fibrations},
both are ``loose enough'' that they agree up to equivalence.

\begin{thm}\label{thm:vfib-vcat-ssmon}
  The biequivalence $\Theta : \iVCAT \simeq \VCAT : \Gamma$ is a
  symmetric monoidal biequivalence.
\end{thm}
\begin{proof}
  First of all, evidently $\cI \cong \Gamma(\delta 1)$, hence $\delta
  1 \simeq \Theta\cI$.  Now let \sA and \sB be indexed \V-categories;
  we define an equivalence $\Theta\sA\otimes \Theta\sB \simeq
  \Theta(\sA\otimes_\S \sB)$.  In one direction we have an indexed
  functor:
  \[F:\Theta\sA\otimes \Theta\sB \too \Theta(\sA\otimes_\S \sB)\]
  defined on objects by $F(a,b) = (\pi_{\eb}^* a,\pi_{\ea}^* b)$.  In
  the other direction we have a non-indexed functor
  \[G:\Theta(\sA\otimes_\S \sB) \too \Theta\sA\otimes \Theta\sB \]
  defined at an object $(a,b)\in (\sA\otimes_\S \sB)^X$ by $G(a,b) =
  (a,b)$ with $G_{(a,b)} = \Delta_X$.  We leave it to the reader to
  verify that these are inverse equivalences and support the
  additional coherent structure of a symmetric monoidal biequivalence.
\end{proof}

Now recall that a \emph{monoidal object}, or \emph{pseudomonoid}, in a
monoidal 2-category consists of an object $W$ with a multiplication
$m\maps W\ten W\to W$ and a unit $e\maps I\to W$ together with the
usual coherent associativity and unit isomorphisms.  We thus obtain
notions of \textbf{monoidal \V-category} and \textbf{indexed monoidal
  \V-category} by using the tensor products $\otimes$ and $\otimes_\S$
respectively.

\autoref{thm:vfib-vcat-ssmon} tells us that if \sW is a \V-fibration,
then up to equivalence there is no difference between these two
notions.  However, on the surface they look quite different.

On the one hand, a monoidal \V-category has an \emph{external product}
$\boxtimes:\W\otimes\W\to\W$.  If \W is a \V-fibration, then we may
assume $\boxtimes$ to be indexed, so that $\e(a\boxtimes b) = \e
a\times \e b$, just like for the external produt of \V itself.

On the other hand, an indexed monoidal \V-category is equipped with a
\emph{fiberwise product} $\W\otimes_\S \W\to\W$, which takes two
objects $a,b\in\W^X$ to $a\boxtimes_X b \in\W^X$, just as for the
fiberwise product of \V.  Indeed, an indexed monoidal \V-category is
easily seen to be just an indexed \V-category for which each fiber
$\W^X$ is a monoidal $\sV^X$-enriched category and the transition
functors $f^*$ and their coherence isomorphisms are strong monoidal.

It should not now be surprising that the equivalence between these two
types of monoidal structure on a given \V-fibration \W exactly
parallels the equivalence between external and fiberwise products for
\V described in \SS\ref{sec:indexed-moncats}.  This is easy to see
concretely by tracing through the equivalence constructed in
\autoref{thm:vfib-vcat-ssmon}.  (If \W moreover admits indexed
\V-coproducts as in \SS\ref{sec:indexed-limits}, then we can define a
canceling product for it as well.)

\begin{eg}\label{eg:psh-mon}
  In fact, if we take $\V=\sPsh(\S,\nSet)$ so that indexed
  \V-categories are precisely \S-indexed categories, then this
  equivalence between the two types of monoidal structure on a
  \V-fibration reduces more or less exactly to our development in
  \SS\ref{sec:indexed-moncats}.
\end{eg}

Of course, it is easy to define \emph{symmetric} monoidal
\V-categories, but \emph{closed} ones require the machinery of
profunctors yet again, as pioneered
by~\cite{ds:monbi-hopfagbd,dms:antipodes,street:frob-psmon}.  First we
note the following.

\begin{lem}
  For \V-categories $A,B,C$, there is an equivalence of categories
  \begin{equation}
    \ch: \VPROF(A, B \otimes C) \toiso \VPROF(B\op\otimes A, C)
    \tag*{\endproofbox}
  \end{equation}
\end{lem}
% \begin{proof}
%   Straightforward.
% \end{proof}

We will also need to know that $\ch$ respects composites and homs.

\begin{lem}\label{thm:cptcl-comp}
  For \V-profunctors $H:A\hto B\otimes C$, $K:A'\hto A$,
  $L:B\hto B'$, and $M:C\hto C'$ we have
  \begin{equation}
    \ch\big(K \odot H \odot (L\otimes M)\big) \cong
    (L\op\otimes K) \odot \ch H \odot M
  \end{equation}
  where $L\op$ is the evident induced profunctor $(B')\op\hto
  B\op$.\endproof
\end{lem}

\begin{lem}\label{thm:cptcl-hom}
  For \V-profunctors $H:A\hto B\otimes C$, $K:A\hto A'$,
  $L:B'\hto B$, and $M:C'\hto C$ we have
  \begin{equation}
    \ch\big((L\otimes M) \rhd H \lhd K\big) \cong
    M \rhd \ch H \lhd (L\op\otimes K).
    \tag*{\endproofbox}
  \end{equation}
\end{lem}

Now suppose \W is a symmetric monoidal \V-category (as usual, symmetry
of \W is not required, but keeping track of right versus left homs is
tedious).  If $m:\W\otimes \W\to \W$ is its tensor product \V-functor,
we have an induced profunctor $\W(m,1):\W\hto \W\otimes \W$, and hence
a profunctor $\ch(\W(m,1)):\W\op\otimes \W\hto \W$.

\begin{defn}
  A symmetric monoidal \V-category \W is \textbf{closed} if there is a
  \V-functor $h:\W\op\otimes \W \to \W$ and an isomorphism
  \begin{equation}
    \ch({\W(m,1)}) \cong \W(1,h).\label{eq:closedvcat}
  \end{equation}
\end{defn}

If \W is a \V-fibration, we may take $h$ to be an \emph{indexed}
\V-functor $\W\op\otimes_\S \W \to \W$, and by
\autoref{rmk:ivprof-eqv} we may consider~\eqref{eq:closedvcat} an
isomorphism of \emph{indexed} \V-profunctors as in
\autoref{def:ivprof}.  If we unravel this explicitly, what it says is
merely that the symmetric monoidal $\V^X$-enriched category $\W^X$ is
closed, for each $X$, and that the transition functors $(f^*)_\bullet
\W^X \to \W^Y$ are closed monoidal (this is encoded in the
\V-functoriality of $h$).  We denote these \emph{fiberwise homs}
in $\W^X$ by $\uWW^X(-,-)$ (by contrast with
the \V-valued fiberwise hom $\usW^X(x,y)\in\V^{X}$).

\begin{eg}\label{eg:psh-closed}
  In particular, for $\V=\sPsh(\S,\nSet)$ as in \autoref{eg:psh-mon},
  then this notion of closedness for indexed monoidal \V-categories
  reduces essentially to \autoref{thm:three-homs}\ref{item:closed-1}.
\end{eg}

On the other hand, we might remain in the world of non-indexed
\V-profunctors, but still assume that \W is a \V-fibration, so that
$h$ might as well be indexed.  In this case, $h$ has the right type to
be an external-hom such as in
\autoref{thm:three-homs}\ref{item:closed-3}.  Namely, for objects $x$
and $y$ with extents $\ex$ and \ey, $h(x,y)$ is an object of
$\W^{\ex\times\ey}$, which we denote $\uWW(x,y)$ (by contrast with
the \V-valued external hom $\usW(x,y)\in\V^{\ex\times\ey}$).  The
universal property of these homs is expressed by an isomorphism
\begin{equation}\label{eq:extexthom}
  \usW(x\boxtimes y, z)\cong \usW\big(x,\uWW(y,z)\big)
\end{equation}
natural in $x,y,z$.  And as usual, given objects $\uWW(y,z)$ with
isomorphisms~\eqref{eq:extexthom} that are natural in $x$, then we can
construct a unique (indexed) \V-functor $\W\op\otimes\W \to \W$
making~\eqref{eq:extexthom} natural in $y$ and $z$ as well.  In
particular, since this functor preserves restrictions like any indexed
\V-functor, we have
% for $f\maps Y\to\ey$ and $g\maps Z\to\ez$ we have
% \begin{align*}
%   \usW(x, (f\times g)^*\uWW(y,z))
%   &\iso (f\times g \times 1)^* \usW(x,\uWW(y,z))\\
%   &\iso (f\times g \times 1)^* \usW(x\boxtimes y,z)\\
%   &\iso \usW((1\times f)^*(x\boxtimes y), g^*z)\\
%   &\iso \usW(x\boxtimes f^*y, g^*z)\\
%   &\iso \usW(x, \uWW(f^*y,g^*z)).
% \end{align*}
% and thus
\begin{equation}
  (g\times f)^*\uWW(y,z) \iso \uWW(f^*y,g^*z).\label{eq:extext-compat}
\end{equation}
which is a version of the compatibility condition from
\autoref{thm:three-homs}\ref{item:closed-3}.

However, the universal property~\eqref{eq:extexthom} itself looks
surprisingly different from that in
\autoref{thm:three-homs}\ref{item:closed-3}.  But if \sW has indexed
\V-coproducts preserved by $\boxtimes$, then~\eqref{eq:extexthom} is
equivalent to a universal property looking more like
\ref{thm:three-homs}\ref{item:closed-3}.  Namely,
given~\eqref{eq:extexthom}, for $x\in\W^{X\times Y}$, $y\in\W^Y$, and
$z\in\W^X$, we have
\begin{align}
  \usW^X(x\boxtimes_{[Y]} y, z)
  &\cong \Delta_X^* \usW(\pi_{Y!} \Delta_Y^* (x\boxtimes y), z)\\
  &\cong \Delta_X^* \pi_{Y*} \Delta_Y^* \usW(x\boxtimes y, z)\\
  &\cong \Delta_X^* \pi_{Y*} \Delta_Y^* \usW(x, \uWW(y, z))\\
  &\cong \pi_{Y*}\Delta_{X\times Y}^* \usW(x, \uWW(y, z))\\
  &\cong \pi_{Y*} \usW^{X\times Y}(x, \uWW(y, z))\label{eq:extext-other}
\end{align}
which is a version of~\eqref{eq:closed-3-adjn}, enhanced for
compatibility with \V, analogously to~\eqref{eq:enr-indexed-coprod}.
Conversely, if we assume~\eqref{eq:extext-other} and
also~\eqref{eq:extext-compat}, then for $x\in\W^X$, $y\in\W^Y$, and
$z\in\W^Z$ we have (omitting the symbol $\times$ in most places, for
conciseness):
\begin{align}
  \usW(x\boxtimes y,z)
  &\cong \usW^{X Y Z}
  \big(\pi_Z^*(x \boxtimes y), \pi_{X Y}^*z\big)\\
  &\cong \usW^{X Y Z}
  \big(\pi_{XY!}\Delta_{XY!}\Delta_{XY}^*
  (\pi_{YZ}^*x \boxtimes \pi_X^* y),
  \pi_{X Y}^*z\big)\\
  &\cong \usW^{X Y Z}
  \big(\pi_{XY!}(1\times \Delta_{XY})^*(\Delta_{XY}\times 1)_!
  (\pi_{YZ}^*x \boxtimes \pi_X^* y),
  \pi_{X Y}^*z\big)\\
  &\cong \usW^{X Y Z}
  \big(\Delta_{XY!}\pi_{YZ}^*x \boxtimes_{[XY]} \pi_X^* y,
  \pi_{X Y}^*z\big)\\
  &\cong \pi_{XY*} \usW^{X Y X Y Z}
  \big(\Delta_{XY!}\pi_{YZ}^*x,
  \uWW( \pi_X^* y, \pi_{X Y}^*z)\big)\\
  &\cong \pi_{XY*} \Delta_{XY*} \usW^{X Y Z}
  \big(\pi_{YZ}^*x,
  \Delta_{XY}^* (\pi_X  \pi_{X Y})^* \uWW(y,z)\big)\\
  &\cong \usW^{X Y Z}
  \big(\pi_{Y Z}^* x, \pi_X^*\uWW(y,z)\big)\\
  &\cong \usW(x,\uWW(y,z)).
\end{align}
Of course, the equivalence between the two kinds of \V-profunctors
implies that the fiberwise homs $\uWW^X(-,-)$ and external homs
$\uWW(-,-)$ for a symmetric monoidal \V-category are interderivable,
with formulas just like those in \autoref{thm:three-homs}.  And if \W
has indexed \V-products, we can also define a \emph{canceling hom}
$\uWW^{[X]}(-,-)$ just as in \autoref{thm:three-homs}.

\begin{eg}
  As promised in \SS\ref{sec:indexed-moncats},
  interpreting~\eqref{eq:extexthom} for $\sPsh(\S,\nSet)$-categories
  yields a characterization of external-homs for ordinary indexed
  monoidal categories that doesn't require indexed coproducts.  Recall
  that the $\sPsh(\S,\nSet)$-valued external-hom of an \S-indexed
  category \V is given at $x\in\V^X$ and $y\in\V^Y$ by
  \begin{equation}
    \begin{array}{rcl}
      (\S/(X\times Y))\op &\too& \nSet\\
      (Z\xto{(f,g)} X\times Y) &\mapsto& \V^{Z}(f^*x,g^*y).
    \end{array}
  \end{equation}
  Thus,~\eqref{eq:extexthom} consists of isomorphisms
  \begin{equation}
    \V^U\big((f,g)^*(x\boxtimes y), h^*z\big) \cong
    \V^U\big(f^*x, (g,h)^* \usV(y,z)\big)
  \end{equation}
  for $x\in \V^X$, $y\in\V^Y$, $z\in\V^Z$, $f:U\to X$, $g:U\to Y$, and
  $h:U\to Z$, satisfying appropriate sorts of naturality.  In
  particular, with this characterization the compatibility condition
  $(g\times f)^*\usV(x,y) \cong \usV(f^*x,g^*y)$ is automatic.
\end{eg}

Finally, we have all the ingredients of the following.

\begin{defn}\label{def:vcosmos}
  If \V is an \S-indexed cosmos, then a \textbf{\V-cosmos} is a closed
  symmetric monoidal indexed \V-category which is \om-complete and
  \om-cocomplete (in the sense of \autoref{thm:constr-of-colim}).
\end{defn}

\begin{eg}
  Since indexed $\self(\S)$-categories are just ordinary indexed
  categories with the property of being ``locally small'', a
  $\self(\S)$-cosmos is just an ordinary \S-indexed cosmos with this
  property.
\end{eg}

\begin{eg}
  An indexed $\fam(\bV)$-category of the form $\fam(\bC)$ is a
  $\fam(\bV)$-cosmos just when \bC is a \bV-cosmos in a classical
  sense, namely a complete and cocomplete closed symmetric monoidal
  \bV-enriched category.
\end{eg}

\begin{eg}
  Since $\sPsh(\S,\bV)$ is not in general a cosmos,
  \autoref{def:vcosmos} is not quite right for it.  Instead, we may
  reasonably define an \textbf{\S-indexed \bV-enriched cosmos} to be a
  closed symmetric monoidal indexed $\sPsh(\S,\bV)$-category with
  ``locally small'' limits and colimits in the sense of
  \autoref{eg:psh-cplt}.

  For instance, if \S is locally cartesian closed, complete, and
  cocomplete, then the cosmos $\sAb(\S)$ is enriched over abelian
  groups in this sense.
\end{eg}

\begin{eg}
  If \S is locally cartesian closed, complete and cocomplete, then
  $\self(\S)$ is an \S-indexed \S-enriched cosmos, where in addition
  to being the base of the indexing, we regard \S as a classical
  cartesian monoidal category.  Similarly, the cosmoi \sK and $\sK_*$
  from Examples~\ref{eg:top-mf} and~\ref{eg:ret-mf} are indexed
  enriched cosmoi over a good category of topological spaces.  The
  interaction of these iterated enrichments on $\sK$-categories and
  $\sK_*$-categories, along with variations with an action by a fixed
  topological group (as in \autoref{eg:gactions-mf}) is discussed in
  detail in~\cite[Ch.~10]{maysig:pht}.
\end{eg}

Now, if \W is a symmetric monoidal indexed \V-category, then there is
a lax symmetric monoidal morphism $\ord\W\to\V$ over \S, defined on
the fiber over $X$ by $W\mapsto \usW^X(\I_X,W)$.  Moreover, the
following triangle commutes up to isomorphism:
\begin{equation}\label{eq:laxmoroverS}
  \vcenter{\xymatrix{\V \ar[dr] && \ord\W \ar[dl] \ar[ll]\\
      & \sPsh(\S,\nSet).}}
\end{equation}
Conversely, if $\ord\W$ is a symmetric monoidal \S-indexed category
equipped with a lax morphism satisfying~\eqref{eq:laxmoroverS}, and
moreover $\ord\W$ is \emph{closed}, then by applying the induced
change-of-cosmos functor to the \W-category \W, we obtain a closed
symmetric monoidal indexed \V-category structure on \W.  These two
constructions are inverses, so just as for classical monoidal
categories, we have an equivalence between
\begin{enumerate}
\item lax symmetric monoidal morphisms of fibrations $\sW\to\sV$ over
  \bS, where \sW\ is closed symmetric monoidal,
  satisfying~\eqref{eq:laxmoroverS}, and\label{item:lsmmf}
\item closed symmetric monoidal \sV-fibrations.\label{item:csmvf}
\end{enumerate}
Similarly, we can show that the lax morphism $\ord\W\to\V$ has a
strong monoidal left adjoint if and only if the \V-category \W has
fiberwise tensors, and that in this case $\ord\W$ is a cosmos if and
only if \W is a \V-cosmos.  This gives an alternative approach to many
of our examples from \SS\ref{sec:indexed-moncats}, but we will not
revisit them all here.

We end with a version of the Day convolution monoidal
structure~\cite{day:closed} for \V-categories.  Note that the
monoidal \V-category $A$ appearing below is \ka-small, hence probably
not a \V-fibration; thus its tensor product and unit morphisms may not
be indexed \V-functors.

\begin{thm}\label{thm:day-presheaf}
  Let \V be an indexed cosmos which is \ka-complete and
  \ka-cocomplete, and let $A$ be a \ka-small symmetric monoidal
  \sV-category.  Then $\cPk A$ is a closed symmetric monoidal
  \sV-category.
\end{thm}
\begin{proof}
  We define the product $\mhat:\cPk A\ten \cPk A \to \cPk A$ to be the
  classifying map of the profunctor
  \[ \cPk A \ten \cPk A \xhto{(Y_A \ten Y_A) \odot A(1,m)} A,
  \]
  where $m: A\ten A\to A$ is the tensor product of $A$.  The displayed
  composite exists since $A\ten A$ is \ka-small.  The unit $\delta
  1\to \cPk A$ is the classifying map of $A(1,i)$, where $i: \delta
  1\to A$ is the unit of $A$.  By full-faithfulness and
  pseudofunctoriality of representable profunctors, the coherence data
  for $A$ lift automatically to the corresponding profunctors, and
  thence to their classifying functors; thus $\cPk A$ is symmetric
  monoidal.  For closedness, we use \autoref{thm:pshf-ff} to compute
  \begin{align}
    \cPk A(\mhat,1)
    &\cong \big((Y_A \ten Y_A) \odot A(1,m)\big) \rhd Y_A\\
    &\cong (Y_A \ten Y_A) \rhd \big(A(1,m) \rhd Y_A\big)\\
    &\cong (Y_A \ten Y_A) \rhd Y_A(m,1).
  \end{align}
  Thus, by \autoref{thm:cptcl-hom} and \autoref{thm:pshf-ff} again, we
  have
  \begin{align}
    \ch(\cPk A(\mhat,1))
    &\cong Y_A \rhd \big(\ch(Y_A(m,1)) \lhd (\cPk A \otimes Y_A\op)\big)\\
    &\cong \cPk(1,\hhat)
  \end{align}
  where $\hhat:(\cPk A)\op\times\cPk A \to \cPk A$ is the classifying
  functor of $\ch(Y_A(m,1)) \lhd (\cPk A \otimes Y_A\op)$.  Thus,
  $\cPk A$ is closed.
\end{proof}

\begin{rmk}
  In fact, for \autoref{thm:day-presheaf} it suffices for $A$ to be a
  symmetric monoidal object in $\VPROF$, rather than $\VCAT$ --- that
  is, a \emph{promonoidal} \V-category.  As in the classical
  case~\cite{day:closed,day:closed-ii}, promonidal structures on $A$
  are actually equivalent to closed symmetric monoidal structures on
  $\cPk A$.
\end{rmk}

This allows us to produce easily two of the most important topological
examples.

\begin{eg}
  The cosmos $\sK_*$ of sectioned topological spaces from
  \autoref{eg:ret-mf} is enriched and fiberwise-tensored over the
  cosmos $\sK$ of topological spaces from \autoref{eg:top-mf}, so we
  have a monoidal adjunction $\sK \toot \sK_*$.  Let $\sI$ be the
  topologically enriched category of finite-dimensional inner product
  spaces and linear isometric isomorphisms, and regard it as a
  set-small \sK-category with all objects having extent $1$, as in
  \autoref{eg:pshvs-over-enriched}.

  Now change cosmos along the left adjoint $\sK \to \sK_*$ (which adds
  disjoint sections) to obtain a set-small $\sK_*$-category $\sI_+$.
  Then $\cP (\sI_+\op)$, with its Day convolution monoidal structure
  from \autoref{thm:day-presheaf}, is the $\sK_*$-cosmos of
  \emph{parametrized \sI-spaces} described
  in~\cite[\SS11.1]{maysig:pht}.

  Finally, one-point compactifications yield a sphere object $S\in \cP
  (\sI_+\op)^1$ which is a commutative monoid.  Thus, regarding $\cP
  (\sI_+\op)$ as an ordinary cosmos as above, and applying
  \autoref{eg:modules-mf}, we obtain the $\sK_*$-cosmos of
  parametrized orthogonal spectra from~\cite{maysig:pht}.
  (In~\cite{maysig:pht}, everything has an additional action by a
  fixed topological group; this can easily be added using
  \autoref{eg:gactions-mf}.)
\end{eg}

\begin{eg}
  Recall from \autoref{eg:actions} that we have a
  $\nGrp(\nTop)$-indexed monoidal category $\sAct(\nTop)$, where
  $\nGrp(\nTop)$ is the category of topological groups and
  $\sAct(\nTop)^G$ is the cartesian monoidal category of $G$-spaces.
  From \autoref{eg:ret-mf} we obtain a $\nGrp(\nTop)$-indexed monoidal
  category $\sAct(\nTop)_*$ of based spaces with group actions.

  Let $\cG\subseteq \nGrp(\nTop)$ be the full subcategory of finite
  groups, and $\topg$ the restriction of $\sAct(\nTop)_*$ to \cG.
  This is a fiberwise complete and cocomplete \cG-indexed cosmos.

  Let $\cI_\cG$ be the \topg-category with objects
  $(G\in\cG,n\in\lN, \rho:G\to O(n))$; that is, finite-dimensional
  representations of finite groups.  The extent of such a
  representation is of course $G$, while
  $\U{\cI_\cG}(\rho,\rho')\in(\topg)^{G\times G'}$ is the space of
  linear isometric isomorphisms $\lR^n \toiso \lR^{n'}$ (which is of
  course empty unless $n=n'$) with a disjoint basepoint added, with
  $(G\times G')$-action by conjugation.  Of course, this may be
  obtained by change of cosmos from an unbased version.

  Now $\cI_\cG$ is a set-small \topg-fibration, and it is moreover
  symmetric monoidal under the direct sum of representations.
  Therefore, by \autoref{thm:day-presheaf}, we have a \topg-cosmos
  $\cP(\cI_\cG\op)$ of ``$\cI_\cG$-spaces''.
  (In~\cite{bohmann:globalspectra}, only the objects of
  $\cP(\cI_\cG\op)$ of extent $1$ are called $\cI_\cG$-spaces; those
  of extent $G\in\cG$ have an additional $G$-action on each of their
  spaces.)

  Finally, one-point compactifications yield a canonical sphere object
  $S\in \cP(\cI_\cG\op)^1$ which is a commutative monoid, and so again
  from \autoref{eg:modules-mf} we obtain a \topg-cosmos whose objects
  of extent $1$ are the \emph{orthogonal \cG-spectra}
  of~\cite{bohmann:globalspectra}.
\end{eg}

% \bibliographystyle{alpha}
% \bibliography{all}

\end{document}